\documentclass[12pt]{amsart}
\usepackage{amsmath, amsthm, amssymb,cite}
\usepackage{fullpage}
\usepackage{color}
\usepackage{hyperref}

\newtheorem{theorem}{Theorem}
\newtheorem{lemma}{Lemma}
\newtheorem{proposition}[lemma]{Proposition}
\newtheorem{corollary}[lemma]{Corollary}
\newtheorem{definition}[lemma]{Definition}
\newtheorem{remark}[lemma]{Remark}

\numberwithin{lemma}{section}

\newcommand{\Ez}{E_0}

\newcommand{\Elind}{E^{(2)}_{lin}}

\newcommand{\Ent}{E^{n,(3)}}

\newcommand{\Ennfz}{E^{n}_{NF,0}}
\newcommand{\Ennf}{E^{n}_{NF}}
\newcommand{\Ennfhigh}{E^{n}_{NF,high}}

\newcommand{\M}{\mathfrak M}

\newcommand{\B}{\mathfrak B}
\newcommand{\C}{\mathfrak C}

\numberwithin{equation}{section}

\newcommand{\Z}{{\mathbb Z}}
\newcommand{\tW}{{\tilde W}}
\newcommand{\tQ}{{\tilde Q}}
\newcommand{\ttW}{\tilde{\tilde W}}
\newcommand{\ttQ}{\tilde {\tilde Q}}

\renewcommand{\H}{{\mathcal H }}
\renewcommand{\AA}{\mathbf A}
\newcommand{\WH}{{\mathcal{WH} }}

\renewcommand{\SS}{\mathbf S}

\newcommand{\tG}{\tilde{G}}
\newcommand{\ttG}{\tilde{\tilde{G}}}
\newcommand{\ttK}{\tilde{\tilde K}}
\newcommand{\tK}{\tilde{K}}
\newcommand{\ttw}{\tilde{\tilde w}}
\newcommand{\ttq}{\tilde{\tilde q}}

\newcommand{\sgn}{\mathop{\mathrm{sgn}}}

\newcommand{\dH}{{\dot{\mathcal H} }}
\newcommand{\at}{\left\langle \frac{t}{\alpha} \right\rangle}

\newcommand{\W}{{\mathbf W}}

\newcommand{\err}{\text{\bf err}}

\newcommand{\uu}{{\mathbf u}}

\newcommand{\qq}{{\mathbf q}}
\newcommand{\ww}{{\mathbf w}}
\renewcommand{\ggg}{{\mathbf k}}

\begin{document}

\title{The lifespan of small data solutions in two dimensional capillary water
  waves }

\author{Mihaela Ifrim}
\address{Department of Mathematics, University of California at Berkeley}
\thanks{The first author was supported by the Simons Foundation}
\email{ifrim@math.berkeley.edu}

\author{ Daniel Tataru}
\address{Department of Mathematics, University of California at Berkeley}
 \thanks{The second author was partially supported by the NSF grant DMS-1266182
as well as by the Simons Foundation}
\email{tataru@math.berkeley.edu}

\begin{abstract}
  This article is concerned with the incompressible, irrotational
  infinite depth water wave equation in two space dimensions, without
  gravity but with surface tension.  We consider this problem
  expressed in position-velocity potential holomorphic coordinates,
  and prove that small data solutions have at least cubic lifespan
while small localized data leads to global solutions.  
\end{abstract}

\maketitle

\section{Introduction}

We consider the incompressible,  infinite depth water wave
equation in two space dimensions, without gravity but with surface
tension. This is governed by the incompressible Euler's equations
with boundary conditions on the water surface. 

We first describe the equations. We denote the water domain at time $t$ by $\Omega(t)$, and the water
surface at time $t$ by $\Gamma(t)$. We think of $\Gamma(t)$ as being
either asymptotically flat at infinity or periodic. The fluid velocity
is denoted by $u$ and the pressure is $p$. Then $u$ solves the 
Euler equations inside $\Omega(t)$, 
\begin{equation}
\left\{
\begin{aligned}
& u_t + u \cdot \nabla u = \nabla p 
\\
& \text{div } u = 0
\\
& u(0,x) = u_0(x),
\end{aligned}
\right.
\end{equation}
while on the boundary we have the dynamic boundary condition
\begin{equation}
 p = -2 \sigma {\bf H}  \ \ \text{ on } \Gamma (t ),
\end{equation}
and the kinematic boundary condition 
\begin{equation}
\partial_t+ u \cdot \nabla \text{ is tangent to } \bigcup \Gamma (t).
\end{equation}
Here ${\bf H}$ is the mean curvature of the boundary and $\sigma$
represents the surface tension.

Under the additional assumption that the flow is irrotational, we can
write $u$ in terms of a velocity potential $\phi$ as $u = \nabla
\phi$, where $\phi$ is harmonic with appropriate decay at infinity.
Thus $\phi$ is determined by its trace on the free boundary
$\Gamma (t)$. Denote by $\eta$ the height of the water surface as a
function of the horizontal coordinate. Following Zakharov's formulation, we introduce $\psi =\psi (t,x)\in \mathbf{R}$ to be the trace of the velocity potential $\phi$ on the boundary,
$\psi (t,x)=\phi (t,x,\eta(t,x))$. Then the fluid dynamics can be expressed in terms of a one-dimensional evolution of  the pairs of variables
$(\eta,\psi)$, namely 
\begin{equation}
\left\{
\begin{aligned}
& \partial_t \eta - G(\eta) \psi = 0 \\
& \partial_t \psi -{\bf H}(\eta)  +\frac12 |\nabla \psi|^2 - \frac12 
\frac{(\nabla \eta \cdot \nabla \psi +G(\eta) \psi)^2}{1 +|\nabla \eta|^2} = 0 .
\end{aligned}
\right.
\end{equation}
This is the Eulerian formulation of the capillary  water wave equations.

\subsection{Water waves in holomorphic coordinates.}
One main difficulty in the study of the above equation is the presence
of the Dirichlet to Neumann operator $\mathcal D$, which depends on
the free boundary.  This is one of the motivations for our choice to
use a different framework to study these equations, namely the
holomorphic (conformal) coordinates. To the best of our knowledge, these were
first used by Ovsjannikov~\cite{MR0347218} in the study of the dynamical 
problem, though they had been used much earlier for the study of static configurations
(solitons). Later, these coordinates were further
developed by Wu~\cite{MR1471885} and
Dyachenko-Kuznetsov-Spector-Zakharov~\cite{DYACHENKO199673} in the
context of gravity waves. In \cite{MR1471885} these coordinates are
used in combination with the Lagrangian coordinates in order to
develop the well-posedness theory for gravity waves.  More recently,
the holomorphic (conformal) coordinates were heavily exploited by the
authors in earlier work \cite{HIT}, \cite{IT} in the context of two
dimensional gravity water waves. In particular, \cite{HIT} is where a
complete well-posedness theory for gravity waves is developed in
holomorphic coordinates.

One can view the choice of coordinates (e.g. Lagrangian, Eulerian, holomorphic)
as gauge fixing for the gauge group of changes of coordinates; then it becomes
natural to study problems in the most favorable gauge.

The holomorphic coordinates are defined via a conformal map
$\mathcal{F}: \mathbb{H} \to \Omega(t)$, where $\mathbb{H}$ is the
lower half plane, $\mathbb{H}:=\left\lbrace \alpha -i\beta\ : \ \beta
  > 0 \right\rbrace $. Here $\mathcal{F}$ maps the real line into the
free boundary $\Gamma (t)$. Thus the real coordinate $\alpha$
parametrizes the free boundary, which is denoted by $Z(t,\alpha)$. In
order to uniquely determine the parametrization we impose a boundary
condition at infinity,
\[
\lim_{\alpha \to \infty} Z(t,\alpha) - \alpha = 0.
\]
In the periodic case we simply ask for $Z(t,\alpha) - \alpha $ to be periodic. We call these the holomorphic coordinates.

In order to describe the velocity potential $\phi$ we use the
 function $Q= \phi+i\psi$ where $\psi$ represents the harmonic
 conjugate of $\phi$. Both functions $Z-\alpha$ and $Q$ admit bounded
 holomorphic extensions into the lower half plane, which implies that
 their Fourier transforms are supported in $(-\infty,0]$.  We call
 such functions holomorphic functions. They can be described by the relation $Pf = f$,
where  $P$ represents the projector operator  to negative frequencies. The water wave equation will  define a flow in the class of holomorphic functions
for the pair of variables $(W,Q)$ where $W = Z-\alpha$.

A full derivation of the holomorphic form of gravity water waves was
included in an appendix to the paper \cite{HIT}. The minor
modifications which are necessary for the periodic case are also
described there. The equations with surface tension accounted are
quite similar, and are derived in the same fashion:
\begin{equation}
\label{ww2d}
\left\{
\begin{aligned}
& W_t + F (1+W_\alpha) = 0\\
& Q_t + F Q_\alpha  + P\left[ \frac{|Q_\alpha|^2}{J}\right] -\sigma P\left[ \frac{-i}{2+W_{\alpha}+\bar{W}_{\alpha}}\frac{d}{d\alpha}\left( \frac{W_{\alpha}-\bar{W}_{\alpha}}{\vert 1+W_{\alpha}\vert }\right)\right]  = 0, \\
\end{aligned} 
\right.
\end{equation}
where
\begin{equation}
\label{rww2d}
 F = P\left[\frac{Q_\alpha - \bar Q_\alpha}{J}\right], \qquad J = |1+W_\alpha|^2.
 \end{equation}
Simplifying we obtain the fully nonlinear system
\begin{equation}\label{ngst}
\left\{
\begin{aligned}
& W_t + F (1+W_\alpha) = 0 \\
& Q_t + F Q_\alpha  + P\left[ \frac{|Q_\alpha|^2}{J}\right] +i\sigma P\left[ \frac{W_{\alpha \alpha}}{J^{1/2}(1+W_{\alpha})}-\frac{\bar{W}_{\alpha \alpha}}{J^{1/2}(1+\bar{W}_{\alpha})}\right]  = 0.\\
\end{aligned} 
\right.
\end{equation}
The problem \eqref{ngst} is invariant with respect to translations, and also 
with respect to the scaling
\begin{equation}\label{scaling}
(W(t,\alpha), Q(t,\alpha)) \to (\lambda^{-1} W(\lambda^\frac32 t,\lambda \alpha), 
\lambda^{-\frac12} Q(\lambda^\frac32 t,\lambda \alpha)).
\end{equation}
The Hamiltonian for this system is 
\begin{equation*}
\mathcal{H}(W, Q)=\int \Im (Q\bar{Q}_{\alpha})
+\frac{1}{4}\sigma (J^{\frac{1}{2}} - 1 -\Re W_\alpha) \, d\alpha.
\end{equation*}
Since $W$ and $Q$ only appear in differentiated form in the equations above,
differentiating with respect to $\alpha$ yields a self contained 
system for $(W_\alpha,Q_\alpha)$, namely
\begin{equation*}
\left\{
\begin{aligned}
  &W_{\alpha t} + b W_{\alpha \alpha} +
    \frac{1}{1+\bar W_\alpha}\left(Q_{\alpha\alpha} - \frac{Q_\alpha}{1+W_\alpha}
      W_{\alpha \alpha}\right) = - (1+W_\alpha) \bar F_\alpha
- \left[\frac{\bar Q_\alpha}{1+\bar W_\alpha}\right]_\alpha 
 \\
  &Q_{t\alpha} + b Q_{\alpha\alpha}  +
    \frac{1}{1+\bar W_\alpha} \frac{Q_\alpha}{1+W_\alpha}\!\!\left(\!\!Q_{\alpha\alpha} - \frac{Q_\alpha}{1+W_\alpha}  W_{\alpha \alpha}\!\! \right)
        +i\sigma P\left[ \frac{W_{\alpha \alpha}}{J^{1/2}(1+W_{\alpha})}\right]  _{\alpha}=
\\ & \qquad \qquad \qquad  \qquad \qquad \qquad \quad \  i\sigma P \left[\frac{\bar{W}_{\alpha \alpha}}{J^{1/2}(1+\bar{W}_{\alpha})}\right]_{\alpha}  - Q_\alpha \bar F_\alpha+ \bar P\left[ \frac{|Q_\alpha|^2}{J}\right]_\alpha ,
\end{aligned}
\right.
\end{equation*}
where the advection velocity $b$ is given by 
\begin{equation}\label{b-def}
b := 2 \Re P\left[ \frac{Q_\alpha}{J}\right].
\end{equation}
The terms on the right are mostly antiholomorphic and can be viewed as
lower order when projected on the holomorphic functions. Examining the
expression on the left one easily identifies the leading surface
tension term in the second equation, added to a first order system.
The first order system is degenerate, and has a double speed $b$. Then
it is natural to diagonalize it. This is done exactly as in the case
of gravity waves using the operator
\begin{equation}
\AA(w,q) := (w,q - Rw), \qquad R := \frac{Q_\alpha}{1+W_\alpha}.
\label{defR}
\end{equation}
The factor $R$ above has an intrinsic meaning, namely it is the
complex velocity restricted on the water surface. We also remark that
\[
\AA(W_\alpha,Q_\alpha) = (\W,R), \qquad \W : = W_\alpha.
\]
Thus, the pair $(\W,R)$ diagonalizes the differentiated
system. Indeed, a direct computation yields the self-contained system
\begin{equation} \label{diff-ngst}
\left\{
\begin{aligned}
 & \W_{ t} + b \W_{ \alpha} + \frac{(1+\W) R_\alpha}{1+\bar \W}   =  (1+\W)M
\\
& R_t + bR_\alpha +\frac{ i a}{1+\W} + \frac{ i\sigma}{1+\W}P\left[ \frac{\W_{ \alpha}}{J^{1/2}(1+\W)}\right]_{\alpha}  =  \frac{ i\sigma}{1+\W} P \left[\frac{\bar{\W}_{ \alpha}}{J^{1/2}(1+\bar{\W})}\right]_{\alpha},
\end{aligned}
\right.
\end{equation}
where the real {\em frequency-shift} $a$ is given by
\begin{equation}
a := i\left(\bar P \left[\bar{R} R_\alpha\right]- P\left[R\bar{R}_\alpha\right]\right),
\label{defa}
\end{equation}
and  the auxiliary function $M$ is given by
\begin{equation}\label{M-def}
M :=  \frac{R_\alpha}{1+\bar \W}  + \frac{\bar R_\alpha}{1+ \W} -  b_\alpha =
\bar P [\bar R Y_\alpha- R_\alpha \bar Y]  + P[R \bar Y_\alpha - \bar R_\alpha Y].
\end{equation}
Here we used the notation $Y:=\dfrac{W_{\alpha}}{1+W_{\alpha}}$. Both equations \eqref{ngst} and \eqref{diff-ngst} will be used in the sequel.
\begin{remark}
  For convenience we have written the equations with the surface
  tension $\sigma$ included. However, in our case $\sigma$ can be
  scaled out, so through the rest of the paper we will simply set $\sigma = 1$.
\end{remark}

\subsection{Prior work.}
First of all, there is a large body of work devoted to the local well-posedness in the Cauchy
problem for water wave equations  in all combinations, with or without gravity, surface 
tension, finite or infinite depth, periodic or not.

This begins with the early work of Nalimov~\cite{MR0609882} (infinite bottom)
and Yoshihara~\cite{MR728155,MR660822}, (finite bottom, with/without tension)
for the small data problem. The first well-posedness results
for large data are due to Wu~\cite{MR1641609,MR1471885} (no surface tension).
Further results were later obtained by Christodoulou-Lindblad~\cite{MR1780703},
Lannes~\cite{MR2138139}, Coutand and Shkoller~\cite{MR2291920}, Lindblad~\cite{MR1934619},
Beyer-G\"unther~\cite{MR1637554}, Schweizer~\cite{MR2185062}
Ambrose and Masmoudi \cite{MR2162781}, Shatah and Zeng~\cite{MR2388661}. More
recent work has also considered the role of dispersion in such
problems, beginning with Christianson, Hur and Staffilani\cite{MR2763354} and
continuing to the work of Alazard, Burq and
Zuily~\cite{MR2931520, MR2805065,2014arXiv1404.4276A}, 
which represents the current state of
the art. In particular, it is the result in \cite{MR2805065} that we rely on 
for local well-posedness, though even Nalimov's result would serve 
if we were to restrict ourselves to more regular data.

There is also considerable work done for traveling waves, which are
known to exist in both two and three dimensions, see for instance
Amick-Kirchg\"assner~\cite{MR963906}, Groves-Sun \cite{MR2379653}, and
Buffoni-Groves-Sun-Wahlen~\cite{MR2997362},
Groves-Wahlen~\cite{MR2847283}. However, all these results are
concerned with a certain regime ($\beta > \frac13$) in
gravity-capillary problems with either finite or infinite bottom. For
a broader view, see also the recent book of
Constantin~\cite{MR2867413}. To the best of our knowledge, there are
no results on the existence of pure capillary solitary waves in our
setting of infinite depth.

There are fewer long time  results for the small data problem.
 The first results in this direction are due
to Wu~\cite{MR2507638}, who proved almost global existence of gravity waves in
dimension two. The similar three dimensional problem 
was solved independently by Germain-Masmoudi-Shatah~\cite{MR2993751}
and  Wu ~\cite{MR1641609}. Closer to our goal here, the three dimensional 
capillary water wave problem was shown to have global solutions
by Germain-Masmoudi-Shatah~\cite{MR3318019}.

Even more recently, Wu's two dimensional almost global result was
improved to a global result in independent work by
Ionescu-Pusateri~\cite{MR3121725,2013arXiv1303.5357I} and
Alazard-Delort~\cite{MR3460636,MR3429478}. Shortly
after, the authors were able to provide shorter and simpler proofs,
first for the almost global result in joint work with
Hunter~\cite{HIT}, and next for the global result \cite{IT}. The
present article is a natural continuation of this work.

We also note that some preliminary results on a related simplified model problem
were very recently announced\footnote{The full article \cite{IP-model}
was only posted (shortly) after this work.}  by Ionescu and Pusateri\footnote{F. Pusateri,
talk in the Banff workshop of May 11-16, 2014}. Their model only deals 
with the high frequency part of the problem, truncating off the 
low frequency part of the nonlinearity and also removing  the leading order 
quasilinear effects.  Their approach is based on a direct implementation of 
our ``modified energy method'' as originally introduced in \cite{HITW}.

\subsection{ Main results.}

We first describe the function spaces that we use in order to describe
the solutions to the capillary water waves.  The scale $\dH^s$ of
function spaces used in this article is defined by
\[
\dH^s = \dot H^s \times \dot H^{s-\frac12}.
\]
Here the energy corresponds to $s = 1$, while the scaling is at $s = \frac32$.
On occasion we will also use the corresponding inhomogeneous spaces,
\[
\H^s =  H^s \times  H^{s-\frac12}.
\]

Our main goal here is to study the long time persistence of
small data solutions. What we do not do here is to consider the local
well-posedness problem; instead, for that we will rely on the local
well-posedness result of Alazard-Burq--Zuily~\cite{MR2805065}. Their
result is stated in Eulerian coordinates. However, at least for small
data, it is straightforward to transfer it to the holomorphic coordinates,
as the transition map is bounded in Sobolev spaces. 
For our purposes here it suffices to restrict ourselves to integer $s$; then we have:

\begin{theorem}[Alazard-Burq--Zuily~\cite{MR2805065}] The system \eqref{ngst} is locally 
well-posed for small data in $\H^k$ for $k \geq  4$. 
\end{theorem}

This result includes existence, uniqueness and $C^1$ dependence on
data in a weaker topology (e.g. $\H^{k-1}$). If one allows noninteger $k$,
their result applies in the range $k > 3$. A more precise version of
this result in holomorphic coordinates, which also applies to large
data, will be discussed elsewhere. 

All of our results apply to water waves with data as in the above
theorem. However, the bounds on the solutions will not depend as much
on the low frequency regularity of the data, and we will be able to
phrase that in terms of some homogeneous Sobolev norms.

For the estimates on the persistence of solutions we will follow 
an analogue of the set-up in \cite{HIT} and introduce the control norms
\begin{equation}\label{A-def}
A := \|\W\|_{L^\infty}+\| Y\|_{L^\infty} + \Vert D^{-\frac{1}{2}}R\Vert_{L^{\infty}\cap B^{0,\infty}_{2}},
\end{equation}
respectively
\begin{equation}\label{B-def}
B :=\| D^{\frac{1}{2}}\W_{\alpha}\|_{L^\infty} + \Vert R_{\alpha}\Vert_{L^\infty}.
\end{equation}
We note that $A$ is scale invariant, and scales like the $\dH^\frac32$
norm of $(W,Q)$. $B$ on the other hand, scales like the $\dH^3$ norm
of $(W,Q)$. In the corresponding results in \cite{HIT} we were able to
use $BMO$ norms instead of $L^\infty$ in the corresponding definition
of the $B$ norm. In this regard, the results here are slightly less
accurate.
 
Our first result asserts that the persistence of solutions does not depend on the full Sobolev
norm of the solutions, instead it is controlled by the weaker control norm $B$:

\begin{theorem}
\label{baiatul}
Let $n \geq 4$. Then  solutions to the water wave equation
\eqref{ngst} with $\H^n$ data, which is small in $\dH^4 \cap \dH^1$, persist
with uniform bounds in both $\H^n$ and $\dH^4 \cap \dH^1$
for as long as the time integral $\int B$ remains
small. The same result holds in the periodic setting.
\end{theorem}

In particular, we obtain the following cubic lifespan result:

\begin{theorem}
\label{t:cubic}
  Let $\epsilon \ll 1$.  Assume that the initial data for the equation
  \eqref{ngst} on either $\mathbb{R}$ or $\mathbb{S}^1$ satisfies
\begin{equation}
\|(W(0), Q(0))\|_{ \dH^4 \cap \dH^1 } \leq \epsilon.
\end{equation}
Then the solution exists on an $\epsilon^{-2}$ sized time interval
$I_\epsilon = [0,T_\epsilon]$, and satisfies a similar bound. In
addition, the estimates
\[
\sup_{t \in I_\epsilon} \| (\W(t), R(t))\|_{\dH^n} \lesssim \| (\W(0), R(0))\|_{\dH^n},
\qquad n \geq 1,
\]
hold in the same time interval whenever the right hand side is finite.
\end{theorem}

The proof of both theorems above is based on the cubic high frequency energy estimates
in Section~\ref{s:high-en}.

Next we turn our attention to the problem with small localized data,
and show that in this case the solutions are global in time. To state the result we
take advantage of the scale invariance of the equation \eqref{scaling}. This is 
in the spirit of the classical vector field method, first introduced by Klainerman~\cite{Kl}
in the context of nonlinear wave equations. The generator of the scaling group is 
\[
\SS(W,Q) = ((S-1)W,(S-\frac12)Q), \qquad S = \frac32 t \partial_t + \alpha \partial_\alpha.
\]
Then $\SS(W, Q)$ solve the linearized water wave equations, which are described
later in Section~\ref{s:lin}. 

Unlike the case of nonlinear wave equations, here it suffices
to use the scaling field just once, without reiterating it. This is  because this is a
one dimensional problem, with a better range of exponents for Sobolev embeddings.
This was first observed by Ionescu-Pusateri~\cite{2013arXiv1303.5357I} in the context of 
one dimensional gravity  waves.

In order to obtain sufficient pointwise decay we need to use lower Sobolev norms,
both for $(W,Q)$ and for $\SS(W,Q)$. Precisely, we will work in the space 
$\dH^\sigma$ where $0 < \sigma < \frac18$  is a sufficiently small parameter.
Then we  define the weighted energy
\begin{equation}\label{WH}
  \|(W,Q)(t)\|_{\WH}^2 :=  \|(W,Q)(t)\|_{\dH^{10} \cap \dH^\sigma}^2 +  
\|\SS(W,Q)(t)\|_{\dH^1 \cap \dH^\sigma}^2.
\end{equation}
Here we will use a stronger control norm, namely
\begin{equation}
\|(W,Q)\|_{X} := \sum_{j=0,1,2,3,4} \| D^{j+\frac12} W\|_{L^\infty} + \|D^j Q\|_{L^\infty}.
\end{equation}
The $L^\infty$ bound for $W$ will be also controlled but in a more loose fashion,
and  will play a lesser role. Our global result is as follows:
\begin{theorem} \label{t:almost}
For each initial data
$(W(0),Q(0))$ for the system \eqref{ngst} satisfying
\begin{equation}\label{data}
\|(W,Q)(0)\|_{\WH}^2  \leq \epsilon \ll 1,
\end{equation}
the solution is global 
and satisfies
\begin{equation}\label{energy}
\|(W,Q)(t)\|_{\WH}^2  \leq C \epsilon t^{C\epsilon^2},
\end{equation}
as well as 
\begin{equation}\label{point}
\|(W,Q)\|_{X}  \leq C \frac{\epsilon}{\sqrt{t}},
\end{equation}
and 
\begin{equation}\label{pointw}
\|W\|_{L^\infty}  \leq C \epsilon (|\alpha|+t^\frac23)^{\sigma-\frac12}.
\end{equation}

\end{theorem}

The proof of the theorem uses a bootstrap argument. Precisely, it suffices
to prove the conclusion of the theorem under the additional bootstrap assumptions
\begin{equation}\label{energy-boot}
\|(W,Q)(t)\|_{\WH}^2  \leq C \epsilon t^{\frac1{18}},
\end{equation}
as well as 
\begin{equation}\label{point-boot}
\|(W,Q)\|_{X}  \leq 2C \frac{\epsilon}{\sqrt{t}},
\end{equation}
and 
\begin{equation}\label{pointw-boot}
\|W\|_{L^\infty}  \leq 2C \epsilon (|\alpha|+t^\frac23)^{\sigma-\frac12},
\end{equation}
for a fixed large universal constant $C$ independent of $\epsilon$, 
and for $\epsilon$ small enough.

Now we are able to divide the proof into several steps. 

\begin{itemize}
\item[(i)] The high frequency part of the $(W,Q)$
energy estimate (i.e., the $\dH^{10}$ bound) in \eqref{energy} is proved in Section~\ref{s:high-en}, using our 
method of {\em quasilinear modified energy estimates} developed in \cite{HITW}, \cite{HIT}.
\item[(ii)] The low frequency part of the $(W,Q)$ energy estimate
  (i.e. the $\dH^{\sigma}$ bound) in \eqref{energy} is proved in
  Section~\ref{s:low-en}, using Shatah's normal form method
  \cite{MR803256}.
\item[(iii)] The energy estimates for  $S(W,Q)$ in \eqref{energy} are proved in 
Section~\ref{s:lin}, where the linearized equations are considered. This 
is also done in two steps, first for the $\dH^1$ norm, using a quasilinear modified energy,
 and then for the $\dH^\sigma$ norm, using normal forms.
\item[(iv)] A preliminary set of pointwise bounds is obtained from the energy 
estimates via Klainerman-Sobolev type inequalities in Section~\ref{s:ks}. 
These have an additional $t^{c\epsilon^2}$ factor compared with \eqref{point},
so are not quite sufficient to close the argument. However, they have additional decay 
away from the region $|\alpha| \approx t$, so they suffice outside a region 
$t^{-\delta} \leq |\alpha|/t \leq t^\delta$. 

\item[(v)] The final pointwise bounds in \eqref{point} are obtained in the last section
using the method of {\em testing by wave packets} developed in \cite{TI}, \cite{IT}.
\end{itemize}

Before we get to the above steps, in the next two sections we describe
the classes of multilinear operators used in this paper, and we do the
normal form computation for all the non-resonant quadratic and cubic
interactions which are needed later on.

\subsection{Further remarks.}

As mentioned above, this article is a natural continuation of authors' earlier work in 
\cite{HIT}, \cite{TI}. Compared to earlier work of Wu~\cite{MR2507638}, 
Ionescu-Pusateri~\cite{2013arXiv1303.5357I}
and Alazard-Delort~\cite{MR3460636,MR3429478} there 
are three key improvements in \cite{HIT}, \cite{IT}.

\begin{itemize}
\item The use of holomorphic (conformal) coordinates, which greatly
  simplifies the form of the equations and eliminates the use of the
  Dirichlet to Neumann map. This idea is due to
  Ovsjannikov~\cite{MR0347218}, and was further developed in
  \cite{MR1471885}, \cite{DYACHENKO199673} and in the form used here in \cite{HIT}.

\item The \emph{modified energy method} in \cite{HIT}, which can be
  viewed as a \emph{quasilinear} adaptation of Shatah's normal form
  method or equivalently of the I-method. This addresses the well
  known issue that normal form transformations are poorly adapted to
  quasilinear evolutions. The idea is that instead of transforming the
  equations, one can construct a modified energy which is both adapted
  to the quasilinear problem and provides cubic estimates. Our energies
are constructed starting from the corresponding normal form computation.
Alternatively one can base the construction on an I-method computation,
which can be viewed as a subset of the former. The para-diagonalization
idea of Alazard and Delort~\cite{MR3460636,MR3429478}  (see also Delort's 
earlier work~\cite{Delort}) is also a very useful alternate tool that leads toward similar results.

\item The \emph{method of testing by wave packets} in \cite{TI,IT},
  which is a more efficient way to obtain asymptotic equations in
  problems which exhibit modified scattering asymptotics. One should
  compare this with the previous approaches. In the semilinear setting, the first one, 
introduced  by Hayashi-Naumkin~\cite{MR1613646} and refined by
  Kato-Pusateri~\cite{MR2850346}, was Fourier based. The
  Lindblad-Soffer~\cite{MR2199392} idea was to instead produce an
  asymptotic equation by looking at the solution along rays in the
  physical space. In the quasilinear  water wave setting such ideas
were first introduced in the work of  Ionescu-Pusateri~\cite{MR3121725,2013arXiv1303.5357I} and
Alazard-Delort~\cite{MR3460636,MR3429478}; broadly speaking, their approach can be viewed as 
a quasilinear development of the ideas in  \cite{ MR2850346}, respectively 
\cite{MR2199392}.

\end{itemize}

All of these ideas come equally handy in the present paper. The main steps are 
as follows:

\bigskip

{\bf (i) Normal form analysis.} The analysis here parallels the one
for gravity waves, in that bilinear resonant interactions occur only
when either one of the input frequencies or the output frequency is
zero. Thus, a-priori in the normal form there is a singularity at
frequency zero.  Using the full structure of the equations some
cancellations occur (a partial null condition), which is critical for
the present result.  A similar difficulty arises in the work of
Germain-Masmoudi-Shatah~\cite{MR3318019} in higher
dimension. One entirely expected difference, compared to our previous
work \cite{HIT}, \cite{IT}, is that here the normal form symbol is no
longer a polynomial, so that we need to deal with multiplier type
bilinear forms. This is relatively standard and arises in most
problems, beginning with Shatah's original work \cite{MR803256}.

\bigskip 

{\bf (ii) High frequency analysis and energy estimates.} The
quasilinear modified energy method applies as in \cite{HIT}, and
serves to prove high frequency energy estimates. One additional
difficulty here is that the top three terms in the normal form energy
expansion need to be corrected in a quasilinear fashion, unlike in our
gravity waves paper \cite{HIT} where only the top term needs to be
adjusted.

\bigskip 

{\bf (iii) Low frequency analysis and normal forms corrections.}  
The low frequency analysis has a semilinear nature, and is treated
in a manner which is based purely on normal forms.
 The idea is very simple, namely to use Shatah's normal form method to eliminate
quadratic interactions.    In order to gain enough
control on the errors, we also need to go one step further, and
eliminate as well certain cubic unbalanced non-resonant interactions.
Our energy spaces $\dH^\sigma$ at low frequencies are chosen in a manner 
which is similar to Germain-Masmoudi-Shatah~\cite{MR3318019}, 
though their motivation is somewhat different. 

\bigskip

{\bf (iv) The asymptotic equations.} These are used in order to extend
the pointwise decay bounds past the exponential time. The bounds we
obtain from Klainerman-Sobolev type estimates are strong enough both
for very small frequencies (and low speeds) and for very large
frequencies (and high speeds). Thus, we only need to consider the
intermediate range of frequencies not very far from $1$. There the
method of testing by wave packets applies exactly as in \cite{IT}.

\subsection{ Subsequent work} 
The same problem was considered in a later article by
Ionescu-Pusateri~\cite{IP3}.  Their main global result can be viewed
as an improvement of our result in terms of the low frequency
assumptions.  Roughly speaking, it corresponds to setting $\sigma =
1/2$ in \eqref{WH} and Theorem~\ref{t:almost}.  This in turn allows
for data $ W$ (or $\eta$ in the Eulerian setting) which may decay to
different heights at $\pm \infty$.  The price to pay is that unlike
here, an almost global bound for the linearized equation is no longer
established in \cite{IP3}; instead, only the bound for the scaling
derivative of the solution is proved. A version of
Theorem~\ref{t:cubic} is also proved in \cite{IP3}, but in this case
their result is strictly weaker in terms of the Sobolev norms which
are used.

The proof of the result in \cite{IP3} is based on the same modified
energy construction introduced in our earlier work \cite{HITW},
\cite{HIT} and also used here. However, this is applied in the
Eulerian setting and using also Alazard and Delort's
~\cite{MR3460636,MR3429478} para-diagonalization
idea.

\bigskip

{\bf Acknowledgements:}  The final part of this  research was carried out 
while both authors were visiting the Haussdorf Institute in Bonn. 


\section{ Paraproduct type multilinear operators}

A key role in the analysis in this paper is played by paraproduct type
translation invariant bilinear and multilinear forms. Since our state
space consists of holomorphic functions, i.e., with negative spectrum,
we will only consider operators acting on this space.  We begin with
bilinear forms, which are of two types:
\begin{itemize}
 \item holomorphic, i.e., of the form
\[
\widehat{B(u,v)}(\zeta) = \int_{\zeta = \xi+\eta} m(\xi,\eta) \hat u(\xi) \hat v(\eta) d\xi ,
\]
\item mixed, i.e., of the form
\[
\widehat{B(u,v)}(\eta) = 1_{\eta > 0} \int_{\eta = \zeta-\xi}  m(\xi,\zeta) \hat u(\zeta) \bar{\hat{v}}(\xi) d\xi .
\]
\end{itemize}
These represent paradifferential generalizations of the standard
products $uv$, respectively $P(u \bar v)$, restricted to the holomorphic class.
Our bilinear symbols $m$ will always be homogeneous, and smooth away from $(0,0)$.
 
\begin{definition}
  We denote by $\M_k$ the class of symbols $m$ as above which are
  homogeneous of order $k \in \Z$ and smooth away from $(0,0)$. The
  corresponding classes of operators are denoted by $OP\M_k^{(2,0)}$
  for operators of holomorphic type, by $OP\M_k^{(1,1)}$ for operators
  of mixed type, and by $OP\M_k$ for combinations thereof. By
  $L^{(2,0)}_k$, $L^{(1,1)}_k$ and $L_k$ we denote operators in the
  respective classes.
\end{definition}

Such bilinear and  multilinear forms satisfy multiplicative estimates which are 
quite similar to product type bounds, see Coifman-Meyer~\cite{MR518170}, 
Kenig-Stein~\cite{MR1682725},
Muscalu-Tao-Thiele~\cite{MR1887641}, Muscalu~\cite{MR2371442}:

\begin{proposition}\label{bi-est}
Let $L_0$ be an order zero bilinear form as above. Then we have 
\[
\|L(f,g)\|_{L^r} \lesssim \|f\|_{L^p} \|g\|_{L^q}, \qquad 
\frac{1}p+\frac1q = \frac1r, \quad 1 \leq r < \infty \ \ \ \  1 < p ,q \leq \infty .
\]
\end{proposition}

We also need to work with trilinear forms in the sequel. There we have three
types, which are denoted using the superscripts $(3,0)$, $(2,1)$ respectively 
$(1,2)$.  However, in order to deal with all the cases arising  in this article, 
 we need to enlarge slightly the class of homogeneous 
symbols we allow.  Since we are only concerned with homogeneous symbols,
it suffices to describe the zero order symbols.
There are three such classes of symbols of order zero that we need to allow:
\begin{itemize}
\item smooth homogeneous symbols,
\item compositions of two smooth bilinear symbols,
\item polyhomogeneous symbols of the form
\[
|\xi_1|^{c_1} |\xi_2|^{c_2} |\xi_3|^{c_3} |\xi_{out}|^c_4 m(\xi_1,\xi_2,\xi_3,\xi_4), \qquad c_1,c_2,c_3,c_4 \geq 0, 
\]
with smooth homogeneous $m$.
\end{itemize}
Here $\xi_{out}:= \pm\xi_1 \pm \xi_2 \pm \xi_3$ represents the output
frequency. For such symbols we also have the following result:

\begin{proposition}\label{tri-est}
Let $L_0$ be an order zero trilinear form as above. Then we have 
\[
\|L(f,g,h)\|_{L^r} \lesssim \|f\|_{L^p} \|g\|_{L^q}\|h\|_{L^s}, \qquad 
\frac{1}p+\frac1q +\frac1s= \frac1r, \quad 1 \leq r < \infty, \ \ 1 < p,q,s \leq \infty. 
\]
\end{proposition}

\begin{proof}[Outline of proof]
  This proposition is special case of the results in
  \cite{MR518170,MR1682725,MR1887641,MR2371442} for the first two
  classes of symbols above. For the third class it is a small
  variation on the same theme.  By duality and separation of variables
  arguments the problem reduces to the bilinear case and paraproduct
  type symbols of the form
\[
a(\xi_1,\xi_2) = |\xi_1|^c m(\xi_1,\xi_2), \qquad c > 0 
\]
with $m$ smooth, homogeneous, and localized in $ |\xi_1| \ll |\xi_2|$.  Localizing this 
further to regions where $ |\xi_1| \approx 2^j |\xi_2|$, $j < 0$, this is now a classical paraproduct
with a $2^{-cj}$ bound, and the $j$ summation is straightforward.
\end{proof}  

Since our problem is a system and the two variables $(W,Q)$ correspond to different
homogeneities, we will also need notations for homogeneous bilinear and trilinear 
forms. Precisely, for $k \in \Z$ we define classes $(\B_k^2,\C_k^2)$, respectively 
$(\B_k^3,\C_k^3)$ of bilinear and trilinear  forms as follows:

\begin{itemize}
\item $\B_k^2(W,Q)$ contains bilinear forms of the type
$L_{k+1}(W,W) + L_k(Q,Q)$.

\item  $\C_k^2(W,Q)$ contains bilinear forms of the type
$L_{k+1}(W,Q)$.

\item $\B_k^3(W,Q)$ contains trilinear forms of the type
$L_{k+2}(W,W,W) + L_{k+1}(W,Q,Q)$.

\item  $\C_k^3(W,Q)$ contains trilinear forms of the type
$L_{k+2}(W,W,Q)+ L_{k+1}(Q,Q,Q) $.
\end{itemize}


\section{Normal form expansions and corrections}

In this section we collect most of our normal form computations. The goal here is
twofold; first to use quadratic normal form corrections in order to eliminate the quadratic 
terms in the equation,  and then to use cubic normal form corrections in order 
 to eliminate certain non-resonant parts of cubic  terms in the equation.

To understand these computations, one needs to start from the dispersion relation
associated to the linear system
\begin{equation} \label{ww-lin}
\left\{
\begin{aligned}
& W_t + Q_\alpha = 0 \\
& Q_t +i W_{\alpha\alpha} = 0,
\end{aligned} 
\right.
\end{equation}
on holomorphic functions. Recast as a scalar equation this becomes
\[
(\partial_t^2 +  i  \partial_\alpha^3) W = 0,
\]
which immediately leads to the dispersion relation
\[
\tau = \pm |\xi|^\frac32, \qquad \xi < 0.
\]
Thus resonances in bilinear interactions correspond to zeroes
of the expression $|\xi|^\frac32 \pm |\eta|^\frac32 \pm |\xi+\eta|^\frac32$. But this 
cannot vanish unless either $\xi$ or $\eta$ are zero. Consequently, 
all quadratic interactions are removable, but with an appropriate loss of derivatives
at (the resonant) frequency zero. As a guideline, if $\xi$, and $\eta$ 
have dyadic sizes $\lambda_0 \leq \lambda_1$, then in the worst case we have
\[
\left| |\xi+\eta|^\frac32 - |\xi|^\frac32 - |\eta|^\frac32 \right| \approx \lambda_0 \lambda_1^\frac12.
\] 
Thus, compared to the  quadratic component of the nonlinearity, the quadratic
normal form looses one derivative at low frequency and half a derivative
at high frequency. A similar analysis applies in the case of cubic energy errors.
For later use, we observe that we have the relation
\[
\prod_{\pm} |\xi|^\frac32 \pm |\eta|^\frac32 \pm |\xi+\eta|^\frac32 = \xi^2 \eta^2(9 \xi^2+ 14\xi \eta + 9 \eta^2).
\]
The last expression above is elliptic, and will appear in the
denominator of all symbols for bilinear normal form corrections later on.

Next we discuss  trilinear terms. There we have four frequencies to deal with, $\xi_1,\xi_2$
and $\xi_3$ for the inputs, and  $\xi_0= \pm \xi_1 \pm \xi_2 \pm \xi_3$ for the 
output.  In this case we can have 
\[
|\xi_1|^\frac32 \pm |\xi_2|^\frac32 \pm 
|\xi_1|^\frac32 \pm |\xi_2|^\frac32 = 0
\]
 if and only if there are two pairs of equal frequencies
and matching signs in the last relation.  Fortunately in our analysis
we only need to consider the case where the four dyadic frequencies are unbalanced,
\[
\lambda_0 \ll \lambda_1 \leq \lambda_2 \approx \lambda_3.
\]
Then the interactions are non-resonant, and we have as in the bilinear case
\begin{equation}\label{nonres-tri}
\left| |\xi_1|^\frac32 \pm |\xi_2|^\frac32 \pm 
|\xi_1|^\frac32 \pm |\xi_2|^\frac32 \right| \gtrsim \lambda_1 \lambda_2^\frac12.
\end{equation}

\subsection{ The main normal form transformation}

We begin with the quadratic and cubic expansion in the equation
\eqref{ngst}, and then we compute the normal form transformation which
eliminates the quadratic terms from the equation. 
Starting with 
\[
F \approx Q_\alpha - Q_\alpha W_\alpha + P[\bar Q_\alpha W_\alpha - Q_\alpha \bar W_\alpha] 
+ P[ (Q_\alpha -\bar Q_\alpha) (4|\Re W_\alpha|^2- |W_\alpha|^2)] 
\]
we compute  the multilinear expansion:
\begin{equation} \label{ww-multi}
\left\{
\begin{aligned}
& W_t + Q_\alpha = G^{(2)} + G^{(3)} + G^{(4+)}, \\
& Q_t +i W_{\alpha\alpha} = K^{(2)}  + K^{(3)}+K^{(4+)},
\end{aligned} 
\right.
\end{equation}
where the quadratic terms $(G^{(2)},K^{(2)}) \in
(\C^2_{-1}(W_\alpha,Q_\alpha),\B_0^2(W_\alpha,Q_\alpha))$ are given by
\begin{equation}\label{ww-quad}
\begin{cases}
G^{(2)} := \ P\left[ Q_\alpha \bar W_\alpha - \bar Q_\alpha W_\alpha \right]&
\\
K^{(2)} :=   -  Q_\alpha^2 -  P\left[ |Q_\alpha|^2\right] 
+\frac{i}2  P\left[ W_{\alpha \alpha}(3W_\alpha + \bar W_\alpha)  
- \bar W_{\alpha \alpha}{W}_{\alpha}\right] ,&
\end{cases}
\end{equation}
and  the cubic terms $(G^{(3)},K^{(3)})$ are given by 
\begin{equation}\label{ww-cubic}
\begin{cases}
G^{(3)} = \  W_\alpha ( Q_\alpha W_\alpha 
- P[\bar Q_\alpha W_\alpha - Q_\alpha \bar W_\alpha] ) 
- P[ (Q_\alpha -\bar Q_\alpha) (4|\Re W_\alpha|^2- |W_\alpha|^2)] &
\\
K^{(3)} =   \ Q_\alpha ( Q_\alpha W_\alpha 
- P[\bar Q_\alpha W_\alpha - Q_\alpha \bar W_\alpha] ) + 2 P[\Re \W_\alpha|Q_\alpha|^2]&
\\ 
 \qquad \quad  \ - i P\left[W_{\alpha\alpha} ( \frac{15}8 W_\alpha^2 + \frac{3}4 |W_\alpha|^2 + \frac38 \bar W_\alpha^2)  - \bar W_{\alpha\alpha} (  \frac{3}4 |W_\alpha|^2 + \frac38  W_\alpha^2)\right].&
\end{cases}
\end{equation}

 The role of the normal form transformation is to eliminate the quadratic terms $(G^{(2)},K^{(2)})$
from the equation \eqref{ww-multi}. It is obtained  as follows:

\begin{proposition}\label{p:nf}
The normal form transformation for the equation \eqref{ww-multi} is given by 
\begin{equation}\label{nft}
\left\{
\begin{aligned}
\tW = & \ W + W_{[2]}
\\
\tQ = & \ Q  + Q_{[2]},
\end{aligned}
\right.
\end{equation}
where the bilinear forms $(W_{[2]},Q_{[2]}) \in (\B_0(W,Q), \C_0(W,Q))$
have the form
\begin{equation}\label{nft-exp}
\left\{
\begin{aligned}
W_{[2]} =  & \ B^h(W,W) +  C^h(Q,Q)+ B^a(W,\bar W) +  C^a(Q,\bar Q)
\\
Q_{[2]} = & \ A^h(W,Q)+ A^a(W,\bar Q)+ D^a(Q,\bar W),
\end{aligned}
\right.
\end{equation}
with $A^h$, $A^a$, $B^h$, $B^a$, $D^a \in OP\M^1$, 
while $C^h$, $C^a \in OP\M^0$,  whose symbols are as follows:
\[
B^h = - i (\xi+\eta) \frac{\frac94(\xi+\eta)^2 -2 \xi \eta}{9(\xi+\eta)^2 - 4 \xi \eta}, \qquad
C^h = -  i \frac{\frac32(\xi+\eta)^2}{9(\xi+\eta)^2 - 4 \xi \eta},
\]
\[
A^h = i \frac{3\xi^3 - \frac32 \xi^2 \eta - 5 \xi \eta^2 -\frac92 \eta^3}{9(\xi+\eta)^2 - 4 \xi \eta},
\]
respectively
\[
\begin{cases}
A^a (\zeta, \eta)= \dfrac{ 3i (\xi+\eta) ( \xi^2 + \frac32 \xi \eta +\frac{3}{2}\eta^2)}{9 (\xi+\eta)^2 - 4 \xi \eta }, \quad  &B^a (\zeta, \eta)=  \dfrac{-i (\xi+\eta )( \frac92 \xi^2 + \frac{19}{2} \xi \eta +6\eta^2)}{9 (\xi+\eta)^2 - 4 \xi \eta },
\\ 
&\\
C^a (\zeta, \eta)= \dfrac{- 3 i (\xi+\eta)^2}{9 (\xi+\eta)^2 - 4 \xi \eta },\quad &D^a(\zeta , \eta) =\dfrac{-i (\xi+\eta) ( \frac92 (\xi+\eta)^2 - 4 \xi \eta )}{9 (\xi+\eta)^2 - 4 \xi \eta}.
\end{cases}
\]

\end{proposition}

We postpone the proof of the Proposition~\ref{p:nf} for the end of the section, and instead 
we discuss some consequences.  

The main direct consequence of the normal form transformation is the fact that
the normal form variables $(\tW,\tQ)$ solve a cubic equation,
\begin{equation}\label{nft-ww}
\left\{
\begin{aligned}
 \tW_t + \tQ_\alpha =  \ \tG(W,Q)
\\
 \tQ_t +i \tW_{\alpha\alpha} = \ \tK(W,Q),
\end{aligned} 
\right.
\end{equation} 
where the expressions $(\tG,\tK)$ contain only cubic and higher order terms.

These equations cannot be used directly at high frequencies, as they neglect the 
quasilinear character of the problem. However, they are very useful at low frequencies,
where one can roughly view the problem as a semilinear equation.
For later use we decompose them into cubic and higher order terms,
\[
\tG = \tG^{(3)} + \tG^{(4+)}, \qquad \tK = \tK^{(3)} + \tK^{(4+)},
\]
where
\begin{equation}\label{gkt}
\begin{cases}
\tilde{G}^{(3)}=2B^h(G^{(2)}, W)+ 2C^h(K^{(2)}, Q)+B^{a}(G^{(2)},\bar{W})+B^a(W,\bar{G}^{(2)})&\\
                                                \hspace{3.8cm} +\left( C^a(K^{(2)}, \bar{Q})+C^a(Q,\bar{K}^{(2)})\right) +G^{(3)} (W, Q)&\\
 \tilde{K}^{(3)}=A^h(W, K^{(2)}) + A^h(G^{(2)}, Q)+A^a(G^{(2)},\bar{ Q})+A^{a}(W, \bar{K}^{(2)})+D^a(K^{(2)}, \bar{W})\\
                                                \hspace{3.6cm} +D^a(Q, \bar{G}^{(2)}) +K^{(3)}(W,Q),&\\
                           \end{cases}
\end{equation}
and
\begin{equation*}
\begin{cases}
\tilde{G}^{(4+)}=2B^h(G^{(3+)}, W)+ 2\sigma^{-1}C^h(K^{(3+)}, Q)+B^{a}(G^{(3+)},\bar{W})+B^a(W,\bar{G}^{(3+)})&\\
                                                \hspace{4.25cm} +\sigma^{-1}\left( C^a(K^{(3+)}, \bar{Q})+C^a(Q,\bar{K}^{(3+)})\right) +G^{(4+)} (W, Q)&\\
 \tilde{K}^{(4+)}=A^h(W, K^{(3+)}) + A^h(G^{(3+)}, Q)+A^a(G^{(3+)},\bar{ Q})+A^{a}(W, \bar{K}^{(3+)})+D^a(\bar{K}^{(3+)}, \bar{W})\\
                                                \hspace{4.1cm} +D^a(Q, \bar{G}^{(3+)}) +K^{(4+)}(W,Q).&\\
                           
\end{cases}
\end{equation*}

While the above expressions are somewhat more complicated than pure
multiplication operators,  for some of our purposes  it will  suffice to
extract the leading order term in the normal form expression and treat the rest 
in a simpler manner. For the low frequency portion of the energy estimates we 
will be particularly interested in the contributions which contain undifferentiated
$W$, and, to a lesser extent,  those containing undifferentiated $Q$'s.
To compute those, we first consider the expansion for $(\tW,\tQ)$, which is:
\begin{equation}\label{nft-lead}
\left\{
  \begin{aligned}
\tW = & \ W - \frac12 \Re W W_\alpha - i  (\frac{1}6 Q^2 + \frac13 \bar Q Q)+ 
L_{-1}(W_\alpha,  W_\alpha) + L_{-2}(Q_\alpha,Q_\alpha) 
\\
\tQ = & \ Q - \frac12 \Re W Q_\alpha - \frac13 \Re Q W_\alpha + 
L_{-1}(W_\alpha, Q_\alpha) .
\end{aligned}
\right.
\end{equation}
Now we observe that only the time differentiation of the leading terms
can yield expressions with undifferentiated $W$ or $Q$. One can see
this with a computation as follows:
\[
\frac{d}{dt} L_{-1}(W_\alpha,Q_\alpha) = L_{0}(W_t,Q_\alpha) + L_0(W_\alpha,Q_t).
\]
Thus using \eqref{nft-lead} we can compute
\begin{equation} \label{gkt-nodiff}
\left\{
  \begin{aligned}
 \tilde{G}^{(3)}=&  \ - \frac12 \Re W G^{(2),\alpha}+ L_0(W_\alpha,W_\alpha,Q_\alpha)
+ L_0(Q,Q_\alpha,Q_\alpha)
\\ 
 \tilde{K}^{(3)}= & \  - \frac12 \Re W K^{(2),\alpha} -\frac13 \Re Q  G^{(2),\alpha} +
L_0(Q_\alpha,Q_\alpha,W_\alpha) + L_0(W_\alpha,W_\alpha,W_{\alpha\alpha}).
\end{aligned}
\right.
\end{equation}

 Finally, for the global pointwise estimates in the last section, we
  primarily need to understand the resonant interactions of three
  equal frequency waves. These are not always captured in the
expansion in  \eqref{nft-lead}.

\begin{proof}[Proof of Proposition~\ref{p:nf}]

  We distinguish two types of quadratic terms on the right in
  \eqref{ww-multi}, the holomorphic ones and the mixed ones. The normal
  forms for the two are completely separate, so we do two computations:
  
\bigskip

{\bf (i) Holomorphic terms:} Here we work with the system
\begin{equation*}
\left\{
\begin{aligned}
& W_t + Q_\alpha =cubic, \\
& Q_t +i W_{\alpha\alpha} =  -  Q_\alpha^2 
+\frac{3 i}2   W_{\alpha \alpha} W_\alpha
  + cubic.
\end{aligned} 
\right.
\end{equation*}
By checking parity, our normal form 
must be
\[
\tW = W + B^h(W,W) + C^h(Q,Q), \qquad \tQ = Q + A^h(W,Q) ,
\]
where $B^h$ and $C^h$ are symmetric bilinear forms with symbols $B^h(\xi,\eta)$,
$C^h(\xi,\eta)$ and $A^h$ is arbitrary.

We compute 
\begin{equation*}
\left\{
\begin{aligned}
& \tW_t + \tQ_\alpha =  - 2B^h(W,Q_\alpha) - 2i C^h(W_{\alpha\alpha},Q)    + \partial_\alpha
A^h(W,Q) + cubic
 \\
& \tQ_t +i\tW_{\alpha\alpha} =  -A^h(Q_\alpha,Q) - i\sigma A^h(W,W_{\alpha\alpha})
+ i  \partial_\alpha^2( B^h(W,W) +  C^h(Q,Q))
-  Q_\alpha^2 \\
&\hspace{2.4cm}+\frac{3 i}2  W_{\alpha \alpha} W_\alpha
  + cubic.
\end{aligned} 
\right.
\end{equation*}
We denote the two input frequencies by $\xi$ and $\eta$. Then for the symbols we get 
\[
\begin{cases}
 2\eta B^h - 2 \xi^2 C^h - (\xi+\eta) A^h =  \  0&
\\
 (\xi A^h)_{sym} + (\xi+\eta)^2 C^h =  -\ i  \xi \eta  &
\\
(\eta^2 A^h)_{sym} -( \xi+\eta)^2 B^h =   \ \frac{3i}4\xi \eta(\xi+\eta).&
\end{cases}
\]
Here $sym$ stands for the symmetrization. From the first equation we get
\[
A^h = \frac{1}{\xi+\eta}( 2 \eta B^h - 2 \xi^2 C^h).
\]
Hence
\[
(\xi A^h)_{sym} = \frac{1}{\xi+\eta}( 2 \xi \eta B^h - (\xi^3+\eta^3) C^h),
\quad 
(\eta^2 A^h)_{sym} = \frac{1}{\xi+\eta}( (\xi^3+ \eta^3) B^h - 2 \xi^2 \eta^2 C^h).
\]
Replacing back into the original equations we get
\[
\begin{cases}
2 \xi \eta B^h  + 3 \xi \eta(\xi+\eta) C^h = -i \xi \eta(\xi+\eta)&\\
3 \xi \eta(\xi+\eta) B^h + 2\xi^2 \eta^2 C^h =  -\frac{3i}4\xi \eta(\xi+\eta)^2,&
\end{cases}
\]
and after simplifications
\[
\begin{cases}
 2  B^h + 3 (\xi+\eta) C^h = - i (\xi+\eta)&
\\
3 (\xi+\eta) B^h + 2\xi \eta C^h = - \frac{3i}4 (\xi+\eta)^2.&
\end{cases}
\]
Solving this yields
\[
A^h = i \frac{3\xi^3 - \frac32 \xi^2 \eta - 5 \xi \eta^2 -\frac92 \eta^3}{9(\xi+\eta)^2 - 4 \xi \eta},
\]
\[
B^h = - i (\xi+\eta) \frac{\frac94(\xi+\eta)^2 -2 \xi \eta}{9(\xi+\eta)^2 - 4 \xi \eta}, \qquad
C^h = -  i \frac{\frac32(\xi+\eta)^2}{9(\xi+\eta)^2 - 4 \xi \eta}.
\]
Here  the denominator is always nonzero.

\bigskip 

{\bf (ii) Mixed terms:} Here we have
\begin{equation*}
\left\{
\begin{aligned}
& W_t + Q_\alpha = P\left[ Q_\alpha \bar W_\alpha - \bar Q_\alpha W_\alpha\right]  + cubic \\
& Q_t +i W_{\alpha\alpha} =   -  P\left[ |Q_\alpha|^2\right] 
+\frac{i}2  P\left[ W_{\alpha \alpha} \bar W_\alpha 
- \bar W_{\alpha \alpha}{W}_{\alpha}\right]   + cubic.
\end{aligned} 
\right.
\end{equation*}
By checking parity, our normal form must be
\[
\tW = W + B^a(W,\bar W) + C^a(Q,\bar Q), \quad \tQ = Q + A^a(W,\bar Q)
+ D^a(Q,\bar W) ,
\]
where $B^a$ and $C^a$ are even bilinear forms
and $A^a$ is arbitrary.

We compute 
\begin{equation*}
\left\{
\begin{aligned}
\tW_t + \tQ_\alpha =  & \  - B^a(W,\bar Q_\alpha) - B^a(Q_\alpha,\bar W) 
- i C^a(W_{\alpha\alpha},\bar Q) +  i C^a(Q,\bar W_{\alpha\alpha})
\\ & \    
 + \partial_\alpha A^a(W,\bar Q)  + \partial_\alpha D^a(Q,\bar W) + 
P\left[ Q_\alpha \bar W_\alpha - \bar Q_\alpha W_\alpha \right] + 
cubic
 \\
\tQ_t +i \tW_{\alpha\alpha} =  & \  - A^a(Q_\alpha,\bar Q) 
+ i A^a(W,\bar W_{\alpha\alpha}) - D^a( Q, \bar Q_\alpha) - i  D^a(W_{\alpha \alpha} , \bar W)
-  P\left[ |Q_\alpha|^2\right] 
\\ & \ + i  \partial_\alpha^2( B^a(W,\bar W) +  C^a(Q,\bar Q))
+\frac{i}2  P\left[ W_{\alpha \alpha} \bar W_\alpha 
- \bar W_{\alpha \alpha}{W}_{\alpha}\right] 
  + cubic.
\end{aligned} 
\right.
\end{equation*}
We denote the two input frequencies by $\zeta$ and $\eta$, where
$\eta$ is always associated to the function that is
conjugated. Matching the like terms we get the symbol relations
\[
\begin{cases}
i \eta B^a + i \zeta^2 C^a + i(\zeta -\eta) A^a =  \zeta \eta &
\\
- i \zeta B^a - i \eta^2 C^a + i(\zeta-\eta) D^a = - \zeta \eta &
\\
- i \zeta A^a + i \eta D^a - i (\zeta-\eta)^2 C^a = \zeta \eta &
\\
- i \eta^2 A^a + i \zeta^2 D^a - i (\zeta-\eta)^2 B^a =   \frac12 \zeta \eta(\zeta+\eta). &
\end{cases}
\]
After algebraic manipulations we pull out a factor of $\eta^2(\zeta-\eta)^2$
to obtain the simpler system
\[
\left[\begin{array}{cccc}
1 & -1 & \zeta+\eta & 1 \cr 
0 & \zeta & \eta^2 & \eta -\zeta \cr
0  & -2 & \eta  & -1   \cr 
0 & 0 & 3\zeta & 2 
\end{array} \right]
\left[ \begin{array}{c} A \cr B \cr C \cr D \end{array}\right] 
= - i \zeta  \left[ \begin{array}{c} 0 \cr \eta \cr -\frac32 \cr 2 \end{array}\right] .
\]
This has solutions
\[
\begin{cases}
A^a =   \dfrac{ 3i \zeta( \zeta^2 - \frac12 \zeta \eta +\eta^2)}{9 \zeta^2 - 4 \zeta \eta + 4 \eta^2}, \quad  &B^a =  \dfrac{-i \zeta( \frac92 \zeta^2 + \frac12 \zeta \eta +\eta^2)}{9 \zeta^2 - 4 \zeta \eta + 4 \eta^2},
\\ 
&\\
C^a =   \dfrac{- 3 i \zeta^2}{9 \zeta^2 - 4 \zeta \eta + 4 \eta^2},\quad &D^a =  \dfrac{-i \zeta( \frac92 \zeta^2 - 4 \zeta \eta +4 \eta^2)}{9 \zeta^2 - 4 \zeta \eta + 4 \eta^2}.
\end{cases}
\]
Substituting $\zeta = \xi+\eta$ we obtain the relations from the proposition.

\end{proof}

\subsection{Cubic normal form  corrections}

Here we construct a cubic normal form correction for a trilinear non-resonant form.
We will apply our result both to the water wave equation and to its linearization,
so it is convenient to work with uncoupled variables
$(W_\alpha^1,  Q_\alpha^1)$,   $(W_\alpha^2, Q_\alpha^2)$, $(W_\alpha^3, Q_\alpha^3)$.

\begin{proposition}\label{p:nf-cubic}
  Suppose that $(W_\alpha^1, Q_\alpha^1)$, $(W_\alpha^2, Q_\alpha^2)$,
  $(W_\alpha^3,  Q_\alpha^3)$ solve the linear system \eqref{ww-lin}
  and have frequencies restricted to the range 
\[
|\xi_1 + \xi_2 - \xi_3| \ll \min \{ |\xi_1|,|\xi_2|\}.
\]
Let $(G,K)$ be trilinear forms  of type $(2,1)$ or $(1,2)$,
\[
(G,K) \in (  \C_{-2}((W_\alpha^1 Q_\alpha^1),   (W_\alpha^2, Q_\alpha^2), 
(W_\alpha^3 Q_\alpha^3)) ,  \B_{-1}((W_\alpha^1 Q_\alpha^1),   (W_\alpha^2, Q_\alpha^2), 
(W_\alpha^3 Q_\alpha^3))).
\]
Then there exists trilinear forms
\[
(W_{[3]},Q_{[3]}) \in  (\B_{0}( (W^1, Q^1),   (W^2, Q^2), 
(W^3, Q^3)),
\C_{0}((W^1, Q^1),   (W^2, Q^2), 
(W^3, Q^3)))
\]
so that we have 
\begin{equation}\label{cubic-cor}
\left\{
\begin{aligned}
& W_{[3],t} + Q_{[3],\alpha} = G \\
& Q_{[3],t} +i W_{[3],\alpha\alpha} = K.
\end{aligned} 
\right.
\end{equation}
\end{proposition}
We note that there are more parity combinations possible in a result of this type.
To shorten the computations, we confine ourselves to the situation above, which is the 
only one occurring in the present paper.

\begin{proof}
To limit the number of cases in our discussion we confine ourselves to $(2,1)$ forms,
and also set $G=0$. The remaining cases are all quite similar.
Then there are four terms in $K$,
\[
K = K_1(W_\alpha^1,W_{\alpha}^2, \bar Q_\alpha^{3})
+  K_2(W_\alpha^1,Q_{\alpha}^2, W_\alpha^{3})
+  K_3(Q_\alpha^1,W_{\alpha}^2, W_\alpha^{3})
+ K_4(W_\alpha^1,W_\alpha^2,W_\alpha^3).
\]
Substituting directly into the equations, we would obtain an $8 \times 8$ system, which is 
fairly large. Hence it is convenient to diagonalize it first.
For that we make the substitution 
\[
Z^j_{\pm} := Q^j \pm i  |D|^\frac12 W^j ,
\]
where $Z^j_{\pm}$ solve the scalar equations
\[
\partial_t Z^j_{\pm} = \pm i  |D|^\frac32  Z^j_{\pm}.
\]
Further, we set 
\[
Z_{[3],\pm} := Q_{[3]}  \pm i |D|^\frac12 W_{[3]}. 
\]
Then the system \eqref{cubic-cor} reads
\[
\partial_t Z_{[3],\pm} = \pm i  |D|^\frac32  Z^j_{\pm} + K.
\]
If we were to set 
\[
K := K_{\pm,\pm,\pm} = Z^1_{\pm} Z^2_{\pm} \bar Z^3_{\pm},
\]
then for $Z_{[3],\pm}$ we obtain the expression 
\[
Z_{[3],\pm} = L_{\pm,\pm,\pm,\pm}(Z^1_{\pm} , Z^2_{\pm}, \bar Z^3_{\pm}),
\]
where the multilinear form $L_{\pm,\pm,\pm,\pm}$ has symbol
\[
 L_{\pm,\pm,\pm,\pm}(\xi_0,\xi_1,\xi_2,\xi_3,\xi_4) = \frac{-i }{ \pm |\xi_0|^{\frac32} \mp |\xi|_1^{\frac32} \mp |\xi_2|^\frac32
\pm |\xi_3|^\frac32} .
\]
Here the signs on the bottom correspond to complex conjugates in the
corresponding quadrilinear form.  In our case we have an expression in
$W^i$ and $Q^i$, which must be written in terms of $Z^i_{\pm}$, and
the final result must be reexpressed in terms of $W^i$ and $Q^i$.
Thus we obtain a combination of symmetrizations and
antisymmetrizations of the above symbol with respect to the four
signs. Half of these are zero by symmetry. For the rest, we quickly
list them and describe their size.

We begin with the terms in $Q_{[3]}$, for which we have the following
symbols based on the arguments for inputs and for the outputs.
\[
L(WWW \to WWQ) =
\sum_{\pm}   \frac{ -  (\pm   |\xi_3|^{-\frac12} ) }
{ \pm |\xi_0|^{\frac32} \mp |\xi|_1^{\frac32} \mp |\xi_2|^\frac32 \pm |\xi_3|^\frac32}, 
\]
\[
L(WWW \to QQQ) =
\sum_{\pm}   \frac{ (\pm  |\xi_1|^{-\frac12}) (\pm |\xi_2|^{-\frac12})(\pm  |\xi_3|^{-\frac12})
 }
{ \pm |\xi_0|^{\frac32} \mp |\xi|_1^{\frac32} \mp |\xi_2|^\frac32 \pm |\xi_3|^\frac32} ,
\]
\[
L(WQQ \to QQQ) = 
\sum_{\pm}   \frac{ -  (\pm  |\xi_1|^{-\frac12})  
 }
{ \pm |\xi_0|^{\frac32} \mp |\xi|_1^{\frac32} \mp |\xi_2|^\frac32 \pm |\xi_3|^\frac32} ,
\]
\[
L(WQQ \to WWQ ) = 
\sum_{\pm}   \frac{ -  (\pm  |\xi_2|^{-\frac12}) 
 }
{ \pm |\xi_0|^{\frac32} \mp |\xi|_1^{\frac32} \mp |\xi_2|^\frac32 \pm |\xi_3|^\frac32} ,
\]
\[
L(WQQ \to QWW ) = 
\sum_{\pm}   \frac{ -  (\mp  |\xi_1|^{\frac12}) (\pm |\xi_2|^{-\frac12}) (\pm |\xi_3|^{-\frac12}) 
 }
{ \pm |\xi_0|^{\frac32} \mp |\xi|_1^{\frac32} \mp |\xi_2|^\frac32 \pm |\xi_3|^\frac32} .
\]
Here the signs on the top are matched to the signs associated to the
same variable on the bottom. We need to understand the size of these
symbols when $|\xi_0| \ll |\xi_{med}|
:=\min\{|\xi_1|,|\xi_2|,|\xi_3|\}$. We also denote $|\xi_{hi}|
:=\max\{|\xi_1|,|\xi_2|,|\xi_3|\}$. We first look at the size of the
denominators, for which we use the relation \eqref{nonres-tri}.  This
suffices if we have only high frequencies in the numerator, and we get
a symbol bound of the type
\[
|L| \lesssim |\xi_{med}|^{-1} |\xi_{hi}|^{-k},
\]
where the integer $k$ is dictated by the corresponding homogeneity of $L$.
It is also not difficult to see that all half integer powers disappear in the symmetrizations,
and by taking a Taylor series at $\xi_0=0$ we are left with a sum of 
polyhomogeneous symbols and a smooth (i.e., with many derivatives) remainder.

Consider now the case when the medium frequency appears in the numerator with a 
negative power. Then the sum is antisymmetrized with respect to the medium frequency,
so we get an extra cancellation, as in 
\[
\begin{split}
\frac{|\xi|_{med}^{-\frac12}}{ \pm |\xi_0|^{\frac32} + |\xi_{med}|^{\frac32} \mp |\xi_2|^\frac32 \pm |\xi_3|^\frac32} - \frac{|\xi|_{med}^{-\frac12}}{ \pm |\xi_0|^{\frac32} - |\xi_{med}|^{\frac32} \mp |\xi_2|^\frac32 \pm |\xi_3|^\frac32} = \\
\frac{-2 |\xi_{med}|} { (\pm |\xi_0|^{\frac32} + |\xi_{med}|^{\frac32} \mp |\xi_2|^\frac32 \pm |\xi_3|^\frac32)   (\pm |\xi_0|^{\frac32} - |\xi_{med}|^{\frac32} \mp |\xi_2|^\frac32 \pm |\xi_3|^\frac32) } .
\end{split}
\]
Then we obtain the same symbol bound as before.

We continue with the terms in $W_{[3]}$, for which we have the following
symbols based on the arguments for inputs and for the outputs.
\[
L(WWW \to WWW) = 
\sum_{\pm}   \frac{ -  (\pm   |\xi_0|^{-\frac12} ) }
{ \pm |\xi_0|^{\frac32} \mp |\xi|_1^{\frac32} \mp |\xi_2|^\frac32 \pm |\xi_3|^\frac32} ,
\]
\[
L(WWW \to WQQ) = 
\sum_{\pm}   \frac{ (\pm  |\xi_0|^{-\frac12}) (\pm |\xi_2|^{-\frac12})(\pm  |\xi_3|^{-\frac12})
 }
{ \pm |\xi_0|^{\frac32} \mp |\xi|_1^{\frac32} \mp |\xi_2|^\frac32 \pm |\xi_3|^\frac32} ,
\]
\[
L(WQQ \to WQQ) = 
\sum_{\pm}   \frac{ -  (\pm  |\xi_0|^{-\frac12})  
 }
{ \pm |\xi_0|^{\frac32} \mp |\xi|_1^{\frac32} \mp |\xi_2|^\frac32 \pm |\xi_3|^\frac32} ,
\]
\[
L(WQQ \to QWQ ) = 
\sum_{\pm}   \frac{ -  (\pm  |\xi_0|^{-\frac12}) (\pm  |\xi_1|^{-\frac12}) (\mp  |\xi_2|^{\frac12}) }
{ \pm |\xi_0|^{\frac32} \mp |\xi|_1^{\frac32} \mp |\xi_2|^\frac32 \pm |\xi_3|^\frac32} ,
\]
\[
L(WQQ \to WWW ) = 
\sum_{\pm}   \frac{   (\pm  |\xi_0|^{-\frac12}) (\pm |\xi_2|^{-\frac12}) (\pm |\xi_3|^{-\frac12}) 
 }
{ \pm |\xi_0|^{\frac32} \mp |\xi|_1^{\frac32} \mp |\xi_2|^\frac32 \pm |\xi_3|^\frac32} .
\]
The size of these symbols is better than before, in spite of the bad factor 
$|\xi_0|^{-\frac12}$. Indeed, the antisymmetrization with respect to $\xi_0$ yields
a $|\xi_0|$ factor instead, and we obtain a very favorable bound
\[
|L| \lesssim |\xi_{hi}|^k,
\]
and a matching polyhomogeneous expression.
Thus the proof of the Proposition is concluded. 
\end{proof}


\section{The quasilinear modified  energy method  
at high frequencies} \label{s:high-en}

The main idea in this section is to use the normal form transformation
constructed in the previous section in order to produce a cubic energy
functional for our problem in $\dH^k$ for $k \geq 2$. The properties
of this energy functional are summarized in the following proposition:

\begin{proposition}\label{t:en=small}
  For any $n \geq 3$ there exists an energy functional $\Ent$ which
  has the following properties as long as $A \ll 1$:

(i) Norm equivalence:
\begin{equation*}
\Ent (\W,R)= (1+ O(A)) \Ez (\partial^{n-1} \W, \partial^{n-1} R),
\end{equation*}

(ii) Cubic energy estimates:
\begin{equation*}
\frac{d}{dt} \Ent (\W,R)  \lesssim_A AB \Ent (\W,R).
\end{equation*}
\end{proposition}

The conclusion of Theorem~\ref{t:cubic} is obtained in a standard fashion
from this proposition via Gronwall's inequality and Sobolev embeddings.
Since $A$ and $B$ are also controlled by $\|(W,Q)\|_{X}$, it also serves to 
prove that the bootstrap assumption \eqref{point-boot} implies the high frequency
part of the $(W,Q)$ bound in \eqref{energy}, namely 
\begin{equation}\label{energy-hi-proof}
\|(W,Q)\|_{\dH^{10}} \lesssim \epsilon t^{C \epsilon^2}.
\end{equation}

\begin{proof}
  The proof uses the {\em modified energy method}  introduced in \cite{HITW},
  \cite{HIT}, which is in some sense a quasilinear improvement of
  Shatah's normal form method \cite{MR803256}. The analysis begins
  with following two remarks:

\begin{itemize}
\item On one hand, one can construct a quasilinear energy functional
  $E^{n}_{nl}$ for our problem, which is equivalent to the $\dH^n$
  energy of $(W,Q)$, and for which we have straightforward energy
  estimates. However, these estimates are quadratic, in the sense that
\[
\frac{d}{dt} E^{n}_{nl} \lesssim_A B \|(W,Q)\|_{\dH^n}^2,
\]

\item  On the other hand, using the normal form variables $(\tW,\tQ)$, one can define 
a normal form energy $\Ennfz $, given by
\[
\Ennfz = \int \left(|\tW^{(n)}|^2
+ \Im [ \tQ^{(n)} \bar\tQ^{(n)}_{\alpha} ]\right)\, d\alpha.
\]
This satisfies cubic estimates, but they involve higher derivatives, 
\[
\frac{d}{dt}  \Ennfz \lesssim_A AB \|(W,Q)\|_{\dH^n \cap \dH^{n+3}},
\]
which do not close. Further, the normal form energy is not equivalent to the $\dH^n$
  energy of $(W,Q)$, and also contains undifferentiated $W$ and $Q$'s. 
\end{itemize}
 
Our goal is to produce an energy functional $\Ent$ which mixes the
best features of both these energies. This means that its cubic
expansion should match the cubic expansion of $\Ennfz$, while its high
frequency part must match the quasilinear energy.
To achieve our goal we start with the normal form energy
\[
\Ennfz = \int \left(|\tW^{(n)}|^2
+ \Im [ \tQ^{(n)} \bar\tQ^{(n)}_{\alpha} ]\right)\, d\alpha.
\]
Since $(\tW,\tQ)$ are normal form variables, this functional satisfies
an energy equation of the form
\begin{equation}
\frac{d}{dt} \Ennfz = quartic + higher,
\label{cubic_energy_estimate}
\end{equation}
but it has several defects:
\begin{enumerate}
\item It is expressed in terms
of $Q^{(n)}$ rather than the natural variable $R^{(n-1)}$.
\item It is not equivalent to the linear energy $\Elind(\W^{(n-1)},R^{(n-1)})$.
\item Its energy estimate has a loss of derivatives.
\end{enumerate}
However, the last two issues concerning $\Ennfz$ arise only at the
level of quartic and higher order terms. This motivates our strategy,
which is to modify $\Ennfz$ by quartic and higher terms to obtain a
``good'' energy $\Ent$ without spoiling the cubic energy estimate
(\ref{cubic_energy_estimate}). We carry out this procedure in two steps: 

\begin{itemize}
\item[(i)] We compute the cubic expansion of the normal form energy,
separating a principal part, which only involves differentiation,
and a lower order part, which uses bilinear multipliers.
\item[(ii)] We construct a modified normal form energy $\Ennf$ that depends
on $(\W^{(n-1)}, R^{(n-1)})$ and is equivalent to the linearized energy
$\Elind(\W^{(n-1)},R^{(n-1)})$; this addresses the issues (1) and (2) above, but not (3);
\item[(iii)] We consider the leading order part $\Ennfhigh$ and modify
that in a quasilinear fashion, in order to address the issue (3) above.
This modification is in part inspired from the analysis of the linearized
equation, and is needed due to the quasilinear nature of
our problem. Thus, we obtain an energy $\Ent$ with good cubic
estimates.
\end{itemize}

\begin{remark} The construction of the quasilinear modified energy in
  the current case turns out to be more involved than in the two
  dimensional gravitational non-surface tension water wave equations,
  see \cite{HIT}.
  \end{remark}

\bigskip

 {\bf The expansion of the normal form energy $\Ennfz$.}
We compute the leading part of $\Ennfz$ and express it in terms of
  $(\W,R)$ and their derivatives; for this we retain the quadratic and
  cubic expressions that appear when expanding each of the terms
  $\Vert \partial^n \tW\Vert^2_{L^2}$, $\Vert \partial^n
  \tQ\Vert^2_{\dot{H}^{\frac{1}{2}}}$ respectively.

We begin with the quadratic part of  $\Ennfz$, expressed in terms of 
$(W,Q)$, which is the integral 
\[
I_1 = \int |W^{(n)}|^2 + \Im ( \bar Q^{(n)}  Q^{(n-1)})\, d\alpha.
\]  
  Next we look at  the cubic part of $\Vert \partial^n \tW\Vert^2_{L^2}$, which is 
\begin{equation}
\label{calc1}
\begin{aligned}
 \int_{\zeta=\xi+\eta}\bar{W}(\zeta)W(\xi)W(\eta)(\xi +\eta)^{2n}B^h(\xi, \eta) +\bar{W}(\xi)W(\zeta)\bar{W}(\eta)\xi ^{2n}B(\zeta, \eta) d\xi d\eta+ c.c.,
\end{aligned}
\end{equation}
where $c.c$ stands for the complex conjugate, and  is as follows
\begin{equation}
\label{calc2}
\begin{aligned}
\int_{\zeta=\xi+\eta}W(\zeta)\bar{W}(\xi)\bar{W}(\eta)(\xi +\eta)^{2n}\bar{B}^h(\xi, \eta) +W(\xi)\bar{W}(\zeta)W(\eta)\xi^{2n}\bar{B}(\zeta, \eta)\, d\xi d\eta.
\end{aligned}
\end{equation}
We add the first term in \eqref{calc1} with the last term in \eqref{calc2}, and obtain
\begin{equation}
\int_{\zeta=\xi+\eta}\bar{W}(\zeta)W(\xi)W(\eta) \mathcal{B}_{\bar WWW} (\xi, \eta)\, d\xi d\eta+\int_{\zeta=\xi+\eta} W(\zeta)\bar{W}(\xi)\bar{W}(\eta) \overline{\mathcal{B}}_{\bar WWW} (\xi, \eta) \, d\xi d\eta,
\end{equation}
where 
\begin{equation*}
\begin{aligned}
\mathcal{B}_{\bar WWW} (\xi, \eta):&= (\xi +\eta)^{2n}B^h(\xi, \eta) + \xi^{2n}\bar{B}(\zeta, \eta)\\
&=\frac{i(\xi+\eta)}{9(\xi +\eta)^2-4\xi\eta} \left[ -(\xi+\eta)^{(2n)}\left(\frac{9}{4}(\xi +\eta)^2-2\xi \eta \right)  \right. \\
& \left.   \qquad \qquad  \qquad \qquad  \qquad+ \xi^{2n}\left( \frac{9}{4}\xi^2 +\frac{19}{4}\xi \eta +3\eta^2\right)+\eta^2 \left( \frac94 \eta^2 +\frac{19}{4}\xi \eta +3\xi^2\right)  \right] .
\end{aligned}
\end{equation*}
Expanding the symbol ${\mathcal B}_{\bar WWW} $ we have 
\[
\mathcal B_{\bar WWW}(\xi,\eta) =  -i(\xi+\eta)^n \!\! \left[  (\frac{n}2 - \frac14) (\xi^{n} \eta +\xi \eta^n)  + \frac12 
(\frac59 n + \frac19) (\xi^{n-1} \eta^2  + \xi^2 \eta^{n-1})\right] \!+ \xi^3 \eta^3 \M^{(2,0)}_{2n-5}. 
\]
 Hence, the leading energy component generated by ${\mathcal B}_{\bar WWW} $ 
is
\[
\int -\left( n -\frac12 \right) \bar{W}^{(n)}W^{(n)}  W_{\alpha} -   (\frac59 n + \frac19)   
\bar{W}^{(n)}W^{(n-1)}  W_{\alpha\alpha}
 \, d\alpha.
\]
Adding the complex conjugate  gives the leading integral
\[
I_2 = \int -\left( 2n - 1 \right) \bar{W}^{(n)}W^{(n)}  \Re W_{\alpha} +     (\frac{10}9 n + \frac29)   
\Im(\bar{W}^{(n)}W^{(n-1)})  \Im W_{\alpha\alpha}
 \, d\alpha.
\]
The remainder has the form
\[
\tilde I_2 = \Re \int \bar \W L_{2n-6}(W_{\alpha \alpha \alpha}, W_{\alpha \alpha \alpha}) \, d\alpha .
\]

There are other types of cubic terms appearing in the energy functional $\Ennfz$
associated to the normal form equations (\ref{nft-ww}), and we organize
them in two categories:
 \begin{itemize}
 \item terms involving the factors $\bar{W}QQ$, and their complex conjugates;
   they are cubic expressions tied to the bilinear forms $C^h$ and
   $A$,
  \item terms involving the factors  $\bar{Q}WQ$, and their complex conjugates; they are cubic expressions corresponding to the bilinear forms $A^h$, $C$ and $D$.
 \end{itemize}
 Analogous to the previous computations we  compute the
 symbol of each of the expressions in discussion, and obtain:
 \[
 \begin{aligned}
{\mathcal B}_{\bar{W}QQ} = &\   (\xi+ \eta) ^{2n}C^h(\xi, \eta)+\xi ^{2n}\bar{A}(\zeta, \eta)
\\ = & \ \frac{-3i(\xi +\eta)}{9(\xi+ \eta)^2-4\xi\eta} \left[ \frac{1}{2}(\xi+\eta)^{2n+1}  -\xi^{2n-1} \left( \xi ^2+\frac32 \xi \eta+\frac32 \eta^2\right) \right],
 \end{aligned}
 \]
 and also 
 \[
 \begin{aligned}
{\mathcal B}_{\bar Q W Q}= & \  \xi^{2n}\bar C(\zeta, \eta)+\zeta^{2n-1} {A}^{h}(\xi, \eta)+\eta^{2n-1}\bar D(\zeta. \xi)
\\
= & \ \xi^{2n}\dfrac{ 3 i (\xi+\eta)^2}{9 (\xi+\eta)^2 - 4 \xi \eta } - i(\xi +\eta)^{2n-1} \frac{3\xi^3 - \frac32 \xi^2 \eta - 5 \xi \eta^2 -\frac92 \eta^3}{9(\xi+\eta)^2 - 4 \xi \eta}  \\ & \ 
\hspace{2cm}- i \eta^{2n-1} \dfrac{ (\xi+\eta)( \frac92 (\xi+\eta )^2 - 4 \xi\eta )}{9 \xi^2 +14 \xi \eta + 9\eta^2}.
\end{aligned}
 \]
For both symbols we compute the expansion
\[ 
{\mathcal B}_{\bar{W}QQ} =   \frac{i}{3} (n-\frac14) (\xi+\eta)^n (\xi^{n-1} \eta + \xi \eta^{n-1}) + 
\xi^2 \eta^2 \M^{(2,0)}_{2n-4}
,
\]
respectively
\[
{\mathcal B}_{\bar Q W Q} =   i (\xi+\eta)^{n} \left(-\frac13(2n-\frac72) \xi^{n-1} \eta  +  (n-1) \eta^{n-1} \xi
-(\frac{17}{18}n - \frac{11}{18}) \eta^{n-2} \xi^2 \right) +\eta^2 \xi^3 \M^{(2,0)}_{2n-5}.
\]
This gives the following leading energy components:
\[
I_3 = 2 \Re \int  - i \frac23 (n-\frac14) \bar W^{(n)} Q^{(n-1)} Q_\alpha \, d \alpha ,
\]
respectively 
\[
\begin{split}
I_4 = & 2 \Re \int i  \bar Q^{(n)} \left(-\frac13(2n-\frac72) W^{(n-1)} Q_\alpha  +  (n-1) Q^{(n-1)} W_\alpha
-(\frac{17}{18}n - \frac{11}{18}) Q^{(n-2)} W_{\alpha\alpha}\right) \, d\alpha 
\\ = & \  \int -2(n-1) \Re W_\alpha \Im (\bar Q^{(n)} Q^{(n-1)}) + \frac23(2n-\frac72) \Im( \bar W^{(n)} Q^{(n-1)}
\bar Q_\alpha) 
\\ & \ \ \ \ 
- (\frac{8}9 n - \frac{2}{9}) \Im W_{\alpha\alpha}|Q^{(n-1)}|^2 
  \, d\alpha ,
\end{split}
\]
together with remainders which have the form
\[
\tilde I_3 = \Re \int \bar \W L^{(2,0)}_{2n-5} (Q_{\alpha\alpha},Q_{\alpha \alpha})  \, d\alpha,
\]
respectively
\[
\tilde I_4 = \Re \int \bar Q_\alpha L^{(2,0)}_{2n-6} (W_{\alpha\alpha\alpha},Q_{\alpha \alpha})  \, d\alpha.
\]

Then the cubic leading expansion of the normal form energy is
\[
E^n_{NF,0} = I_1 + I_2 + I_3 + I_4 + \tilde I_2 +\tilde I_3 + \tilde I_4 + quartic.
\]
We remark here that, in the above expressions there are no more
undifferentiated $(W,Q)$'s, and also that the highest order terms are
at the level of the energy.

\bigskip

{\bf (ii) The modified normal form energy $\Ennf$.}
The goal here  is to replace $Q_\alpha$ by $R$ into the last expansion above. This is a direct substitution in the cubic expressions, where we simply take the sum 
\[
(I_2+ I_3 + I_4 + \tilde I_2 +\tilde I_3 + \tilde I_4)(R,\W).
\]
However, some care must be taken with the second term in  the quadratic part $I_1$. 
We have $Q_\alpha = R(1+\W)$ therefore
\[
Q^{(n)}=R^{(n-1)}(1+\W)+(n-1)R^{(n-2)}\W_{\alpha}+W^{(n)}R+...
\]
Substituting into the corresponding part of $I_1$ we obtain
\[
\begin{split}
\int  \Im (\bar Q^{(n)} Q^{(n-1)})\, d\alpha= & \ \int (1+2\Re \W)\Im (\bar R^{(n-1)}R^{(n-2)})-(2n-3)\bar R^{(n-2)}R^{(n-2)}\Im \W_{\alpha} \\ 
& \quad +2 \Im \left( R^{(n-2)}\bar W^{(n)}\bar R\right)\, d\alpha+ (\tilde I_3 + \tilde I_4)(R,\W) + quartic.
\end{split}
\]  
The last two terms combine with $(I_2+I_3)(\W,R)$. Summarizing
and reorganizing terms, the expressions 
$(I_1+I_2+I_3+I_4)(\W,R)$ are replaced by $J_1+J_2+J_3$ where
\begin{equation*}
\begin{aligned}
J_1 &=  \int (1- \left( 2n - 1 \right)\Re \W)
|\W^{(n-1)}|^2 + (1-2(n-2) \Re \W)\Im (\bar R^{(n-1)}R^{(n-2)}) 
\, d\alpha ,\\
J_2 &= \int  -\frac83 (n-\frac14)  \Re R   \, \Im \left( R^{(n-2)}\bar W^{(n)}\right) \, d\alpha ,\\
J_3 &=  \int  \Im \W_\alpha \left[(\frac{10}9 n + \frac29)   
\Im(\bar{\W}^{(n-1)}\W^{(n-2)})  - (\frac{26}9 n - \frac{29}{9}) 
R^{(n-2)} \bar R^{(n-2)}\right] \, d\alpha ,
\end{aligned}
\end{equation*}
and the lower order terms $I_2+I_3+I_4$ are replaced by $\tilde
J_1+\tilde J_2+\tilde J_3$, which have the form
\begin{equation*}
\begin{aligned}
\tilde J_1 = \Re \int \bar \W L_{2n-6}^{(2,0)}(\W_{\alpha \alpha},& \W_{\alpha \alpha}) \, d\alpha,\quad
 \tilde J_2 = \Re \int \bar \W L^{(2,0)}_{2n-5} (R_{\alpha},R_{ \alpha})  \, d\alpha,\\
 \tilde J_3 &=  \Re \int \bar R L^{(2,0)}_{2n-6} (\W_{\alpha\alpha},R_{\alpha})  \, d\alpha .
 \end{aligned}
\end{equation*}

Thus our modified normal form energy has the form
\[
\Ennf = J_1+J_2+J_3 + \tilde J_1+ \tilde J_2 + \tilde J_3;
\]
and, by construction, has the property that
\[
\Ennf = \Ennfz + quartic + higher .
\]
and thus
\[
\frac{d}{dt} \Ennf = quartic + higher .
\]

\bigskip
{\bf The quasilinear modified energy.}
As defined above, the remaining issue with the energy $\Ennf $ is that its time derivative
cannot be controlled at the same level due to the quasilinear character of the problem.
The last modification below addresses this issue. We observe that these quasilinear
modifications are only needed for the terms $J_1$, $J_2$ and $J_3$ above;
the lower order terms $\tilde J_1,\tilde J_2 ,\tilde J_3$ will be left unchanged.
The quasilinear corrections to $J_1$, $J_2$ and $J_3$ are defined as follows:
\begin{equation}
\begin{split}
E_1 = &  \int  J^{-(n-\frac12)} \left[|\W^{(n-1)}|^2 + (4n-3) \Im\left( \frac{\W_{\alpha}}{1+\W}\right)  \Im ( \bar \W^{(n-1)} \W^{(n-2)})\right]
\\ &  \qquad + J^{-(n-2)} \Im (\bar R^{(n-1)}  R^{(n-2)})  \, d\alpha .
\end{split}
\end{equation}
Here we have added a lower order term in the square bracket, in order
to obtain a bounded time derivative. This is the analogue of the lower
order term arising in the operator $L$ in the study of the linearized problem later on. 
To compensate for that, a matching term is subtracted from $J_3$.
The remaining corrections are
\begin{equation}
E_2 = \int  -\frac83 (n-\frac14)  \Re R   \, \Im \left( (1+\W)^2 R^{(n-2)}\bar \W^{(n-1)}\right)\, d\alpha , \hspace{1in}
\end{equation}
\begin{equation}
E_3 = \int  -(\frac{26}9 n - \frac{29}{9})  \Im \left( \frac{\W_\alpha}{1+\W}\right) \left(  J^{-(n-\frac12)} \Im(\bar \W^{(n-1)} \W^{(n-2)}) + J^{-(n-2) }|R^{(n-2)}|^2\right)  \, 
d\alpha. 
\end{equation}
Our final energy is 
\begin{equation}
\Ent = E_1+ E_2 + E_3 + \tilde J_1 + \tilde J_2 + \tilde J_3.
\end{equation}
We further note that all the terms in $\Ent$ share the same scaling law.

\bigskip

We now proceed to prove that $\Ent$ has the properties in the proposition.
The energy equivalence property 
\[
\Ent = (1+O(A)) \| (\W,R)\|_{\dH^{n-1}}^2 =  (1+O(A)) \| (W,Q)\|_{\dH^{n}}^2
\]
reduces to estimating the cubic and higher terms in $\Ent$ by $O(A))
\| (\W,R)\|_{\dH^{n-1}}^2$. All these terms have the same scaling, and in effect can be viewed 
as multilinear expressions in $\W, R$ and $Y = \frac{\W}{1+\W}$ and their conjugates.
Then the desired conclusion follows easily by Proposition~\ref{bi-est}
interpolation and H\"older's inequality (see also the similar argument
in \cite{HIT}).

It remains to consider the time derivative of $\Ent$.
By inspection we see that 
\[
\Ent = \Ennf + quartic + higher = \Ennfz + quartic + higher,
\]
therefore we have 
\[
\frac{d}{dt} \Ent =  quartic + higher.
\]
A good way to rephrase this is by introducing the truncation operator
$\Lambda^{\geq 4}$ which selects the $quartic + higher$ part of an expression
in $(\W,R)$. Thus the above can be rewritten as
\[
\frac{d}{dt} \Ent =  \Lambda^{\geq 4} \frac{d}{dt} \Ent = 
\Lambda^{\geq 4}  \frac{d}{dt} E_1 + \Lambda^{\geq 4}  \frac{d}{dt} E_2 +  \Lambda^{\geq 4}  \frac{d}{dt} E_2
+  \Lambda^{\geq 4}  \frac{d}{dt} \tilde J_1 +  \Lambda^{\geq 4}  \frac{d}{dt} \tilde J_2 +  \Lambda^{\geq 4}  \frac{d}{dt} \tilde J_3.
\]
Now we can separately estimate each term in the above sum. Again, they
all share the same scaling law. The expressions arising here are fully 
nonlinear; however, one can view them as multilinear forms in $\W$, $R$ and $Y$
(and their conjugates).

For the purpose of this computation we denote by $\err$ any
multilinear form in $\W$, $R$ and $Y$ and their complex conjugates, with the
same scaling as above, which

\begin{itemize}
\item involves only differentiation, multiplication and at most one
  $L_0$ bilinear form before the final integration in $\alpha$.

\item does not contain a bilinear expression involving more derivatives
  than the energy $\Ent$, i.e., more that $\bar \W^{(n-1)}
  \W^{(n-1)}$, $ \bar R^{(n-1)} R^{(n-2)}$, respectively $\bar
  \W^{(n-1)} R^{(n-2)}$.

\item does not contain any undifferentiated $R$ as the lowest frequency factor.
\end{itemize}
A quick parity analysis reveals that such expressions contain either

\begin{enumerate}
\item exactly one $R$ factor and a total number of $2n-1$ derivatives, or
\item exactly three $R$ factors and a total number of $2n-2$ derivatives.
\end{enumerate}

All such multilinear forms  can all be estimated by Proposition~\ref{bi-est}, Sobolev embeddings,
interpolation and H\"older's inequality in a scale invariant fashion
as follows:
\[
|\Lambda^{\geq 4} \err| \lesssim_A  AB \|(\W,R)\|_{\dH^{n-1}}^2.
\]
For more details we refer the reader to the similar calculations in \cite{HIT}, and in particular
to Lemma 2.4 in appendix B of \cite{HIT}.

In particular we directly have
\[
\frac{d}{dt} \tilde J_1+\tilde J_2+\tilde J_3 = \err ,
\] 
which allows us to dispense with these terms. It remains to carefully
compute the time derivative of each of $E_1$, $E_2$ and $E_3$, in
order to insure that we obtain only $\err$ terms.  \medskip

{\bf a) The time derivative of $E_1$.}
Differentiating any of the coefficients of the leading order terms yields $\err$ type expressions,
so we can write
\[ 
\frac{d}{dt} E_1 = Er_1 + Er_2 + \err,
\]
where
\[
Er_1 = \int \! 2  J^{-(n-\frac12)} \Re (\bar \W^{(n-1)} \W^{(n-1)}_t) + 2
J^{-(n-2)} \Im (\bar R^{(n-1)} R^{(n-2)}_t) + \partial_\alpha
J^{2-n}\Im ( \bar R^{(n-2)} R^{(n-2)}_t) \, d\alpha,
\]
and 
\[
Er_2 = \int 2  (4n-3) \Im\left( \frac{\W_{\alpha}}{1+\W}\right)  \Im ( \bar \W^{(n-1)} \W^{(n-2)}_t)\, d\alpha.
\]
To compute these two expressions we begin with differentiating in the equations 
\eqref{diff-ngst}:
\begin{equation*}
\left\{
\begin{aligned}
 \W^{(n-1)}_{t}= & - b\W^{(n)}- \frac{1+\W}{1+\bar{\W}} R^{(n)}  - n R^{(n-1)}\left(
\frac{\W_\alpha}{1+\bar \W}   - 
 \frac{(1+\W)\bar \W_{\alpha}}{(1+\bar \W)^2}\right) + ...  \\
 R^{(n-2)}_t= &-b R^{(n-1)}   -i \W^{(n)}  \frac{1}{J^\frac12(1+\bar \W)^2} + \frac{i}2  \W^{(n-1)} 
\left(  \frac{(5n-4)\W_\alpha}{J^\frac12(1+\W)^3}  +  \frac{(n-2) \bar \W_\alpha}{J^\frac32(1+\W)} \right)+...
\end{aligned}
\right.
\end{equation*}
where on the right we have kept only the terms which may contribute to
the leading part of the time derivative of $E_1$. Thus integrating by
parts and discarding lower order $\err$ type contributions we obtain
\[
\begin{split} 
Er_1 \approx  & \ 2  \Re \int - J^{-(n-\frac12)} \frac{1+\W}{1+\bar{\W}} R^{(n)} \bar W^{(n)} - 
 J^{-(n-2)}   \frac{1}{J^{1/2}(1+\bar \W)^2} \bar R^{(n-1)} W^{(n+1)}    
\\
& +    \bar W^{(n)} R^{(n-1)}\left[-n J^{-(n-\frac12)}\left( \frac{\W_\alpha}{1+\bar{\W}}   -
 \frac{(1+\W)\bar \W_{\alpha}}{(1+\bar \W)^2}\right)\right. \\ & \left.+ \frac12 J^{-(n-2)}\left(  \frac{ (5n-4)\bar \W_\alpha}{J^\frac12(1+\bar\W)^3}  +\frac{ (n-2)  \W_\alpha}{J^\frac32(1+\bar \W)}\right) +   \partial_\alpha J^{2-n} \frac{1}{J^{1/2}(1+\bar \W)^2} \right] \, d\alpha
\\
 \approx & \ \ 2  \Re \int - J^{-(n-\frac32)} \frac{1}{(1+\bar{\W})^2} (R^{(n)} \bar W^{(n)} +  R^{(n-1)} \bar W^{(n+1)})   
\\ 
& +    J^{-(n-\frac32)} \bar W^{(n)} R^{(n-1)}
\left[ -n\frac{\W_\alpha}{J(1+\bar \W)} + (3n-1) \frac{\bar{\W}_\alpha}{(1+\bar \W)^3} 
 \right]  \, d\alpha
\\
=  &  \ 2  (2n-\frac32) \Re \int J^{-(n-\frac32)}\bar W^{(n)} R^{(n-1)} \left(- \frac{\W_\alpha}{J(1+\bar \W)} +
  \frac{\bar \W_\alpha}{(1+\bar \W)^3}\right)\, d \alpha
\\
= &  \ 2  (4n-3)  \int J^{-(n-\frac32)} 
\Im \left(\bar W^{(n)} R^{(n-1)} \frac{1}{(1+\bar \W)^2}\right) \Im\left( \frac{\W_{\alpha}}{1+\W}\right)\,
d \alpha .
\end{split}
\]
A shorter computation gives 
\[
Er_2 = - \int  2   (4n-3)J^{-(n-\frac32)} 
\Im \left(\bar W^{(n)} R^{(n-1)} \frac{1}{(1+\bar \W)^2}\right) \Im\left( \frac{\W_{\alpha}}{1+\W}\right) 
\, d\alpha +\err,
\]
which cancels $Er_1$ and leads to the desired bound
\[
\frac{d}{dt} E_1 = \err.
\]

\medskip

{\bf b) The time derivative of $E_3$.} Again, if the time derivative applies to the $\Re R$ 
coefficient, to $J$,  or to $\W$, then we get lower order contributions.  It remains to 
consider the two terms when the time derivative applies to either $R^{(n-2)}$ or to $\W^{(n-1)}$.
In both cases we need to consider only the leading order contributions,
\[
 \W^{(n-1)}_{t}=  - \frac{1+\W}{1+\bar{\W}} R^{(n)} + ...  \qquad 
 R^{(n-2)}_t=  - i \W^{(n)}  \frac{1}{J^\frac12(1+\bar \W)^2} +...
\]
Then the two terms exactly cancel and we obtain
\[
\frac{d}{dt} E_3 = \err.
\]

{\bf b) The time derivative of $E_2$.} Again, if the time derivative falls on $\Im \W_\alpha$ 
coefficient, or on $J$,  then we get lower order contributions.  Then
we can write
\[
\frac{d}{dt} E_2 = Er_3 + Er_4 + \err,
\]
where 
\[
Er_3 = 2  \int \Im \W_\alpha  J^{-(n-\frac12)} \Im(\bar W^{(n)} W^{(n-1)}_t) \, d\alpha,
\]
and 
\[
Er_4 = 2 \int \Im \W_\alpha     J^{-(n-2) }\Re(\bar R^{(n-2)} R^{(n-2)}_t ) \,
d\alpha.
\]
In both cases we need to consider only the leading order contributions of the time 
derivatives, which exactly cancel
and we also obtain
\[
\frac{d}{dt} E_2 = \err.
\]
The proof of the proposition is concluded.

\end{proof}


 \section{Energy estimates at low frequency}
\label{s:low-en}

 Here we prove the $\dH^\sigma$ energy estimates for $(W,Q)$. For this we 
work directly with the equation \eqref{ngst}, which we no longer
 need to treat  as a nonlinear problem. Instead we can
 apply directly the normal form transformation to eliminate the quadratic terms,
and then do semilinear analysis on the remaining evolution, taking advantage
of the high frequency bounds in the previous section. 

Unfortunately the normal form transformation \eqref{nft} does not suffice for our 
purposes here. This is largely due to the weaker pointwise control that we 
have on the function $W$, but also due to the fact that we are estimating $Q$ 
in a negative Sobolev space. For these reasons, we will have to remove 
certain cubic non-resonant interactions
 via a further {\em cubic} modification to the normal form.
This motivates the following 

\begin{proposition}\label{p:low-en}
There exist normal form variables $( \ttW,\ttQ)$, obtained by a further cubic 
modification of the normal form \eqref{nft},
\begin{equation}
\ttW = W + W_{[2]} + W_{[3]}, \qquad \ttQ = Q + Q_{[2]} + Q_{[3]},
\end{equation}
where $(W_{[3]},Q_{[3]}) \in  (\B^3_{0}(W,Q), \C^{3}_{0}(W,Q))$,
 with the following properties:  

i) The $\dH^\sigma$  energies are equivalent,
\begin{equation}\label{tt-en}
\| ( \ttW,\ttQ)\|_{\dH^\sigma} \approx_A \|(W,Q)\|_{\dH^\sigma},
\end{equation}

(ii) The variables $( \ttW,\ttQ)$ solve a cubic equation:
\begin{equation}\label{ww-nft-tt}
\left\{
\begin{aligned}
&  \ttW_t + \ttQ_\alpha =  \ \ttG
\\
& \ttQ_t +i \ttW_{\alpha\alpha} = \ \ttK,
\end{aligned} 
\right.
\end{equation} 
where the following estimates hold:
\begin{equation}
\label{tt-gk}
\| ( \ttG,\ttK)\|_{\dH^\sigma} \lesssim   \|(W,Q)\|_{X}^2    \|(W,Q)\|_{\dH^2 \cap \dH^\sigma},
\end{equation}
\begin{equation}
\label{tt-gk4}
\| ( \ttG^{(4+)},\ttK^{(4+)})\|_{\dH^\frac12} \lesssim t^{-\frac12+2\sigma}  \|(W,Q)\|_{X}^2    \|(W,Q)\|_{\dH^2 \cap \dH^\sigma}.
\end{equation}

\end{proposition}
We remark that the last bound \eqref{tt-gk4} is not used in the proof of the energy estimates.
However, it will be very useful in the last section, where we obtain the asymptotic equation.
Its role there is to show that the quartic and higher terms do not affect at all the asymptotic equation.
\medskip

Before proving the proposition, we show how to use it to conclude the
proof of the $\dH^\sigma$ part of the energy bounds for $(W,Q)$ in
\eqref{energy} in our boootstrap argument. Precisely, our goal here is
to prove the bound
\begin{equation}\label{low-sigma}
\| (W,Q) \|_{\dH^\sigma} \lesssim \epsilon t^{C \epsilon^2} 
\end{equation}
under the initial data assumption \eqref{data} 
\begin{equation}\label{data-l}
\| (W,Q)(0) \|_{\dH^\sigma \cap \dH^6} \lesssim \epsilon ,
\end{equation}
and the bootstrap assumptions \eqref{energy-boot}, \eqref{point-boot}
and \eqref{pointw-boot}.  
Linear energy estimates for $( \ttW,\ttQ)$ show that
\[
\frac{d}{dt} \| ( \ttW,\ttQ)\|_{\dH^\sigma} \lesssim  
 \|(W,Q)\|_{X}^2    \|(W,Q)\|_{\dH^2 \cap \dH^\sigma}.
\]
For the $\dH^2$ norm above we use the high frequency bound \eqref{energy-hi-proof}.
Using also the energy equivalence  \eqref{tt-en}, we obtain
\[
\frac{d}{dt} \| ( \ttW,\ttQ)\|_{\dH^\sigma} \lesssim 
 \epsilon^2 t^{-1} ( \epsilon t^{C \epsilon^2} +  \| ( \ttW,\ttQ)\|_{\dH^\sigma}).
\]
Hence by Gronwall's inequality we obtain
\[
\| (\ttW,\ttQ) \|_{\dH^\sigma} \lesssim \epsilon t^{C \epsilon^2} .
\]
Now we can return to \eqref{low-sigma} via the energy equivalence \eqref{tt-en}. 
The rest of the section is devoted to the proof of the proposition.

\begin{proof}[Proof of Proposition~\ref{p:low-en}]
 We begin with some heuristic considerations. 
The straightforward  approach to this proof would be to work 
directly with the normal form variables defined by 
our quadratic normal form transformation,
\[
\tW = W+W_{[2]}, \qquad \tQ = Q+Q_{[2]},
\]
which solve cubic equations,
\begin{equation}\label{nft-ww-re}
\left\{
\begin{aligned}
 \tW_t + \tQ_\alpha =  \ \tG(W,Q)
\\
 \tQ_t +i \tW_{\alpha\alpha} = \ \tK(W,Q),
\end{aligned} 
\right.
\end{equation}
 and then estimate the cubic terms $(\tG,\tK)$ as in the proposition, by combining 
one $L^2$ bound with two $L^\infty$ bounds. However, this naive analysis fails for two reasons:
\begin{itemize}
\item[(i)] The undifferentiated $W$ appears in the normal form. This is a problem because
we  have neither an $L^2$ bound for $W$ at low frequencies, nor an $L^\infty$ bound for $W$
with $t^{-\frac12}$ decay.

\item[(ii)] The expression $\tK$ needs to be estimated in a negative Sobolev space $\dot H^{\sigma-\frac12}$, which is a problem at low frequencies, as it does not follow directly from a
trilinear superposition of  one $L^2$ bound and two $L^\infty$ bounds.
\end{itemize}

Based on these considerations, our strategy will be to separate the 
cubic terms in $(\tG,\tK)$ into a good portion and a bad portion so that 

\begin{itemize}
\item[(a)] The good portion of $(\tG,\tK)$ can be estimated directly using  one $L^2$ bound 
and two $L^\infty$ bounds. 
\item[(b)] The bad portion of $(\tG,\tK)$  is non-resonant, and thus can be eliminated with a further 
cubic normal form correction.
\end{itemize} 

When implementing  the above strategy we will keep in mind several principles:

\begin{itemize}
\item[I.]  The terms with undifferentiated $W$ in both equations in
  \eqref{nft-ww-re} are bad when $W$ has the low frequency, and need
  to be renormalized.
\item[II.] The undifferentiated terms in $\tK$ are bad at low
  frequency, and need to be renormalized when the three input
  frequencies are larger.
\item[III.] The linearization of the renormalized equation will later be used
  as the renormalization of the linearized equation, so we need to make
  sure that our estimates still apply there. The difficulty is that
  for the normalized variables $(w,q)$ we only have access to the
  $L^2$ norm, and not to any decaying $L^\infty$ bound. Thus, in all
  multilinear estimates we have to allow for the case when any one of
  the terms admits only $L^2$ type bounds.
\end{itemize}

We begin with the normal form equations \eqref{nft-ww} for
$(\tW,\tQ)$, with $(\tG,\tK)$ as in \eqref{gkt}.  For the expressions
$\tG$ and $\tK$ we need to separate the worst terms at low frequency,
namely the ones which contain an undifferentiated $W$ or $Q$
factor. This is done using the expansion \eqref{gkt-nodiff}.
Based on this, we first express $\tG$ in the form
\[
\tG = \tG_0 + \tG_1+\tG_2, \qquad 
\]
where $\tG_0$ and $\tG_1$  contain cubic terms with a $W$ factor, respectively without, 
\[
\tG_0= -\frac12 \Re W  G^{(2)}_{\alpha}, \qquad \tG_1 =
L_0(W_\alpha,W_\alpha,Q_\alpha) + L_0(Q,Q_\alpha,Q_\alpha) +
L_0(Q,W_\alpha,\W_{\alpha\alpha} ),
\]
while $\tG_2$ contains terms which are quartic and higher. We further
note that $\tG_2$ may contain a single undifferentiated factor.

For $\tK$ we need a better bound at low frequency, so we split it further as
\[
\tK = \tK_0 + \tK_1+ \tK_2  + \tK_3,
\]
where $\tK_0$ and $\tK_1$ contain all cubic terms with an undifferentiated factor,
\[
 \tK_0= -\frac12 \Re W  K^{(2)}_{\alpha}, \qquad 
\tK_1 = - \frac13 \Re Q  G^{(2)}_\alpha.
\]
$\tK_2$ contains the remaining cubic terms with all factors differentiated, 
\[
\tK_2 = L_0(W_\alpha,W_\alpha,W_{\alpha \alpha}) + L_0(W_\alpha,Q_\alpha,Q_\alpha),
\]
and $\tK_3$ contains all the the quartic terms. We observe that $\tK_3$ may contain 
at most one undifferentiated $W$ factor.

\bigskip {\bf The terms $\tG_0$. $\tK_0$.} These are the worst cubic
terms, since we have less control on $W$ than on everything else.
However, the co-factors $G^{(2)}_\alpha$ and $K^{(2)}_\alpha$ are
differentiated, and also each of the factors in $G^{(2)}$ and $K^{(2)}$ are
differentiated. This allows us to separate a single unfavorable
case, namely when the $W$ factor has frequency less than the
output frequency, and than both of the entries in $G^{(2)}$, respectively
$K^{(2)}$. Precisely, we decompose
\[
\tG_0 = \tG_0^{good}  +\tG_0^{lhh \to h}, \qquad \tK_0 = \tK_0^{good}  +\tK_0^{lhh \to h},
\]
where 
\[
\tG_0^{lhh \to h} = \sum_{\lambda} P_{\gg \lambda} \left[  W_\lambda 
G^{(2)}(W_{\gg \lambda},Q_{\gg \lambda}) \right], \qquad 
\tK_0^{lhh \to h} = \sum_{\lambda } P_{\gg \lambda} \left[  W_\lambda 
K^{(2)}(W_{\gg \lambda},Q_{\gg \lambda}) \right].
\]
We can bound the terms $( \tG_0^{good},  \tK_0^{good})$ as in \eqref{tt-gk} as follows.
If the $W$ frequency is larger than the output frequency, then we can move 
half a derivative from the $G^{(2)}_\alpha$, respectively  $K^{(2)}_\alpha$ onto $W$, obtaining 
terms of the form 
\[
(D^\frac12 L_0^{lhh \to h} ( D^{\frac12} W, G^{(2)}) ,  D^\frac12 L_0^{lhh \to h} ( D^{\frac12} W, K^{(2)})).
\]
The half derivative in front is useful at low frequency in the second term, as it 
allows us to obtain the $\dot H^{\sigma-\frac12}$ bound (at low frequency). 
Now we are left with a trilinear form where all three factors are at least half differentiated.
Thus we can use an $L^2$ bound on any one of them, and an $L^\infty$ bound on the other
two. 

On the other hand, if one of the inputs of $(G^{(2)},K^{(2)})$ has frequency lower or comparable to 
$W$, then we can move  half of its derivative to $W$, and we obtain an expression of the form
\[
(\partial_\alpha L_0^{lhh \to h} ( D^{\frac12} W, G^{(2)}((W,Q), D^\frac12(W,Q)) ,  
D^\frac12 L_0^{lhh \to h} ( \partial_\alpha W, K^{(2)}((W,Q), D^\frac12(W,Q)))).
\] 
Here all three factors inside are again at least half differentiated, so we can  bound 
  any one of them in $L^2$, and the remaining two in $L^\infty$. 

It remains to consider the bad terms $(\tG_0^{lhh \to h},\tG_0^{lhh \to h})$; these are 
non-resonant, so we can factor them out with a cubic normal form correction. The simplest
way to achieve this is to  take advantage of the fact that $G^{(2)}$ and $K^{(2)}$ can be corrected 
with the normal form transform computed in Proposition~\ref{p:nf}. This leads us to define 
the first round of cubic normal form corrections as $(W^{lhh \to h}_{[3],0},Q^{lhh \to h}_{[3],0})$,
which are defined by adding frequency localizations as above to the trilinear expressions
\[
W_{[3],0} = -\frac12 \Re W W_{[2],\alpha}, \qquad 
Q_{[3],0} = -\frac12 \Re W Q_{[2],\alpha}.
 \]
 We have 
\[
\begin{split}
\partial_t W_{[3],0}  + \partial_\alpha Q_{[3],0}   - \tG_{0}  = & \ -\frac12  W_{[2],\alpha} \Re W_t
 - \frac12 \Re W_\alpha Q_{[2],\alpha}
+ \frac12 \Re W  (G^{(2)} - W_{[2],t} - Q_{[2],\alpha} )_\alpha 
\\
= & \  \frac12 \Re Q_\alpha  W_{[2],\alpha}   - \frac12 \Re W_\alpha Q_{[2],\alpha}
 -  \frac12 \Re G  W_{[2],\alpha} - 
\frac12 \Re W (\tG-G^{(3+)})_\alpha .
\end{split}
\]
The first two terms are cubic, and even contain an undifferentiated $W$; 
however, after the $lhh\to h$ localization we are left with a differentiated low frequency
factor, so these terms are handled as discussed above, by moving a half derivative
from the first, low frequency, factor to the undifferentiated factor in $W_{[2]}$ or 
$G_{[2]}$.   The remaining terms are quartic, and may even contain two
undifferentiated $W$ factors. But this is still acceptable, as we can
bound them in $H^1$.  Indeed, since $\sigma$ is small therefore we can
place the undifferentiated factors in $L^4$ but still with some time
decay, and bound each of the two differentiated factors in $L^\infty$
by $t^{-\frac12}$.

Similarly we have 
\[
\begin{split}
\frac{d}{dt} Q_{[3],0} + i  \partial_\alpha^2 W_{[3],0}  -  \tK^{lh}_{0}  =   \frac12 \Re Q_\alpha  Q_{[2],\alpha}  
- i  \frac12 \Re W_{\alpha\alpha} W_{[2],\alpha}  
-  \frac12 G Q_{[2],\alpha} -  \frac12 \Re W (\tK-K^{(3+)})_\alpha .
\end{split}
\]
The first two terms on the right are again cubic, and further, by the
$lhh\to h$ frequency localization we can both move the derivative on
$Q_{[2],\alpha} $, respectively $ W_{[2],\alpha} $ onto the output,
and still be left with a derivative onto the first, low frequency,
factor.  Thus we can bound them as above.  The remaining terms are
quartic, and may even contain two undifferentiated $W$ factors. But
this is still acceptable, as we can bound them in $L^2 \cap \dot
H^{-1}$.  To get the output one derivative lower than before, we again
take advantage of the fact that the $lhh\to h$ frequency localization
allows us to move the derivative in the second factor onto the
output. Hence it remains to estimate the remaining quadrilinear form
in $H^1$, which is done as before.

\bigskip 

{\bf The term $\tG_1$.}
 This is  simpler, as we can place either factor
in $L^2$, and the remaining factors are directly bounded in $L^\infty$
by $t^{-\frac12}  \| (W,Q)\|_{X}$. Again, we bound these terms in the inhomogeneous
Sobolev space $H^1$.

\bigskip 

{\bf The term $\tG_2$.}  This is exactly as before. The three differentiated factors 
are bounded in either $L^2$ or $L^\infty$, while the extra fourth factor provides 
some additional time decay, even if it is only a $W$ factor. 
\bigskip

{\bf The term $\tK_1$.} Just as $\tG_1$, this can be estimated in $L^2$, which takes care
of the high frequencies ($\geq 1$). For low frequencies,  we need to get a better bound 
than $L^2$. For this we commute the co-factor derivative
as 
\[
\begin{split}
\tK_1 = & \ \frac12 \Re Q_\alpha K^{(2)}  - \frac12 \partial_\alpha\left[\Re Q K^{(2)}\right]  .
\end{split}
\]
The first term is moved into $\tK_2$, while the second is an exact derivative 
applied to a trilinear expression which can be placed in $L^2$ just as $\tG_1$.

\bigskip

{\bf  The term $\tK_2$}. This can be estimated directly in $L^2$, so it is 
 good at high frequency ($\geq 1$).  It remains to consider the  low frequencies.
For this we recall that $\tK_2$ has all differentiated factors, more precisely
\[
\tK_2= \tK_2(W_\alpha,Q_\alpha) \in \B^3_{-1} (W_\alpha,Q_\alpha).
\]
We split this into two parts, depending on whether the output frequency is smaller
than all of the inputs,
\[
 \tK_2 = \tK_{2}^{good} +  \tK_{2}^{hhh \to l},
\]
where 
\[
 \tK^{hhh\to l}_{2} = \sum_{\lambda < 1} P_{\ll \lambda} \tK_2(W_{>\lambda,\alpha},Q_{> \lambda,\alpha}).
\]
In $\tK_{2}^{good}$ one of the factors has frequency equal or lower than
the output so we can move a half derivative onto the output, and
estimate the rest in $L^2$ using one $L^2$ and two $L^\infty$ bounds.

It remains to consider the only nontrivial term $\tK^{hhh\to l}_{2}$.
This cannot be estimated directly in $\dot H^{\sigma-\frac12}$ at low
frequency, as the only way to gain improved summability is by trading
off decay.  However, it has the redeeming feature that the balance of
the frequencies is so that the cubic interaction is non-resonant. This
analysis was carried out in Proposition~\ref{p:nf-cubic}, which shows
that we can find a further cubic normal form correction
$(W^{hhh \to l}_{[3],1},\tQ^{hhh \to l}_{[3],1}) \in (\B^3_0(W,Q), \C^3_0(W,Q))$,  which
removes these terms.
 These are trilinear forms with the same frequency localization as $\tK^{hhh \to l}_{2}$.
More precisely, we have
\begin{equation}
\left\{
\begin{aligned}
 W^{hhh \to l}_{[3],1} = & \ L^{hhh\to l}_1(Q,Q,W) + L^{hhh\to l}_2(W,W,W)
 \\
 Q^{hhh \to l}_{[3],1} = & \  L^{hhh\to l}_2(Q,W,W)  + L^{hhh\to l}_1(Q,Q,Q).
\end{aligned} 
\right.
\end{equation}
Introducing  these cubic terms removes the cubic error $\tK^{hhh \to l}_{2}$, but
yields further quartic errors. However, these contain at most one $W$ factor, 
and thus can be included in $\tK_3$.

\bigskip

{\bf The term $\tK_3$.} This can be easily estimated in $L^2$ using
Proposition~\ref{tri-est}, which suffices at high frequency $(\geq 1)$. Hence
it remains to consider the low frequencies.  The worst terms are those
with one undifferentiated $W$, e.g., of the form
\[
L_0(W, Q_\alpha, Q_\alpha, W_\alpha).
\]
where $L_0$ is a composition of a trilinear form as in \eqref{tri-est} with 
a straight multiplication.
But for such a term we can use twice the Sobolev embedding to bound it\footnote{
As estimated here, this bound is borderline; however we can
  gain an additional  small time decay factor by restricting the frequencies in $hhh\to
  l$ interactions to $\lambda_{med} \geq \sqrt{\lambda_{out}
    \lambda_{hi}}$.}, using Proposition~\ref{bi-est}, e.g., as in  
\[
\begin{split}
\|L_0(W, Q_\alpha, Q_\alpha, W_\alpha)\|_{\dot H^{\sigma -\frac12}} \lesssim & \
\|L_0(W, Q_\alpha, Q_\alpha, W_\alpha)\|_{L^p} 
\lesssim \| \tW\|_{L^q}  \| \tQ_\alpha\|_{L^\infty}^2\|\tW_\alpha\|_{L^2}
\\
\lesssim & \ t^{-1} \|(\tW,\tQ)\|^2_{X}\|\tW\|_{\dot H^\sigma} \| \tQ_\alpha\|_{L^2} ,
\end{split}
\]
where 
\[
\frac1p-\frac12 = \frac1q = \frac12 - \sigma .
\] 

 \bigskip

Summarizing, we set $W_{[3]} = W^{lhh\to h}_{[3],0}+ W^{hhh \to l}_{[3],1}$ and $Q_{[3]} = Q^{lhh\to l}_{[3],0}+ Q^{hhh\to l}_{[3],1}$,
and define our final normal form variables 
\begin{equation}\label{nft-tri}
\left\{
\begin{aligned}
 \ttW_t = & W + W_{[2]}+ W_{[3]} 
 \\
 \ttQ_t = & Q + Q_{[2]}+ Q_{[3]} .
\end{aligned} 
\right.
\end{equation}  
 We obtain a good system for
$(\ttW, \ttQ)$:
\begin{equation}\label{nf-final}
\left\{
\begin{aligned}
 \tilde\tW_t + \tilde \tQ_\alpha = &\ \tilde\tG
 \\
 \tilde\tQ_t +i \tilde \tW_{\alpha\alpha} = & \ \tilde \tK,
\end{aligned} 
\right.
\end{equation} 
where $(\tilde\tG,\tilde\tK)$ admit a favorable estimate \eqref{tt-gk}. For later use, we list their cubic parts  $(\tilde\tG^{(3)},\tilde\tK^{(3)})$:
\begin{equation} \label{ttgk-3}
\left\{
\begin{aligned}
\tilde\tG = & \ \tG_0^{good} + \tG^1 + \frac12 P^{lhh\to h} ( \Re Q_\alpha W_{[2],\alpha}- \Re W_\alpha Q_{[2],\alpha})
\\
\tilde \tK = & \ \tK_0^{good} + \tK_1 + \tK^{good}_2 + \frac12 P^{lhh\to h} ( \Re Q_\alpha Q_{[2],\alpha}- i \Re W_{\alpha\alpha} W_{[2],\alpha}).
\end{aligned}
\right.
\end{equation}
Finally, the bound \eqref{tt-gk4} for quartic and higher terms is an
easy consequence of Proposition~\ref{bi-est}. Proposition~\ref{tri-est} is also needed for 
the quartic terms generated by $(W^{hhh \to l}_{[3],1},\tQ^{hhh \to l}_{[3],1})$.
Here it helps that no negative Sobolev norms need to be estimated.  This concludes the proof
of part (ii) of the proposition.

It remains to consider the energy equivalence in part (i). Here we
need to bound the corrections $(W_{[2]},Q_{[2]})$ and
$(W_{[3]},Q_{[3]})$ in $\dH^\sigma$.  For this we use standard
multiplicative estimates as in Proposition~\ref{bi-est}, combined with
Sobolev embeddings. These estimates are somewhat tedious but routine,
and are left for the reader.
\end{proof}
\bigskip


\section{The linearized equation}
\label{s:lin}
Here we derive the linearized equation, and prove energy estimates
for it, first in $\dH^1$ and then in $\dH^\sigma$.

\subsection{ The equations.}
The solutions for the linearized water wave equation around a solution
$\left( W,Q\right) $ to the equation \eqref{ngst} are denoted by
$(w,q)$.  As in \cite{HIT}, it will be convenient to
also switch to diagonal variables $(w,r)$, where
\[
r := q - Rw.
\]
We will employ the variables $(w,r)$ at high frequency, but we will return 
to $(w,q)$ for low frequency bounds. 

\medskip

The linearization of $R$ is
\[
\delta R =
\dfrac{q_{\alpha}- Rw_{\alpha}}{1+\W}
= \dfrac{r_{\alpha}+ R_\alpha w}{1+\W},
\]
while the linearization of $F$ can be expressed in the form
\[
\delta F = P[ m - \bar m],
\]
where the auxiliary variable $m$ corresponds to differentiating
$F$ with respect to the holomorphic variables,
\[
m := \frac{q_\alpha - R w_\alpha}{J} + \frac{\bar R w_\alpha}{(1+\W)^2} =
 \frac{r_\alpha +R_\alpha w}{J} + \frac{\bar R w_\alpha}{(1+\W)^2}.
\]
We denote also
\[
n := \bar R \delta R =\dfrac{\bar R(q_{\alpha}- Rw_{\alpha})}{1+\W}
=  \frac{ \bar R(r_{\alpha}+R_\alpha w)}{1+\W}.
\]
Further, introducing the real coefficient $c$ by 
\[
i c :=  \frac{\W_\alpha}{J^\frac12(1+\W)}
- \frac{\bar \W_\alpha}{J^\frac12(1+\bar \W)},
\]
we have 
\[
i \delta c = p - \bar p,
\]
where
\[
p:=\frac{w_{\alpha \alpha}}{J^{1/2}(1+\W)} -\left( \frac{3\W_{ \alpha}}{2J^{1/2}(1+\W)^2}- \frac{\bar{\W}_{\alpha}}{2J^{3/2}}\right) w_{\alpha}.
\]

Then the linearized water wave equations take the form
\begin{equation}\label{lin(wq)}
\left\{
\begin{aligned}
&w_{t}+ F w_\alpha + (1+ \W) P[ m-\bar m] = 0, \\
&q_{t}+ F q_\alpha + Q_\alpha P[m-\bar m]   +P\left[n+\bar n\right] +i P\left[ p-\bar{p}\right] =0.
\end{aligned}
\right.
\end{equation}
Recalling that $b = F + \dfrac{\bar R}{1+\W}$ and replacing the $q$-equation with the 
corresponding $r$-equation, this becomes (see also the similar computation in \cite{HIT}):
\begin{equation}\label{lin(wr)hom}
\left\{
\begin{aligned}
& (\partial_t + b \partial_\alpha) w  +  \frac{1}{1+\bar \W} r_\alpha
+  \frac{R_{\alpha} }{1+\bar \W} w  = \mathcal{G}_0(w,r)
 \\
&(\partial_t + b \partial_\alpha)  r  - i  \frac{a}{1+\W} w+i P p  + 
 \frac{w}{1+\W}\partial_{\alpha}P c    = \mathcal{K}(w,r),
\end{aligned}
\right.
\end{equation}
where
\begin{equation*}
\left\{
\begin{aligned}
\mathcal{G}_0(w,r) &= \ (1+\W) (P \bar m + \bar P  m) \quad \\
\mathcal{K}(w,r) &=  \  \bar P n - P \bar n +i  P\bar p.
\end{aligned}
\right.
\end{equation*}
To rewrite this in a more useful form we introduce the operator $L$ given by
\[
\begin{split}
  L = & \ \partial_\alpha J^{-\frac12} \partial_\alpha -
  i c\partial_\alpha - i P c_\alpha
  \\
  = & \ \partial_\alpha J^{-\frac12} \partial_\alpha -
  i (c\partial_\alpha+\frac12 c_\alpha) - \frac{i}2 (P c_\alpha- \bar P c_\alpha).
\end{split}
\]
Both $c$ and the last coefficient $i(P c_\alpha- \bar P c_\alpha)$ are
real,  so $L$ is self-adjoint. We further note the quadratic expansion
of $L$, namely
\begin{equation}
L  \approx \partial_\alpha (1 - \Re \W) \partial_\alpha -
  i (2\Im \W_\alpha \partial_\alpha+ \Im \W_{\alpha\alpha}) - \Re \W_{\alpha \alpha} .
\end{equation}
Then our system takes the shorter form
\begin{equation}\label{lin(wr)hom1}
\left\{
\begin{aligned}
& (\partial_t + b \partial_\alpha) w  +  \frac{1}{1+\bar \W} r_\alpha
+   \frac{R_{\alpha} }{1+\bar \W} w  = \mathcal{G}_0(w,r)
 \\
&(\partial_t + b \partial_\alpha)  r  - i  \frac{a}{1+\W} w+i P\left[ \frac{Lw}{1+\W}\right]
  = \mathcal{K}(w,r).
\end{aligned}
\right.
\end{equation}
It is interesting to also observe that after a seemingly innocuous cubic change in 
$\mathcal{G}_0$ we can also replace $\dfrac{R_\alpha}{1+\bar \W}$ by $Pb_\alpha$.
Then our system is rewritten as 
\begin{equation}\label{lin(wr)hom2}
\left\{
\begin{aligned}
& (\partial_t + b \partial_\alpha) w  +  \frac{1}{1+\bar \W} r_\alpha
+  Pb\,  w  = \mathcal{G}(w,r)
 \\
&(\partial_t + b \partial_\alpha)  r  - i  \frac{a}{1+\W} w+i P\left[ \frac{Lw}{1+\W}\right]
  = \mathcal{K}(w,r),
\end{aligned}
\right.
\end{equation}
where the modified function $\mathcal{G}$ is given by 
\begin{equation*}
\mathcal{G} = (1+\W) (P \bar m + \bar P  m) + w(\bar P[ R_\alpha \bar Y] - P[R \bar Y_\alpha]).
\end{equation*}
Projecting \eqref{lin(wr)hom} onto the space of functions with the
spectrum on the negative side of the real line, we obtain
\begin{equation}\label{lin(wr)proj}
\left\{
\begin{aligned}
& (\partial_t + \M_b \partial_\alpha) w  + P \left[ \frac{1}{1+\bar \W} r_\alpha\right]
+    Pb_\alpha \,  w  = G
 \\
&(\partial_t + \M_b \partial_\alpha)  r  - i P\left[ \frac{a}{1+\W} w\right] +i P\left[ \frac{Lw}{1+\W} \right]= K,
\end{aligned}
\right.
\end{equation}
where
\begin{equation*}
\begin{aligned}
G(w,r):=P\mathcal{G}, \quad  K(w,r):=P\mathcal{K}.
\end{aligned}
\end{equation*}
The perturbative terms $G$ and $K$ are split into quadratic and higher
terms as shown below
\[
\begin{split}
   G = &  \,  G^{(2)}+  G^{(3+)},\ \ \   K = \  K^{(2)}+  K^{(3+)}.
\end{split}
\]
For the quadratic parts we have
\begin{equation*}
\begin{aligned}
G^{(2)}(w,r)=P[ R \bar{w}_{\alpha}- \W\bar{r}_{\alpha}],  \qquad  K^{(2)}(w,r)=- P [R \bar r_\alpha] - i  \frac12 P[\bar w_{\alpha\alpha} \W_\alpha - \bar w_\alpha \W].
\end{aligned}
\end{equation*}
The quartic and higher terms in $(G,K)$ are purely perturbative 
in our arguments, so we only need to consider some of the cubic terms.
In the $w$ equation, only the $w$ terms are interesting,
so we can set 
\[
G^{(3)} =  - P[ \W \bar R_\alpha \bar w] - P\left[ (P[R\bar \W_\alpha]-\bar P[R_\alpha \W] w\right] .
\]
In the $r$ equation, all undifferentiated cubic terms are relevant; there we set
\[
K^{(3)} = P[ R (\bar \W \bar r_\alpha- \bar R_\alpha \bar w) ] .
\]

Our main result for the linearized equation asserts that it is
well-posed in $\dH^1 \cap \dH^\sigma$:

\begin{theorem}
  Suppose that the solution $(W,Q) \in C(0,T; \dH^4 \cap \dH^\sigma)$ to
  the water wave equation is well defined in a time interval $[0,T]$
and satisfies the bounds
\[
\| (W,Q)(t)\|_{\dH^3 \cap \dH^\sigma} \leq \epsilon t^\delta, \qquad 
\| (W,Q)(t)\|_{X} \leq \epsilon t^{-\frac12}
\]
with some fixed $\delta \ll 1$, and $\epsilon \ll \delta$.
Then the solution $(w,q)$ to the linearized equation \eqref{lin(wq)} satisfies
the bounds 
\begin{equation}
\| (w,q)(t)\|_{\dH^1 \cap \dH^\sigma} \lesssim  t^{C\epsilon^2} \| (w,q)(0)\|_{\dH^1 \cap \dH^\sigma}.
\end{equation}
\end{theorem}

We remark that one particular solution for the linearized equation is
the scaling derivative $(w,q) = S(W,Q)$.  Then, given the $ S(W,Q)$ part
of the initial data bound \eqref{energy}, we can conclude that
we have
\begin{equation}\label{energy-s-lin}
\|S(W,Q)\|_{\dH^1 \cap \dH^\sigma} \lesssim  \epsilon t^{C\epsilon^2} ,
\end{equation}
which is part of our bootstrap argument in the proof of Theorem~\ref{t:almost}.

Our proof of the theorem requires two steps. First we consider the
high frequencies, for which we use the equation \eqref{lin(wr)proj}.
In order to estimate  the $\dH^1$ size of the solution $(w,r)$ we
produce a quasilinear cubic $\dH^1$ energy, expressed in terms 
of the diagonal variables $(w,r)$, but whose evolution is governed
by the full $\dH^1 \cap \dH^\sigma$ size of $(w,q)$. The outcome of this analysis
is contained in Proposition~\ref{p:lin-hi} below.

It remains to control the low frequencies, which is done using the
equation \eqref{lin(wq)}, expressed in terms of the original $(w,q)$
variables.  For these we instead use directly the normal form method,
with some added cubic corrections; this calculation is carried at the
level of the original variables $(w,q)$. The result is summarized in 
Proposition~\ref{p:lin-low} below.

Finally, the conclusion of the Theorem is obtained by we putting  together
the two results in Proposition~\ref{p:lin-hi}, respectively Proposition~\ref{p:lin-low}.
Unlike the case of gravity waves, the transition between the two sets of variables
$(w,r)$ and $(w,q)$ is harmless, due to the estimate
\begin{equation}
\| R w\|_{\dot H^\frac12 \cap \dot H^{\sigma-\frac12}} \lesssim \epsilon 
\|w\|_{\dot H^1 \cap \dot H^\sigma} .
\end{equation}

\subsection{\protect{$\dH^1$}  energy estimates.}
The goal of this section is to produce a cubic $\dH^1$ energy functional 
for the linearized equation. Precisely, we have

\begin{proposition}\label{p:lin-hi}
There exists an energy functional ${\bf E}^1(w,r)$ with the following two properties:

(i) Energy equivalence:
\begin{equation}\label{lin-eqe1}
{\bf E}^1(w,r) = (1+O(\epsilon)) \|(w,q)\|_{\dH^1}^2 + O(\epsilon) \|(w,q)\|_{\dH^\sigma}^2 ,
\end{equation}

(ii) Cubic energy estimates:
\begin{equation}\label{lin-dte1}
\frac{d}{dt} {\bf E}^1(w,r) \lesssim  \|(W,Q)\|_{X}^2 (\|(w,q)\|_{\dH^1}^2 + \|(w,q)\|_{\dH^\sigma}^2).
\end{equation}
\end{proposition}

We remark that the above estimate is not self-contained, and it can be
used only coupled with a similar $\dH^\sigma$ result. This issue is on
one hand due to our definition of $r$, which involves the
undifferentiated variable $w$.  One could avoid this by redefining $r$
as $r = q - T_R w$. However, this would make our equations and
analysis more complicated, and would serve little purpose in the
present paper, where the $\dH^\sigma$ result is the primary
one. Secondly, in order to define the normal form energy corrections
we have to use the undifferentiated $w$ regardless. So a lower norm
seems necessary for the cubic estimates in any case.

\begin{proof}

Naively we start with the linear energy $\| (w,r)\|_{\dH^1}^2$ and construct a suitable 
modification of it.There are three main difficulties in this process:

a) The equations are quasilinear, so a quasilinear correction is needed in order 
to account for the high frequencies.

b) The linear energy is not cubic, so a normal form type cubic correction is needed
in order to eliminate cubic errors.

c) The $L^2$ norm of $w$ is not controlled, so quartic error terms containing undifferentiated
$w$ cannot be controlled directly. A further quartic normal form correction to the 
energy needs to be used for such terms.

We will pursue these steps in the above order. 

\bigskip

a) The natural quasilinear energy functional associated to the above system is obtained 
 if we multiply the second equation by $i r_\alpha$ and the first by $Lw$.
These multipliers are matched so that the mixed $r_\alpha w$ terms cancel in the 
computation. Thus it is natural to define
 \begin{equation}
\label{first-e}
E^2_{lin}(w,r)= \int_{\mathbb{R}}   -   \Re (Lw \cdot \bar w) +
\Im (  r \partial_\alpha \bar r)
 \, d\alpha.
\end{equation}
Now we  compute the time derivative of the above energy.
We use the second equation to write
\[
\begin{split}
\frac{d}{dt} \int \Im (  r \partial_\alpha \bar r)  \, d \alpha = & \ 2 \Im  \int    (\partial_t
+ \M_b \partial_\alpha) r  \, \partial_\alpha \bar r \, d \alpha 
=  2 \Re \int ( i a w -  Lw)  \, \frac{ \bar r_\alpha}{1+\W}\, - i \mathcal K \bar  r_\alpha  d \alpha .
\end{split}
\]
To compute in the first equation we introduce the skew-adjoint operator
\[
N := \partial_t + b \partial_\alpha + Pb.
\]
Then we can write
\[
\frac{d}{dt} \int_{\mathbb{R}}    \Re (Lw \cdot \bar w)
 \, d\alpha =  \Re \int  - 2 Lw  \, \frac{ \bar r_\alpha}{1+\W}\, + 2 \bar {\mathcal G}   Lw 
 + [N,  L] w \cdot \bar w \, d\alpha.
\]
Summing up, the mixed terms cancel, and we obtain
\begin{equation}\label{wr-en}
\frac{d}{dt}E^2_{lin}(w,r)= 2 \Re \int  i  \bar {\mathcal K}  r_\alpha -  \bar {\mathcal G}  L w \,
 d \alpha + \int  [N,  L] w \cdot \bar w \, d\alpha + 
\int  \frac{a w  \bar r_\alpha}{1+\W} \, d\alpha .
\end{equation}
We now successively consider the three integrals above, which are
denoted by $I_1$, $I_2$ and $I_3$. Our goal will be to peel  off good 
terms, retaining a leading part which will be rectified with further energy corrections.
Here, by  good terms we mean terms which can be estimated by the right hand side 
in \eqref{lin-dte1}.
\bigskip

{\bf The integral $I_1$.} We begin with the $\mathcal K$ term. Since
$r_\alpha$ is holomorphic, we can freely replace $\mathcal K$ with
$K$, and then drop the projections. The only terms which cannot be
readily estimated as in the proposition are the cubic terms. Thus for
the corresponding part $I_1(\mathcal K)$ of $I_1$ we obtain
 \[
I_1(\mathcal K) =  2 \Re \int      i  r_\alpha( - \bar R q_\alpha+ i  \frac12 
\bar  \W  w_{\alpha \alpha}
- i  \frac12 \bar \W_\alpha w_\alpha)  \, d\alpha + good.
\]

Next we consider the $\mathcal G$ term. Integrating by parts we rewrite it as
\[
I_1(\mathcal G) = \int J^{-\frac12}  \bar {\mathcal G}_\alpha \cdot w_\alpha
+  \bar {\mathcal G} (  i (c\partial_\alpha+\frac12 c_\alpha) + \frac{i}2 (P c_\alpha- \bar P c_\alpha)\,
d\alpha .
\]
Here we recall that 
\[
\mathcal G = (1+\W) (P \bar m + \bar P  m) + w(\bar P[ R_\alpha \bar Y] - P[R \bar Y_\alpha]).
\]
For the first component of $\mathcal G$ we use the expression of $m$
in terms of $(w,q)$. Thus $\mathcal G_\alpha$ will contain higher
derivatives of $w$ and $q$. But these are either $\bar
w_{\alpha\alpha}$, $\bar q_{\alpha\alpha}$ and $\bar q_\alpha$ inside
a $P$ projector, or $ w_{\alpha\alpha}$, $ q_{\alpha\alpha}$ and $
q_\alpha$ inside a $\bar P$ projector. Thus the extra derivatives can
be harmlessly shifted to their co-factor in our estimates, see
e.g., Lemma 2.1, Appendix B in \cite{HIT}. At the other end, however,
we need to be more careful with the terms which contain
undifferentiated $w$, as we only control $w$ in $\dot H^\sigma$ and
not in $L^2$; for such terms we will use the Sobolev embedding $\dot
H^\sigma \subset L^{p}$ and regain the lost integrability from one of
the other factors.

 To summarize, in $I_1(\mathcal G)$ we can control all quintic and higher terms, and also 
 all quartic terms which do not contain undifferentiated $w$. 
Thus we retain the  cubic part of $I_1(\mathcal G)$, which is the
integral of $ m \cdot w_{\alpha\alpha}$, as well as the quartic terms with a $w$ factor. Expanding these terms, we have   
\[
\begin{split}
I_1(\mathcal G) = & \  2 \Re \int  -   w_{\alpha\alpha}(\bar R w_\alpha -\bar \W q_\alpha) 
 -   \bar w_{\alpha\alpha} w (P(R \bar \W_\alpha) - \bar P(R_\alpha \bar \W)) \, d \alpha
\\ & + 2 \Re \int  w \bar w_\alpha( \bar P(R_\alpha \W)_\alpha - P(R \bar \W_\alpha)_\alpha)
+  i  w  P c_\alpha \bar P(\bar R w_\alpha - \bar \W q_\alpha) \, d\alpha + good .
\end{split}
\]

\bigskip

{\bf The integral $I_2$.}
Here we need  to compute the above (selfadjoint) commutator. We have 
\[
\begin{split}
[N,L] =  & \  [\partial_t + b \partial_\alpha + Pb_\alpha, \partial_\alpha J^{-\frac12} \partial_\alpha
- ic \partial_\alpha  - iPc_\alpha]
\\ = & \ \partial_\alpha (\partial_t + b \partial_\alpha)J^{-\frac12} \partial_\alpha
- i(\partial_t + b \partial_\alpha) c \partial_\alpha  - i(\partial_t + b \partial_\alpha) Pc_\alpha
\\ & \ - \partial_\alpha J^{-\frac12} (b_\alpha \partial_\alpha+ P b_{\alpha\alpha})
-  ( \partial_\alpha b_\alpha - \bar P b_{\alpha\alpha})J^{-\frac12} \partial_\alpha
 + i c ( b_\alpha \partial_\alpha  + P b_{\alpha\alpha}) 
\\ = & \ \partial_\alpha [ (\partial_t + b \partial_\alpha)J^{-\frac12} - 2 J^{-\frac12} b_\alpha] \partial_\alpha - i ( c_t + bc_\alpha - c b_\alpha +  J^{-\frac12} H b_{\alpha \alpha}) \partial_\alpha
\\  & \ - i ( Pc_{\alpha t} + b P c_{\alpha \alpha} -  c P b_{\alpha \alpha})  -  \partial_\alpha( J^{-\frac12} P b_{\alpha \alpha}).
\end{split}
\]
First of all, we need to compute the linear part of this operator 
using 
\[
c \approx 2 \Im \W_\alpha, \qquad b \approx 2 \Re R, \qquad J^{-\frac12} \approx 1 - \Re \W  .
\]
Then we get 
\begin{equation}\label{[nl]1}
[N,L] \approx - 3 \partial_\alpha \Re R_\alpha \partial_\alpha ,
\end{equation}
therefore we can write
\[
I_2 = \int   3  \Re R_\alpha |w_\alpha|^2  d\alpha + quartic+ higher.
\]
However, as before, we also need to retain the quartic terms with undifferentiated
$w$.  Even worse are the terms with $|w|^2$. Hence for the lower order terms in $[N,L]$
we need a good quadratic expansion. Our complete expansion for $[N,L]$ will be
\begin{equation}\label{nl-exp}
[N,L] =  - 3 \partial_\alpha (\Re R_\alpha+quad)  \partial_\alpha
+ (L_2(\W,R) + cubic) \partial _\alpha + i   \partial_t(\bar R R_\alpha) + \partial_\alpha L_2(\W,R) + cubic.
\end{equation}
The quadratic part of the coefficient of  $\partial_\alpha$ is easily seen to have 
the form $L_2(\W,R)$. It remains to compute the  quadratic terms
in the zero order term of the commutator,
\[
\gamma = i ( Pc_{\alpha t} + b P c_{\alpha \alpha} -  c P b_{\alpha \alpha})  -  \partial_\alpha( J^{-\frac12} P b_{\alpha \alpha}).
\]
Precisely, we claim that we can write it in the form
\[
\gamma =i   \partial_t(\bar R R_\alpha) + \partial_\alpha L_2(W,R) + \gamma_3,
\]
where $\gamma_2=$ stands for a quadratic expression in $\W$, $R$ and their derivatives
while $\gamma_3$ contains only cubic and higher terms in $(\W,R)$. Indeed,  
we have 
\[
c \approx \Im (\W_{\alpha}(2 - 3\W - \bar \W)), \qquad b \approx 2 \Re P[ R(1-\W)] ,
\]  
while
\[
P c_{\alpha t} = - 2 \Im R_{\alpha \alpha \alpha} + \partial_\alpha L_2(\W,R),
\]
therefore
\[
\begin{split}
\gamma  = 
& \  (R+\bar R)  \W_{\alpha\alpha\alpha} -  (\W_\alpha - \bar \W_\alpha) R_{\alpha \alpha} 
+ \partial_\alpha L_2(\W,R) + \gamma_3
\\
= & \ i \partial_t (\bar R R_\alpha) + \bar \W_{\alpha\alpha} R_\alpha + \bar \W_\alpha R_{\alpha \alpha}  + \partial_\alpha  L_2(\W,R) + \gamma_3  
\\= & \ i  \partial_t (\bar R R_\alpha)  + \partial_\alpha  L_2(\W,R) + \gamma_3  ,
\end{split}
\]
which concludes the proof of the expansion \eqref{nl-exp}. Thus, for $I_2$ we obtain 
\[
I_2 = \Re \int  -3 \Re R_\alpha |w_\alpha|^2+     
 i \partial_t (RR_\alpha) |w|^2 +  w \bar w_\alpha  L_2(\W,R) \,  d\alpha + good.
\]

\bigskip

{\bf The integral $I_3$.} Here we simply retain the quartic part,
\[
I_3 = - \Re \int w \bar r_\alpha a  \, d\alpha  .
\]

Summing up the contributions from $I_1$, $I_2$ and $I_3$ and
reorganizing terms,  our energy computation so far reads as
\[
\frac{d}{dt}E^2_{lin}(w,r)= Er_1 + Er_2 + Er_3 + Er_4 + good,
\]
where $Er_1$ contains the holomorphic cubic terms,
\[
\begin{split}
Er_1 =  & \ 2 \Re \int      i  r_\alpha( - \bar R r_\alpha+ i  \frac12 
\bar  \W  w_{\alpha \alpha}
- i  \frac12 \bar \W_\alpha w_\alpha)  
 -   w_{\alpha\alpha}(\bar R w_\alpha -\bar \W r_\alpha) \, d\alpha
\\
= & \ 2\Re \int   - i  q_\alpha \bar R q_\alpha
+  \frac{1}2 
q_\alpha \bar  \W  w_{\alpha \alpha}
+   \frac{1}2 q_\alpha \bar \W_\alpha w_\alpha -  \sigma w_{\alpha\alpha} \bar R w_\alpha  \, d\alpha ,
\end{split}
\]
$Er_2$ contains the mixed cubic terms,
\[
Er_2 = \int  3  \Re R_\alpha w_\alpha \bar w_\alpha \, d\alpha ,
\]
$Er_3$ contains the single remaining quartic $|w|^2$ term,
\[
Er_3 =  \Re \int  \partial_t ( i \bar R R_\alpha)|w|^2
\, d \alpha,
\]
and  $Er_4$ contains the quartic terms with a single $w$ factor,
\[
Er_4 = \Re \int   w  w_\alpha L_2(W,R) + w \bar w_\alpha L_2(W,R)
+ w q_\alpha (L_1(\W,\W)+L_0(R,R)) \, d\alpha.
\]

Now we successively correct the energy in order to cancel the four terms above.
\bigskip

{\bf The $Er_1$ term.}
To account for $Er_1$ we introduce a correction of the form
\[
C_1 = 2 \Re \int L^1_1(w, q, \bar R) +  L^2_2(w,w,\bar \W) + L^3_1(q,q,\bar \W)\,
d\alpha,
\]
so that 
\begin{equation}\label{c1-er1}
\frac{d}{dt} C_1 = Er_1 + quartic.
\end{equation}
Here the trililnear forms $L^1_1$ have smooth homogeneous symbols, where the lower
index indicates the order. Thus, the quartic terms on the right in \eqref{c1-er1} may contain
at most one undifferentiated factor, and, if that factor is $w$, then they can be included in 
$Er_4$.

This and the next $Er_2$ part are the counterparts here of the normal
form calculation in Proposition~\ref{p:nf-cubic}, and are a consequence 
of the fact that bilinear interactions are non-resonant in our problem, except 
at frequency zero. Because of the resonance at frequency zero, it is very useful that
all factors in the $Er_1$ and $Er_2$ integrands are differentiated. One could recast this calculation in terms of the outcome of Proposition~\ref{p:nf-cubic},
but it is more efficient to redo it directly.

Denoting the symbols of $L^1_1$, $L^2_2$ and $L^3_1$ by $m_1$, $m_2$, respectively
$m_3$ for the three expressions, from the relation \eqref{c1-er1}  we obtain the system
\begin{equation*}
\left\{
\begin{aligned}
\ [\eta^2 m_1(\xi,\eta)]_{sym}  + (\xi+\eta) m_2(\xi,\eta)  =&  \ \frac12\xi \eta(\xi+\eta)
\\
- (\xi+\eta)^2 m_1(\xi,\eta) - 2\eta m_2(\xi,\eta) + 2\xi^2 m_3(\xi,\eta) =&  \ - \frac12 \xi \eta^2
\\
-[\xi m_1(\xi,\eta)]_{sym}  +  (\xi+\eta) m_3(\xi,\eta) =&  \ - \xi \eta .
\end{aligned}
\right.
\end{equation*}
Expanding the middle equation into a symmetric and an antisymmetric part we obtain
\begin{equation*}
\left\{
\begin{aligned}
\ [\eta^2 m_1(\xi,\eta)]_{sym}  + (\xi+\eta) m_2(\xi,\eta)  = & \ \frac12\xi \eta(\xi+\eta)
\\
- (\xi+\eta)^2 [\eta^2 m_1(\xi,\eta)]_{sym} 
- (\eta^3+\xi^3) m_2(\xi,\eta) + 2\eta^2 \xi^2 m_3(\xi,\eta) = & \ - 
\frac14 \xi \eta(\xi^3+\eta^3)
\\
- (\xi+\eta)^2 [\xi m_1(\xi,\eta)]_{sym} - 2\xi \eta m_2(\xi,\eta) + (\xi^3+\eta^3) m_3(\xi,\eta)
 = & \ - \frac12 \xi^2 \eta^2
\\
-[\xi m_1(\xi,\eta)]_{sym}  +  (\xi+\eta) m_3(\xi,\eta) = & \ - \xi \eta .
\end{aligned}
\right.
\end{equation*}
The matrix of this system is equivalent to the matrix in Proposition~\ref{p:nf-cubic}.
A short computation leads to
\begin{equation*}
\left\{
\begin{aligned}
m_2 =   \frac14(3\xi^2+ 10 \xi \eta + 3\eta^2)  \frac{\xi^2+ 3\xi \eta +\eta^2}{9 \xi^2 + 14 \xi \eta + 9 \eta^2},& \qquad m_3 =  -2 (\xi+\eta)    \frac{\xi^2+ 3\xi \eta +\eta^2}{9 \xi^2 + 14 \xi \eta + 9 \eta^2},\\
[\eta^2 m_1(\xi,\eta)]_{sym} = \frac12(\xi+\eta)(\xi\eta-m_2),& \qquad [\xi m_1(\xi,\eta)]_{sym} = \xi \eta +(\xi+\eta)m_3 .
\end{aligned}
\right.
\end{equation*}
Solving for $m_1$ from the last two relations yields
\[
m_1 = \frac{1}{\eta^3-\xi^3}( \xi \eta(\eta-\xi)(\eta+2\xi)  - (\xi+\eta) (\eta m_2+\xi^2 m_3)).
\]
The last factor vanishes when $\xi= \eta$, and we obtain
\[
m_1 = \frac{1}{\eta^2-\xi \eta +\xi^2}\left( \xi \eta (\eta+2\xi)  
- \frac{(\xi+\eta)(8\xi^2+13\xi\eta-3\eta^2) (\xi^2+ 3\xi \eta +\eta^2)}{4(9 \xi^2 + 14 \xi \eta + 9 \eta^2)}\right ).
\]
Thus, the symbols $m_1$, $m_2$ and $m_3$ are all regular. 

Finally, in order to prove \eqref{lin-eqe1} it suffices to verify that 
\[
|C_1| \lesssim  O(\epsilon) \|(w,r)\|_{\dH^1 \cap \dH^\sigma} ,
\]
which is straightforward by Proposition~\ref{bi-est}.

\bigskip

{\bf The $Er_2$ term.}  Here we seek to find an energy correction
which removes the main cubic error,
\[
Er_2 =  3  \int    \Re R_\alpha w_\alpha \bar w_\alpha \, d\alpha.
\]
Our first candidate is an  energy correction of the form
\[
C_2 = \int A(R,w,\bar r) + B(R,r,\bar w) +  C(\W,w,\bar w) + D(\W,r,\bar r) \, d\alpha.
\]
Then for the symbols of $A,B,C,D$ we get the linear system
\begin{equation*}
\left\{
\begin{aligned}
\zeta^2 A -\eta^2 B -\xi C = & \  \xi \eta \zeta 
\\
-\xi^2 B - \eta C + \zeta^2 D = & \ 0
\\
-\xi^2 A + \zeta C - \eta^2 D = & \ 0
\\
- \eta A +\zeta B-\xi D = & \ 0 
\end{aligned}
\right.
\end{equation*}
(where $\zeta = \xi +\eta$), which has solutions
\begin{equation*}
\begin{aligned}
&A = \frac{6 \zeta^3}{9\xi^2+14\xi\eta + 9\eta^2}, \qquad \quad \  B = \frac{2\zeta(-4\xi^2+3\xi \eta+3\eta^2)}{9\xi^2+14\xi\eta + 9\eta^2},\\
&C = \frac{\xi \zeta (6\xi^2+6\xi\eta+4\eta^2)}{9\xi^2+14\xi\eta + 9\eta^2}, \qquad D = \frac{-4 \xi \zeta^2}{9\xi^2+14\xi\eta + 9\eta^2}.
\end{aligned}
\end{equation*}
As defined, $C_2$ has the property that 
\[
\frac{d}{dt} C_2 = Er_2 + quartic.
\]
However, the quartic terms will include unbounded expressions involving $w_\alpha \bar w_{\alpha \alpha}$ and $r_\alpha \bar r_\alpha$. Thus we cannot work directly with $C_2$, and  instead we need a quasilinear correction for it. In order to find this quasilinear correction
we first separate a leading part for $C_2$,  which is obtained by taking the first term in the expansion of the symbols above
near $\xi = 0$. We obtain the differential trilinear form
\[
\begin{split}
C_{2,high} =  \int \frac43   \Re R  \Im(w_\alpha \bar r)  
- \frac49 \Im \W_\alpha(   \Im ( w \bar w_\alpha) + |r|^2) 
\,  d \alpha,
\end{split}
\]
which has the property that  the reminder $C_{2,low}= C_2 - C_{2,high}$ has the form
\[
C_{2,low} = \Re \int L_0(R_\alpha,w,\bar r) + L_0(R_\alpha,r,\bar w) + 
L_0(\W_\alpha,w,\bar w) + L_{-1}(\W_\alpha,r,\bar r) \, d\alpha.
\]
This has the key property that its time derivative no longer contains
unbounded expressions in $(w,r)$. The quasilinear modification of the
term $C_{2,high}$ is obtained in the same manner as in the high
frequency bounds in Proposition~\ref{t:en=small}, namely by setting
\[
\tilde C_{2,high} =  \int \frac{4}3 \Re  R \Im ((1+\W) w_\alpha \bar r)  - \frac{4}9
\Im \W_\alpha(   J^{-\frac12} \Im ( w \bar w_\alpha) + |r|^2) 
(  \W_\alpha J^{-\frac12} w \bar w_\alpha +  i \W_\alpha  r \bar r) \, d \alpha.
\]
Then we define the final correction
\[
\tilde C_{2} = \tilde C_{2,high} + C_{2,low} ,
\]
so that $\tilde C_{2} = C_2+ quartic$.
Now we consider the error term
\[
\frac{d}{dt} \tilde C_2 - Er_2 = quartic.
\]
  By construction this error term no longer contains any unbounded components,
and any part of it  which has a single undifferentiated $w$ can be included in $Er_4$.
The only remaining issue is whether we obtain any  
$|w|^2$ contributions from the lower order terms in $A$ and $C$. These 
can arise only in the case when the time derivative applies on the conjugated 
factor, and no $\alpha$ derivatives apply to the $w$ factor. Thus this time 
we can replace $C_2$ with its leading low frequency part, obtained by setting 
$\eta=0$ in our symbols. This gives the expression 
\[
C_{2,vlow} = \frac23 \Re  \int iR_\alpha w \bar r +  \W_{\alpha \alpha} w \bar w \, d\alpha .
\]
But from the equations \eqref{lin(wr)hom} we have 
\[
\partial_t r = i  \W_{\alpha\alpha} w + higher\ order, \qquad
\partial_t w = - R_\alpha w + higher\ order,
\]
which shows that the $|w|^2$ terms nicely cancel. We remark that this
is not at all surprising.  Had we done the $Er_2$ computations using
the $(w,q)$ variables, such terms would not have arisen in the first
place; however, this would have been very cumbersome for the high
frequency analysis.  

For \eqref{lin-eqe1} one also needs the bound 
\[
|C_2| \lesssim  O(\epsilon) \|(w,r)\|_{\dH^1 \cap \dH^\sigma} ,
\]
which is a consequence by Proposition~\ref{bi-est}.

\bigskip

{\bf The $Er_3$ term.}
The term 
\[
Er_3 =\Re  \int  \partial_t ( i \bar R R_\alpha)|w|^2 \, d\alpha 
\]
is corrected with 
\[
C_3 = \Re \int  ( i \bar R R_\alpha)|w|^2 \, d\alpha.
\]
This generates further quartic errors with a single undifferentiated 
 $w$, which are placed in $Er_4$. It is also easily bounded as an error term in \eqref{lin-eqe1}.

\bigskip

{\bf The $Er_4$ term.}
Here we need to  consider  quartic terms of the form 
\[
\begin{split}
Er_4 = & \Re \int   w  w_\alpha L_2(\W,R) + w \bar w_\alpha L_2(\W,R)
\\ & + w q_\alpha (L_1(\W,\W)+L_0(R,R)) + w \bar q_\alpha (L_1(\W,\W)+L_0(R,R)) \, d\alpha.
\end{split}
\]
Exactly as in the case of the low frequency bounds in Proposition~\ref{p:low-en},
such terms can be estimated directly unless we have the lowest frequency on
$w$.  Thus it suffices to consider the bad part of $Er_4$, which has the form
\[
Er_{4,bad} =  \Re \int   w  G^{hhh\to l}((w_\alpha,q_\alpha),(W_\alpha,Q_\alpha),
(W_\alpha,Q_\alpha)) \, d\alpha,
\]
with $G \in \C_0^3$ and $G^{hhh\to l}$ defined as in Proposition~\ref{p:nf-cubic},
with three high frequency inputs going to a low frequency output.
 But this frequency balance insures that the above  quartic 
form is non-resonant and can be removed with a normal form energy correction.
To achieve this, for this trilinear form $G$ we apply Proposition~\ref{p:nf-cubic}. We obtain 
a pair of trilinear forms 
\[
(W^{hhh\to l}_{[3]},Q^{hhh \to l}_{[3]}) \in (\B_1((w ,q),
(\W ,Q),(W ,Q)) ,\C_1((w ,q),
(\W ,Q),(W ,Q))),
\]
 so that we have
\begin{equation}\label{cubic-cor1}
\left\{
\begin{aligned}
& W^{hhh\to l}_{[3],t} + Q^{hhh\to l}_{[3],\alpha} = F + quartic \\
& Q^{hhh\to l}_{[3],t} +i W^{hhh\to l}_{[3],\alpha\alpha} = quartic.
\end{aligned} 
\right.
\end{equation}
Then our energy correction is defined as 
\[
C_4 =  \int  w W^{hhh\to l}_{[3]}((w ,q),
(W ,Q),(W,Q))\, d\alpha.
\]
Differentiating this with respect to time and integrating by parts we obtain
\[
\frac{d}{dt} C_4 - Er_{4,bad} = \int w_{\lambda,\alpha} Q^{hhh\to l}_{[3]}((w ,q),
(W ,Q),(W,Q))\, d\alpha + quintic,
\]
where $Q$ has at most one undifferentiated factor.
The quintic term is always easy to bound, and in the remaining quartic term we have a derivative
on the low frequency factor, so this is also admits a favorable estimate.

Finally, by Proposition~\ref{bi-est} it is not difficult to verify
that $C_4$ is a negligible error term for \eqref{lin-eqe1}.
\end{proof}

 \subsection{Low frequency analysis and the normal 
form linearized equation}

The goal of this section is to produce a cubic $\dH^\sigma$ energy functional 
for the linearized equation. Precisely,

\begin{proposition}\label{p:lin-low}
There exists an energy functional ${\bf E}^\sigma(w,q)$ with the following two properties:

(i) Energy equivalence:
\begin{equation}\label{e-sigma-eq}
{\bf E}^\sigma(w,q) =  \|P_{<1} (w,q)\|_{\dH^\sigma}^2+O(\epsilon)  
 \|(w,q)\|_{\dH^\sigma}^2,
\end{equation}

(ii) Cubic energy estimates:
\begin{equation}\label{e-sigma-dt}
\frac{d}{dt} {\bf E}^\sigma(w,q) \lesssim \|(W,Q)\|_{X}^2 \|(w,q)\|_{\dH^\sigma}^2.
\end{equation}
\end{proposition}

\begin{proof}
  The quasilinear nature of the problem is no longer a factor in the
  estimate for low frequencies.  Instead we can treat the equation as
  a semilinear equation, and only the quadratic and cubic terms
  matter. Then our starting point is the equation \eqref{lin(wq)} for
  the $(w,q)$ variables. The motivation is that in this way all
  nonlinear terms contain only differentiated factors, which is
  convenient at low frequency.

  To eliminate the quadratic terms and the bad cubic terms we will
  switch to normal form variables $(\ttW,\ttQ)$.  Rather than redo the
  computation from scratch, at this point it is very convenient to
  directly linearize the equation \eqref{nf-final} and denote by
  $(\ttw,\ttq)$ the associated linearized variables.

Thus, $(\ttw,\ttq)$ can be expressed as 
\begin{equation*}
\left\{
\begin{aligned}
&\ttw =  w +W_{[2]}((w,q),(W,Q))  +W_{[3]}((w,q),(W,Q),(W,Q))\\
&\ttq = q  +Q_{[2]}((w,q),(W,Q))  +Q_{[3]}((w,q),(W,Q),(W,Q)),
\end{aligned}
\right.
\end{equation*}
and their equation is obtained by linearizing the equation \eqref{nf-final},
\begin{equation*}
\left\{
\begin{aligned}
&\ttw_t + \ttq_\alpha =  \ G_{lin}(w,q)\\
&\ttq_t - i \ttw_{\alpha\alpha} = \ K_{lin}(w,q).
\end{aligned}
\right.
\end{equation*}

To conclude the proof, we define our low frequency energy as simply as
\begin{equation}
{\bf E}^\sigma(w,q) := \|P_{<1} (w,q)\|_{\dH^\sigma}^2.
\end{equation}

Then we need to prove the bounds
\begin{equation}\label{en-lin-errs}
\begin{aligned}
&\hspace*{1.5cm}\|P_{<1} (W_{[2]}((w,q),(W,Q)), Q_{[2]}((w,q),(W,Q))\|_{\dH^\sigma} 
 \lesssim \epsilon \| (w,q)\|_{\dH^\sigma},
\\
&\|P_{<1} (W_{[3]}((w,q),(W,Q)),(W,Q)),Q_{[3]}((w,q),(W,Q)),(W,Q)))\|_{\dH^\sigma}
 \|  \lesssim \epsilon \| (w,q)\|_{\dH^\sigma},
\end{aligned}
\end{equation}
respectively 
\begin{equation}\label{eqn-lin-errs}
\|P_{<1} (G_{lin},K_{lin})(w,q)W_{[2]}((w,q),(W,Q))\|_{\dH^\sigma} 
   \lesssim  \|(W,Q)\|_{X}^2 \| (w,q)\|_{\dH^\sigma}.
\end{equation}

{\bf Proof of \eqref{en-lin-errs}.} This argument is relatively straightforward,
once we separate the low and high frequencies of $(w,q)$. We do this 
for $(W_{[2]}, Q_{[2]})$; the argument for $ (W_{[3]}, Q_{[3]})$ is similar.

Suppose that first that $(w,q)$ are supported at low frequency $(< 1)$. Then we have 
control $\|w\|_{L^p}$ and $\|q\|_{L^2}$, as well as all their derivatives.
On the other hand for the output it suffices to bound 
\[
\|W_{[2]}((w,q),(W,Q))\|_{L^2} + \| Q_{[2]}((w,q),(W,Q))\|_{L^q}.
\] 
This follows directly from the bilinear estimates in Proposition~\eqref{bi-est}.

Suppose now that $(w,q)$ are supported at high frequencies ($ > 1$). The bilinear forms $(W_{[2]}, Q_{[2]})$,  $ (W_{[3]}, Q_{[3]})$ involve 
up to two derivatives of $(w,q)$, and we want to eliminate them using the fact that the output is localized at frequency $1$.
For that we write
\[
(w,q) = \partial^3_{\alpha}(w_1,q_1), \qquad (w_1,q_1) \in H^2.
\]
Then we use Leibniz's rule to write
\[
\begin{split}
W_{[2]}((w,q),(W,Q)) = & \  W_{[2]}(\partial_\alpha^2(w,q),(W,Q)) 
\\ = & \  \partial_\alpha^2 W_{[2]}((w_1,q_1),(W,Q)) - 2 \partial_\alpha W_{[2]}((w_1,q_1),\partial_\alpha(W,Q)) \\ & +  W_{[2]}((w_1,q_1),\partial_\alpha^2(W,Q)) .
\end{split}
\]
The derivatives in front are bounded once the $P_{<1}$ localization is added, 
and we conclude again using the bilinear estimates in Proposition~\eqref{bi-est}.

{\bf Proof of \eqref{eqn-lin-errs}.}  
This is a direct consequence of the proof of the estimate \eqref{tt-gk}, where 
we were careful to insure that the estimates for each component can be proved 
by using only an $\dH^\sigma \cap \dH^2$ on an arbitrarily chosen factor.

This yields the desired conclusion directly provided that $(w,q)$ are at low frequency.
If they are at high frequency, then we use Leibniz's rule as above in order to transfer 
the derivatives on either the output, or on the $(W,Q)$ factors.
  \end{proof}


\section{ Klainerman-Sobolev and  pointwise bounds }
\label{s:ks}
Here we use the energy estimates to derive the pointwise bounds with $\epsilon$ loss.
We begin by summarizing all the energy bounds we have at our disposal.
From the energy estimates for $(W,Q)$ we have 
\begin{equation}\label{en}
\|(W,Q)\|_{\dH^{10}\cap \dH^\sigma} \lesssim \epsilon t^{c \epsilon^2 } .
\end{equation}
Using Proposition~\ref{bi-est}, this transfers to the normal form 
variables $(\ttW,\ttQ)$ introduced in Proposition~\ref{p:low-en} with a two derivatives loss,
\begin{equation}\label{NFen}
\|(\ttW,\ttQ)\|_{\dH^8\cap \dH^\sigma} \lesssim \epsilon t^{c \epsilon^2 } .
\end{equation}
In effect, if we also use the pointwise bootstrap assumption
\eqref{point-boot}, then for the difference we get a better bound at low frequency,
\begin{equation}\label{NFen-diff0}
\|(\ttW-W,\ttQ-Q)\|_{\dH^\sigma} \lesssim \epsilon^2 t^{c \epsilon^2 -\sigma},
\end{equation}
and an even better bound at high frequency,
\begin{equation}\label{NFen-diff}
\|(\ttW-W,\ttQ-Q)\|_{\dH^8\cap \dH^\frac12} \lesssim \epsilon^2 t^{c \epsilon^2+\sigma-\frac12 } .
\end{equation}
These are proved using Proposition~\ref{bi-est}. The weaker bound at low frequency 
is due to the fact that   the $\dot H^{\sigma - \frac12}$ bound is obtained by a Sobolev embedding
from the corresponding $L^p$ type bound, which in turn requires combining an $L^2$ bound
with an $L^q$ bound rather than with $L^\infty$.

Also,  from the analysis of the normal form equation in Proposition~\ref{p:low-en} we obtain
\begin{equation}\label{NFeq}
\|(\ttW_t + \ttQ_\alpha, \ttQ_t+i \ttW_{\alpha\alpha})\|_{\dH^\sigma} \lesssim \epsilon t^{-1}
 t^{c \epsilon^2} .
\end{equation}
On the other hand, applying the bounds for the linearized equation to $(w,q)= \SS(W,Q)$,
and combining with the above we have
\begin{equation}\label{Sen}
\| (SW,SQ)\|_{\dH^\sigma \cap \dH^1} \lesssim \epsilon e^{c \epsilon^2 t} .
\end{equation}
Combining the above bound with \eqref{en} we can switch to the normal form variables
$(\ttW,\ttQ)$, 
\begin{equation}\label{SNFen}
\| (S\ttW,S\ttQ)\|_{\dH^\sigma + \dH^0} \lesssim \epsilon  t^{c \epsilon^2 } .
\end{equation}
Finally, we can eliminate the time derivative from the  last two relations 
to arrive at
\begin{equation}\label{Len}
\|(t \ttQ_\alpha-\frac23 \alpha \ttW_\alpha, it \ttW_{\alpha\alpha}-\frac23 \alpha \ttQ_\alpha)
\|_{\dH^\sigma +\dH^0} \lesssim \epsilon t^{c \epsilon^2} .
\end{equation}
Our pointwise bounds will be derived from \eqref{NFen} and
\eqref{Len}.  In order to state them we introduce the notations
\[
\at =  \frac{t}{|\alpha|+t^\frac23}, \qquad \omega(\alpha,t) = \min\left\{ \left(\frac{t}{\alpha}\right)^\frac12, \at^{-\frac12} \right\}  .
\]
Then our main result is as follows:

\begin{proposition}\label{p:K-S}
Assume that $(\ttW,\ttQ)$ satisfy \eqref{NFen} and \eqref{Len}. Then we have the 
pointwise bounds:

a) For $\ttW$ we have 
\begin{equation}\label{point-ttW}
|\ttW| \lesssim \epsilon [\omega(\alpha,t) t^{-\frac12}]^{1-2\sigma} t^{c \epsilon^2},
\end{equation}
respectively
\begin{equation}\label{point-tt}
|D^{j+\frac12} \ttW|+ |D^j \ttQ| \lesssim \epsilon \omega(\alpha,t) t^{-\frac12} t^{c \epsilon^2}, \qquad j = 0,1,2,3,4.
\end{equation}

\end{proposition}

We remark that once we have these bounds, it is straightforward to return to $(W,Q)$:
\begin{corollary}
Suppose that \eqref{en} and \eqref{Len} hold. 
Then we also have 
\begin{equation}\label{point-diff0}
|\ttW - W| \lesssim  [\omega(\alpha,t) t^{-\frac12}]^{1-2\sigma} t^{-\frac14},
\end{equation}
respectively
\begin{equation}\label{point-diff}
|D^{j+\frac12} (W-\ttW)|+ |D^j (Q-\ttQ)| \lesssim \omega(\alpha,t) t^{-\frac12} t^{-\frac14}, 
\qquad j = 0,1,2,3,4.
\end{equation}
In particular  the bounds in the above proposition apply 
to $(W,Q)$ as well. Also $R^{(j)}$ satisfies the same pointwise bounds as $Q^{(j-1)}$
for $j = 0,1,2,3$.
\end{corollary}
The proof of the Corollary is routine, and is left for the reader (see
also the similar argument in \cite{HIT}). The improvement in the bounds for the difference
comes from its bilinear and trilinear character. 

One consequence of the above proposition and corollary is that in our
main bootstrap argument it suffices to replace $(W,Q)$ by $(\ttW,\ttQ)$
in \eqref{point}. It also remains 
to prove \eqref{point} in a restricted region of the form 
$\{ t^{-\delta} \leq |\alpha|/t \leq t^\delta\}$, where $\delta$ is a small universal constant. 
This requires $\epsilon^2 \ll \delta$.

\begin{proof}[Proof of Proposition~\ref{p:K-S}]
To keep the notations simple, for this proof we drop the $\epsilon t^{c\epsilon^2}$ factor,
and also drop the double tilde notation.

Our hypothesis is stable with respect to frequency localizations, so  
we  can study first the frequency localized problem and then add up the results.
By scaling it suffices to consider $(W,Q)$ which are localized at frequency $1$.
Then we have:
\begin{lemma}
Suppose that $(W,Q)$ are localized at frequency $1$ and satisfy 
\begin{equation*}
\left\|\left(t Q_\alpha-\frac23 \alpha W_\alpha, it W_{\alpha\alpha}-\frac23 \alpha Q_\alpha\right)
\right\|_{L^2} \lesssim 1, \qquad \|(Q,W)\|_{L^2} \lesssim \delta \leq 1.
\end{equation*}
Then we have the pointwise bound
\begin{equation*}
|W| + |Q| \lesssim \delta^\frac12 t^{-\frac12}, \qquad |\alpha| \approx t \geq \delta^{-1},
\end{equation*}
respectively
\begin{equation*}
|W| + |Q| \lesssim \min \{\delta,  (t+|\alpha|)^{-1}\} , \qquad otherwise.
\end{equation*}
 \end{lemma}

\begin{proof}
We rewrite the relation in the Lemma as
\[
\|(t \tQ-\frac23 \alpha \tW, it \tW_{\alpha}-\frac23 \alpha \tQ)\|_{L^2} \lesssim 1.
\]
We consider several cases:

(i) $|t| \lesssim \delta^{-1}$. Then we can eliminate the $t$ factor from the first relation
to obtain 
\[
\|\alpha  (W,Q)\|_{L^2} \lesssim 1, \|(W,Q)\|_{L^2} \lesssim \delta,
\]
which easily yields the pointwise bound 
\[
|(W,Q)| \lesssim \min \{ \delta, |\alpha|^{-1} \} .
\]

(ii) $ t > \delta^{-1}$, $|\alpha| \gg t$.
There the operator $L$ is elliptic, and we obtain the elliptic bound
\[
\| \chi_{\alpha \gg t} \alpha (W,Q)\|_{L^2} \lesssim 1,
\]
which leads to the pointwise bound 
\[
|(W,Q)| \leq |\alpha|^{-1}.
\]

(iii) $ t > \delta^{-1}$,  $\alpha \ll t$.
There the operator $L$ is again elliptic, and we obtain the elliptic bound
\[
\| \chi_{\alpha \ll t}  (W,Q)\|_{L^2} \lesssim t^{-1},
\]
which leads to the pointwise bound 
\[
|(W,Q)| \leq t^{-1}.
\]

(iv) $ t > \delta^{-1}$, $|\alpha| \approx t$.
There we combine the two terms in $L$ to obtain
\[
\| (\partial_\alpha + i\frac{4\alpha^2}{9t^2}) \chi_{|\alpha|\approx
  t} W\|_{L^2} \lesssim t^{-1}, \qquad \|W\|_{L^2} \lesssim \delta.
\]
By Cauchy-Schwarz inequality this yields
\[
|W| \lesssim \delta^\frac12 t^{-\frac12},
\]
and the similar bound for $Q$ immediately follows.
\end{proof}

Rescaling the above result, we obtain the following

 \begin{corollary}
Suppose that
 $(W,Q)$  are localized at frequency $\lambda$ with 
\begin{equation*}
\|(t Q_\alpha-\frac23 \alpha W_\alpha, it W_{\alpha\alpha}-\frac23 \alpha Q_\alpha)
\|_{\dH_0} \lesssim 1, \qquad \|(W,Q)\|_{\dH_0} \lesssim \delta.
\end{equation*}
Then the hyperbolic (non elliptic) region is 
\[
|\alpha| \approx \lambda^\frac12 t, \qquad t > \delta^{-1} \lambda^{-\frac32} ,
\]
and there we have the bound for the hyperbolic component $(W_{hyp},Q_{hyp}) = \chi_{|\alpha| \approx \lambda^\frac12 t} 
(W,Q)$: 
\begin{equation} \label{wl-hyp}
|W_{hyp}| \lesssim \delta^\frac12 |\alpha|^{-\frac12} , \qquad 
|Q_{hyp}| \lesssim \delta^{\frac12} |\alpha|^\frac12
t^{-1}.
\end{equation}
Elsewhere we have a better bound for the remaining elliptic component, both pointwise
\begin{equation}\label{wl-ell}
|W_{ell}| \lesssim \min\left\{ \delta \lambda^\frac12,
  \frac{1}{(|\alpha|+\lambda^\frac12 t) \lambda^\frac12}\right\} ,
\qquad |Q_{ell}| \lesssim \min\left\{ \delta \lambda,
  \frac{1}{(|\alpha|+\lambda^\frac12 t)}\right\},
\end{equation}
and in $L^2$:
\begin{equation}\label{wl-ell2}
\|(\delta^{-1} + \lambda(\lambda^{\frac12} t+ |\alpha|)) 
( W_{ell}, \lambda^{-\frac12} Q_{ell})\|\lesssim 1.
\end{equation}

\end{corollary}

To obtain the conclusion of the Proposition we consider a dyadic
decomposition with respect to the dyadic frequency parameter
$\lambda$, 
\[
(W,Q) = \sum_{\lambda \ dyadic} (W_\lambda, Q_\lambda) ,
\]
and then a further decomposition into hyperbolic and elliptic components,
\[
(W_{hyp} ,Q_{hyp}) = \sum_{\lambda \ dyadic} (W_{\lambda,hyp}, Q_{\lambda,hyp}) .
\]
Then the conclusion of the proposition is obtained by doing dyadic
summation of the bounds in the above corollary with respect to the
dyadic frequency $\lambda$.

First we consider frequencies $\lambda$ which are less than $1$. Then
we do the summation with $\delta = 1$ and an extra $\lambda^{-\sigma}$
factor. The non-elliptic regions are essentially disjoint, and span
the range $t^{\frac23} \lesssim |\alpha| \lesssim t$.  There we get
\begin{equation} \label{ks1}
|W_{<1, hyp}| \lesssim t^{-\frac12} \at^{\frac12+\sigma} ,
\end{equation}
and also 
\begin{equation}\label{ks2}
|D^\frac12 W_{<1, hyp}|+ |Q_{< 1, hyp}| \lesssim  t^{-\frac12} \at^{- \frac12+\sigma} ,
\end{equation}
which only gets better for higher derivatives.
 In the elliptic case
we get the sum
\[
|W_{elliptic}| \lesssim \sum_{\lambda < 1} \lambda^{-\sigma} \min\left\{ \lambda^\frac12, \frac{1}{(|\alpha|+\lambda^\frac12 t) \lambda^\frac12} \right\}.
\]
The sum is comparable to the maximum of the summands, which is
attained at the balance point,
\[
\lambda(|\alpha| + \lambda^\frac12 t) = 1.
\]
That gives $\lambda = (|\alpha|+ t^\frac23)^{-1}$, therefore we obtain the pointwise bound
\begin{equation}\label{ks3}
|W_{< 1, ell}| \lesssim (|\alpha|+ t^\frac23)^{\sigma -\frac12} .
\end{equation}
The case of $Q$ or $D^\frac12 W$ is more favorable.
We get the same threshold for $\lambda$ if $|\alpha| < t$, but the bound
we get there is 
\begin{equation}\label{ks4}
|D^\frac12 W_{<1,ell}| + |Q_{< 1, ell}| \lesssim t^{-1}  \at^{1+\sigma} ,
\end{equation}
while for $|\alpha| > t$ we simply get $|\alpha|^{-1}$. Both of these are better than needed.
The bounds for higher derivatives of $W$ and $Q$  are better.

In a similar manner we have the $L^2$ bound 
\begin{equation}\label{ks5}
t \|D^{\sigma+2} W_{<1,ell} \|_{L^2} + 
\| \alpha D^{\sigma+1} W_{<1,ell} \|_{L^2}
+ \| D^\sigma W_{<1,ell} \|_{L^2}  \lesssim 1,
\end{equation}
and similarly for $D^{-\frac12} Q_{<1,ell}$.
In effect the previous pointwise bounds follow by  Bernstein's inequality from the 
$L^2$ bounds.

It remains to consider higher frequencies. There we work with the $L^2$ bound
but set $\delta = \lambda^{-9}$. Thus we obtain a very good  bound for the hyperbolic part,
\begin{equation}\label{ks6}
|D^{4+\frac12} W_{> 1, hyp}|+|D^4 Q_{> 1, hyp}| \lesssim t^{-\frac12}
\left(\frac{t}{\alpha}\right)^{10} , \qquad t > \alpha,
\end{equation}
which suffices.

For the elliptic part we have the bound
\begin{equation}\label{ks7}
|D^{3+\frac12} W_{> 1, ell}|+|D^3 Q_{>1,ell}| \lesssim \sum_{\lambda > 1} \lambda^{-3}\min\left\{ \lambda^\frac12, \frac{1}{(|\alpha|+\lambda^\frac12 t) \lambda^\frac12} \right\}
\lesssim (|\alpha|+t)^{-1},
\end{equation}
which yields the $\lambda = 1$ value, and is no worse than the low frequency case.  
We also have the $L^2$ bound
\begin{equation}\label{ks8}
\|t D^2 W_{> 1,ell}\|_{L^2}+ \| \alpha D W_{>1,ell}\|_{L^2} + \|D^6 W_{>1,ell}\|_{L^2} \lesssim 1.
\end{equation}
\end{proof}


\section{ The wave packet method and global pointwise bounds}
In order to establish the global pointwise decay estimates we use the
method of testing by wave packets, first introduced in paper \cite{IT}
in the context of the one dimensional cubic NLS equation, and later
used in \cite{TI} in the context of the two dimensional gravitational
water waves. 

Our starting point for the analysis is the same as in the previous
section, namely the estimates \eqref{en}-\eqref{Len}. Consequently, we
can also use the conclusion of the previous section, namely the bounds
\eqref{point-ttW}-\eqref{point-diff}. For some portion of the arguments,
we will also rely on the decomposition of $(\ttW,\ttQ)$ into elliptic and hyperbolic
components, arising in the proof of Proposition~\ref{p:K-S}.

We succinctly revisit some details of the wave packet method; we pick a ray
$\{\alpha = v t\}$ and establish decay along this ray by testing with a
wave packet moving along the ray. A wave packet is an approximate
solution to the linear system \eqref{ww-lin}, with $O(1/t)$ errors.

To motivate the definition of this packet we recall some
useful facts. In view of the dispersion relation $\tau = \pm |\xi|^{\frac{3}{2}}$,
a ray with velocity $v$ is associated with waves which have 
spatial frequency 
\[
\xi_v = - \frac{4}{9}v^2.
\]
Secondly, for waves with initial data localized at the origin,
the spatial frequency corresponding with a position $(\alpha,t)$ is 
\[
\xi(\alpha,t) = -\frac{4}{9}\frac{\alpha ^2}{t^2}.
\]
This is associated with the phase function
\[
\phi(t,\alpha) = -\frac{4}{27}\frac{\alpha^3}{t^2}.
\]

Then our wave  packets will be combinations of functions of the form
\[
\uu(t,\alpha) = v^{-\frac12} \chi\left(\frac{\alpha - vt}{t^{\frac12} v^{-\frac12}}\right)e^{i\phi(t,\alpha)},
\]
where $\chi$ is a smooth compactly supported bump function with integral one
\begin{equation}\label{chi-int}
\int \chi (y)\, dy =1.
\end{equation}
Our packets are localized around the ray $\{\alpha = v t\}$ on the
scale $\delta \alpha = t^{\frac{1}{2}} v^{-\frac{1}{2}}$. 
We further remark that there is a threshold
$v \approx t$ above which $\phi$ is essentially zero, and the above considerations
are no longer relevant. We confine our analysis to the region where $\phi$ is strongly
oscillatory, 
\begin{equation*}
|v| \gg  t^{-\frac16}.
\end{equation*}
The power $\frac16$ here is somewhat arbitrary, any choice less than $\frac13$ would do. 
Under this assumption, the function $\uu$ is strongly localized at frequency $\xi_v$.

We apply the method of testing by wave packets which we successfully
used for the water wave equation in \cite{TI}, and for the cubic NLS
in \cite{IT}. Due the fact that we are dealing with a system, and we
need to choose the two components to match.  Our system is simple
enough, so is suffices to first choose the $Q$ component and then use
the first of the two linear equations in \eqref{ww-lin} to match $W$,
\begin{equation*}
(\ww,\qq) = -\frac{27 i}{8v^2} ( -  \uu_{\alpha},  \uu_{t}).  
\end{equation*}

Then we have
\begin{equation}
\label{wq-bold}
\begin{cases}
& \!\!\! \!\!\! \ww =   \dfrac{3\alpha ^2}{2t^2 v^3}\uu+ \dfrac{27 i}{8v^2}\dfrac{1}{t^{\frac{1}{2}}v^{-1}}\chi ' \left(\dfrac{\alpha - vt}{t^{\frac12} v^{-\frac12}}\right)e^{i\phi} =\dfrac{3}{2v} \uu + O(v^{-\frac12}t^{-\frac12})e^{i\phi}
\\ & 
\\
& \!\!\! \!\!\! \qq=  \dfrac{\alpha^3}{t^3v^3} \uu +\dfrac{27 i}{8v^2} \dfrac{\alpha +vt}{2t^{\frac{3}{2}}v^{-1}}\chi'\left(\dfrac{\alpha - vt}{t^{\frac12} v^{-\frac12}}\right) e^{i\phi}= \uu + O(v^{\frac12} t^{-\frac12})e^{i\phi}.
\end{cases}
\end{equation}
The remainder terms on each of the two lines above are also bump
functions on the $t^{\frac12} v^{-\frac12}$ scale, and are better by a
$v^\frac12 t^{-\frac12}$ factor.  Because of this, they will play a
negligible role in most of our analysis. However, we cannot discard
them as they are crucial in improving the error in the second linear
equation in \eqref{ww-lin}, which is given by
\begin{equation}
\label{g-eq}
\ggg  := \partial_t \qq +i\partial^2_\alpha \ww = -\frac{27 i}{8v^2} (\partial^2_t - i \partial_{\alpha}^3) \uu .
\end{equation}
 Indeed,  computing the  error in \eqref{g-eq} we obtain
\begin{equation}
\label{erori}
\begin{aligned}
\ggg =e^{i\phi} \partial_{\alpha}\left[ \frac{v^\frac12(\alpha-vt)}{t}  \chi- \frac{ 9i}{8 t^\frac12}\chi '\right]+ O(t^{-\frac{3}{2}}).
 \end{aligned}
\end{equation}
The first term is the leading one, and, as expected, has size $t^{-1}$
times the size of $\qq$; the remaining terms are better than $t^{-1}$
by an additional factor of $t^{-\frac{1}{2}}$ which shows that the
choice of a such wave packet is a reasonable approximate solution for the linear system. Precisely, as in \cite{TI}, our
test packets $(\ww,\qq)$ are good approximate solutions for the linear
system associated to our problem only on the dyadic time scale $\delta
t\leq t$.

The outcome of testing the normal form solutions to the water wave
system with the wave packet $(\ww,\qq)$ is the scalar complex valued
function $\gamma(t,v)$, defined in the region $\{|v| \leq
t^{\delta}\}$, where $\delta \ll 1 $:
\begin{equation*}
\gamma(t,v) = \frac12 \langle (\ttW,\ttQ),(\ww,\qq)\rangle_{\dot{H}^{\frac{1}{2}}\times L^2},
\end{equation*}
which we will use as a good measure of the size of $(\ttW,\ttQ)$ along
our chosen ray.  Here it is important that we use the complex pairing
in the inner product. 

Now we have two
tasks. Firstly, we need to show that $\gamma$ is a good representation of the 
pointwise size of $(\ttW,\ttQ)$ and their derivatives:

\begin{proposition}
\label{p:diff}
Assume that \eqref{NFen} and \eqref{Len} hold. Then in
$\Omega:=\left\lbrace t^{-\delta}\leq \frac{\vert \alpha \vert
  }{t}\leq t^{\delta}\right\rbrace $ we have the following bounds for
$\gamma$:
\begin{equation}\label{gamma}
\begin{split}
\|(1+ v^{2})^{8}\gamma \|_{L^2_v}  + \|v(1+v^2)^{-\frac{1}{2}} \partial_v\gamma \|_{L^2_v} +  \|v^\frac12 (1+ v^2)^{\frac72}\gamma \|_{L^\infty}  \lesssim \epsilon t^{C^2_{*}\epsilon^2},
\end{split}
\end{equation}
as well as the approximation bounds for $(\ttW,\ttQ)$ and their derivatives:
\begin{equation} \label{pactest}
\begin{split}
&(|D|^{s+\frac12} \ttW,|D|^s \ttQ) (t,vt) =  \ |\xi_v|^{s}
t^{-\frac12} e^{i\phi(t,vt)} \gamma(t,v) (\sgn{v}, 1) + \err_s, 
\end{split}
\end{equation}
where
\begin{equation}\label{pacerr}
\begin{split}
&\|(1+v^{2})^{4-s} \err_{s}\|_{L^2_v}  + \|(1+v^{2})^{4-s}  \err_s\|_{L^\infty} \lesssim  
\epsilon t^{-\frac58 +C_{*}^2\epsilon^2}, \qquad 0 \leq s \leq 4.
\end{split}
\end{equation}
\end{proposition}
By virtue of the difference bounds \eqref{point-diff}, the estimates 
\eqref{pactest}-\eqref{pacerr} also hold with $(\ttW,\ttQ)$ replaced by $(W,Q)$.

Secondly, we need to show that $\gamma$ stays bounded,  which we do by 
establishing an asymptotic differential equation for it:
\begin{proposition}
\label{p:ode}
Assume that the bounds  \eqref{en}-\eqref{Len} hold.
Then within the set $\Omega$ the function $\gamma$ solves an asymptotic ode of the form 
\begin{equation}\label{gamma-ode}
\dot\gamma = -  \frac{8i v^5}{81 t }   \gamma |\gamma|^2 + err_{\gamma}, 
\end{equation}
where $err_\gamma$ satisfies the $L^2$ and $L^\infty$ bounds
\begin{equation}\label{sigma}
 \| (1+v^{2})^{4} err_\gamma\|_{L^2}+ \| (1+v^{2})^{4} err_\gamma\|_{L^\infty}  \lesssim \epsilon
  t^{-\frac{19}{18}}.
\end{equation}
\end{proposition}

We now use the two propositions to conclude the proof of \eqref{point}, and thus the proof 
of Theorem~\ref{t:almost}. 
By virtue of \eqref{pactest} and \eqref{pacerr}, in order to prove \eqref{point}
it suffices to establish its analogue for $\gamma$, namely 
\begin{equation}\label{need}
|\gamma(t,v)| \lesssim \epsilon (1+v^{2})^{-4} \qquad \text{  in  $\Omega$}.
\end{equation}
On the other hand,  from \eqref{point-tt}, we directly obtain
\begin{equation}\label{have}
|\gamma(t,v)| \lesssim \epsilon (1+v^{2})^{-4} \omega(v,t) t^{C_*^2 \epsilon} \text{ in $\Omega$}.
\end{equation}
Our goal now is to use the ode \eqref{gamma-ode} in order to pass from 
the bound \eqref{have} to \eqref{need} along rays $\alpha = vt$. We consider three cases for $v$:
\medskip

{\bf (i)}  Suppose first that
$v \approx 1$, i.e., $|\alpha| \approx t$. Then we initially have 
\[
|\gamma(t)| \lesssim \epsilon, \qquad t \approx 1.
 \]
Integrating \eqref{gamma-ode} we conclude that 
\[
|\gamma(t)| \lesssim \epsilon, \qquad t \geq 1,
\]
and then \eqref{need} follows.

\medskip

{\bf (ii)} Assume now that $v \ll 1$, i.e., $|\alpha| \ll t$. Then, as $t$ increases, the ray $\alpha = vt$
enters $\Omega$ at some point $t_0$ with $v \approx t_0^{-\frac19}$. Then by 
 \eqref{have}, if $\epsilon$ is small enough then we obtain
\[
|\gamma(t_0,v)| \lesssim \epsilon v^{\frac12}  t_0^{C_*^2 \epsilon}  \lesssim \epsilon.
\] 
We use this to initialize $\gamma$. For larger $t$ we use \eqref{gamma-ode}
to conclude that 
\[
|\gamma(t)| \lesssim \epsilon + \int_{t_0}^\infty \epsilon  
s^{-\frac{19}{18}+C\epsilon^2} ds \approx 
 \epsilon  (1+ t_0^{-\frac1{18}+C^2_{*}\epsilon^2 }) \lesssim \epsilon, \qquad t > t_0.
\]
Then  \eqref{need} follows.

\medskip

{\bf (iii)} Finally, consider the case $v \gg 1$, i.e., $|\alpha| \gg t$.
Again, as $t$ increases, the ray $\alpha = vt$ enters $\Omega$ at some point $t_0$ with
 $v \approx t_0^{\frac19}$, therefore by \eqref{have} we obtain
\[
|\gamma(t_0,v)| \lesssim \epsilon  (1+v^2)^{-4} \omega t_0^{C_*^2 \epsilon}  \lesssim \epsilon
 (1+v^2)^{-4}.
\] 
We use this to initialize $\gamma$. For larger $t$ we use \eqref{gamma-ode}
to conclude that 
\[
|\gamma(t)| \lesssim(1+v^2)^{-4}( \epsilon  + \int_{t_0}^\infty \epsilon s^{-\frac{19}{18}+C^2_{*}\epsilon^2})   \lesssim \epsilon(1+v^2)^{-4} ,  \qquad t > t_0.
\]
Then  \eqref{need} again follows.
\bigskip

The remainder of the paper is devoted to the proof of the two propositions above.

\begin{proof}[Proof of Proposition~\ref{p:diff}]
Here we prove Proposition~\ref{p:diff}, using the estimates \eqref{ks1}-\eqref{ks8}
in the proof of Proposition~\ref{p:K-S}. We write $\gamma$ as
\[
\gamma = \frac12 \int -i\ttW_{\alpha}\bar \ww + \ttQ \bar \qq \, d\alpha .
\]
The two terms in the integral are considered separately in the proof of \eqref{gamma},
and the exact form of $(\ww,\qq)$ is not important.
Further, for the purpose of proving \eqref{pacerr} we can simplify somewhat the
expression of $\gamma$. The lower order terms in $(\ww, \qq )$ in
\eqref{wq-bold} are better by a factor of $v t^{-\frac12}$,
therefore we can readily replace $(\ww, \qq )$ by their leading part, modulo errors
which satisfy bounds which are $v t^{-\frac12}$ better than \eqref{gamma}-\eqref{pacerr}. In view of these considerations, it suffices to
prove Proposition~\ref{p:diff} with $\gamma$ redefined as
\begin{equation}
\label{rewrite-gamma}
\gamma(t,v) =  \frac12 \int \left( -\dfrac{3i}{2v} \ttW_{\alpha}+ \ttQ\right)  \bar{\textbf{u}} 
\, d\alpha.
\end{equation}
We split $(\ttW,\ttQ)$ into  hyperbolic and elliptic components, as in the 
proof of Proposition~\ref{p:K-S}. The elliptic components which are spatially 
localized near $\alpha = vt$ have frequencies away from $v^2$, therefore, given the 
localization of $\uu$, we have
\[
  \int \left(-\dfrac{3i}{2v}  \ttW_{ell,\alpha}+\ttQ_{ell}\right)  \bar{\textbf{u}} 
\, d\alpha = O(t^{-N}).
\]
It remains to consider the contributions from the hyperbolic components
\[
(\ttW_{hyp},\ttQ_{hyp}) = \sum_{\lambda\ dyadic}(\ttW_{\lambda, hyp},\ttQ_{\lambda,hyp}) .
\]
This is a diagonal sum in that each the terms are supported in
essentially disjoint dyadic regions $\lambda \approx v^2$. Hence,
modulo $t^{-N}$ errors, it suffices to fix the dyadic frequency
$\lambda$, and restrict $\alpha$ to the region $\alpha^2/t^2 \approx
\lambda$, which corresponds to $v^2 \approx \lambda$.  We note that
the $L^2$ square summability with respect to $\lambda$ in
\eqref{NFen}, \eqref{Len} yields the corresponding summation with
respect to dyadic spatial regions in the proposition.

For $(\ttW_{\lambda, hyp},\ttQ_{\lambda,hyp})$ we have the following bounds 
which are consequences of \eqref{NFen}, \eqref{Len}:

\begin{equation}\label{app1}
\|  (v \ttW_{\lambda,hyp}, \ttQ_{\lambda,hyp})\|_{L^2} \lesssim (1+v^2)^{-5} \epsilon t^{C \epsilon^2},
\end{equation}
\begin{equation}\label{app2}
\|  (v(\ttQ_{\lambda,hyp} - \frac{2\alpha}{3t} \ttW_{\lambda,hyp}) , i \ttW_{\lambda,hyp,\alpha} -  
\frac{2\alpha}{3t} \ttQ_{\lambda,hyp})\|_{L^2} \lesssim
 \frac{(1+v^2)^\frac12}{tv^2} \epsilon t^{C \epsilon^2}.
\end{equation}
Combining the two components in the last equation and using 
again the frequency localization, we obtain also 
\begin{equation}\label{app3}
\| L (v \ttW_{\lambda,hyp}, \ttQ_{\lambda,hyp})\|_{L^2} \lesssim \frac{(1+v^2)^\frac12}{t v}  
\epsilon t^{C \epsilon^2},
 \end{equation}
where we recall that 
\[
L = \partial_\alpha - i \frac{4 \alpha^2}{9 t^2} 
\]
is the operator which cancels the phase $e^{i\phi}$.

 Now it is very easy to prove \eqref{gamma}. 
The first bound \eqref{app1} leads immediately to the first $L^2$ estimate in
\eqref{gamma}, while \eqref{app3}  yields the second $L^2$ estimate in
\eqref{gamma}.  The latter fact is easily seen by integration by
parts, after noting that $t L \uu - \partial_v \uu$ has the same 
form and size as $\uu$. Finally, the pointwise bound in \eqref{gamma} 
follows directly by Cauchy-Schwarz inequality from the two $L^2$ bounds.

It remains to prove the error bound in \eqref{pacerr}.  For that we 
first reduce the problem to either $\ttW$ or $\ttQ$. 
For that we use the two relations in \eqref{app2} to conclude that 
$\gamma$ as defined in \eqref{rewrite-gamma} is given by 
\[
\gamma =  \int  \ttQ_{\lambda,hyp} \bar \uu \, d \alpha + O(t^{-\frac14})
= \int - \frac{2i}{3v} \ttW_{\lambda,hyp} \bar \uu \, d \alpha + O(t^{-\frac14}).
\]
We first prove the $\ttQ_{\lambda,hyp}$ component of the bound \eqref{pacerr} for $s=0$.
For this it suffices  to estimate the difference 
\[
\err_0 =  e^{-i\phi(t,vt)} \ttQ_\lambda(t,vt) - t^{-\frac12} \int \ttQ_\lambda(t,\alpha) 
\bar \uu(\alpha) \, d\alpha,
\]
and show that 
\begin{equation}\label{err0}
\| \err_0\|_{L^2} \lesssim t^{\frac12 }v^{-\frac12} \| L \ttQ_\lambda(t,vt))\|_{L^2},
\qquad \| \err_0\|_{L^\infty} \lesssim t^{\frac14} v^{-\frac14} \| L \ttQ_\lambda(t,vt))\|_{L^2}.
\end{equation}
Denoting $u(t,\alpha) = e^{-i\phi(t,\alpha)} \ttQ_\lambda(t,\alpha))$ so that
$u_\alpha = e^{-i\phi(t,\alpha)} L \ttQ$, we rewrite the above difference as 
\[
\err_0
= t^{-\frac12} \int  ( u(t,vt) -  u(t,\alpha))\overline{\uu(t,\alpha) e^{-i\phi(t,\alpha)}} 
\, d \alpha.
\]
Then the bound \eqref{err0} immediately follows by writing the first difference
in terms of the derivative of $u$ (see also the more detailed computations in \cite{TI,IT}).
To deal with  with other values of $s$ one applies \eqref{err0} to $D^s\ttQ$,
and computes
\[
\int |D|^s \ttQ \ \bar \uu \, d\alpha =  \int\ttQ \  \overline{ |D|^s  \uu} \, d\alpha ,
\]
observing that the difference
\[
 |D|^s  \uu - |\xi_v|^s \uu
\]
has the same form as $\uu$ but has size better by a $t^{-\frac12}
v^\frac12$ factor.  The argument for $\ttW$ is entirely similar.
Again, we refer the reader to the more detailed computations in
\cite{IT}.
\end{proof}

\begin{proof}[Proof of Proposition~\ref{p:ode}]
 In view of the $\dH^\frac12$ energy conservation for
the linear system \eqref{ww-lin}, we directly obtain the relation
\[
\dot \gamma(t) = \int i \ttG \bar \ww_\alpha +  \ttK \bar \qq + \ttQ \bar \ggg   
 \, d\alpha  .
\]
We successively consider all terms on the right. With the exception of a single 
term, namely the resonant part of $G$, see below, all contributions will be placed 
into the error term. 
\bigskip

{\bf A. The contribution of $\bar \ggg$.}  This is 
\[
I_1 =  \int \ttQ e^{-i\phi}  \left(\partial_{\alpha}\left[ \frac{v^\frac12(\alpha-vt)}{t}  \chi+ \frac{ 9i}{8 t^\frac12}\chi '\right]+ O(t^{-\frac{3}{2}})\right) \, d\alpha.
\]
The contribution of the $O(t^{-\frac32})$ error term is estimated
directly by Cauchy-Schwarz inequality.  Also, by \eqref{pactest} we can replace $\ttQ$
in terms of $\gamma$, and the contribution of the error $\ttQ -
t^{-\frac12}e^{i\phi} \gamma(t,v)$ is directly estimated in both $L^2$
and $L^\infty$ via \eqref{pacerr}.  The contribution of $\gamma$, on
the other hand, is written using integration by parts as
\[
\tilde I_1 =  t^{-\frac12} \int  - \gamma_\alpha \left[ \frac{v^\frac12(\alpha-vt)}{t}  \chi+ \frac{ 9i}{8 t^\frac12}\chi '\right] \, d\alpha.
\]
Now we can easily bound the last expression using \eqref{gamma}.

\bigskip

{\bf B. The contributions of $\ttG,\ttK$.} The quartic terms admit the better $L^2$ 
bound in \eqref{tt-gk4}, so their contributions are estimated by Cauchy-Schwarz inequality.
It remains to consider their cubic parts, which are described in \eqref{ttgk-3}.

The analysis of the cubic terms will be unnecessarily long if one
wants to list all the terms. Instead we adopt the same strategy as in
\cite{TI} and classify them in \emph{resonant}, \emph{non-resonant} and
\emph{null} terms. For this we need the following heuristic analysis:

\begin{itemize}
\item[(i)] in the physical space waves at frequency $\xi$ move with
  velocity $\pm \frac32 |\xi|^{\frac12}$. Since our data is localized
  near the origin, it follows that the bulk of the solution at
  $(\alpha,t)$ is at space-time frequency $(\xi,\tau) =
  \left(-\dfrac{4\alpha ^2}{9t^2}, \dfrac{2\alpha}{3t}\right)$. Thus the
  worst cubic interactions are those of waves with equal frequency.
\item[(ii)] In the frequency space, trilinear interactions of equal frequency waves 
$(\xi,\pm \sqrt{ |\xi|})$ (with $ \xi < 0$) can only lead back to the characteristic set if 
exactly one complex conjugation is present.  
\end{itemize}

This leads to the following classification of the cubic terms:
\begin{itemize}
\item[A.] Non-resonant trilinear terms: these are either 
(A1) terms with no complex conjugates, or  
(A2) terms with two complex conjugates.

\item[B.] Resonant trilinear terms: terms with exactly one
  conjugation. As in \cite{TI} for such terms one may further define a notion of
  principal symbol, which is the leading coefficient in the expression
  obtained by substituting the factors in the trilinear form by the
  expressions in \eqref{pactest} \footnote{Which corresponds to all
    three frequencies being equal.}. Thus one can isolate a linear subspace
  of resonant terms for which this symbol vanishes, which we call {\em
    null terms}. Hence on the full class of resonant trilinear terms
  we can further define an equivalence relation, modulo null terms.
\end{itemize}

Based on this, we seek to classify the terms in $(\ttG^{3}(W,Q),
\ttK^{3}(W,Q))$.  To avoid cumbersome notations, we first classify the
terms in $(\tilde{G}^{3}(W,Q), \tilde K^{3}(W,Q))$. As we will see, 
the additional terms in the difference of the two do not contribute at all to the 
asymptotic equation, but only to the error term.

We begin our discussion with the quadratic terms in $G^{(2)}$ and
$K^{(2)}$.  $G^{(2)}$ contains only mixed terms, while $K^{(2)}$
contains both holomorphic and mixed terms. We decompose the latter
into a holomorphic and a mixed part,
\[
 K^{(2)} = K^{(2),h}  + K^{(2),a}  .
\] 
The holomorphic part is 
\[
K^{(2),h} =  -  Q_\alpha^2  
+\frac{3i}2   W_{\alpha \alpha} W_\alpha, 
\] 
while the mixed parts are given by 
\begin{equation*}
\left\{
\begin{aligned}
&G^{(2)} = \   P\left[ Q_\alpha \bar W_\alpha - \bar Q_\alpha W_\alpha \right] \\
&K^{(2),a}   =  \  P\left[ - |Q_\alpha|^2 
+\frac{i}2  ( W_{\alpha \alpha} \bar W_\alpha  
- \bar W_{\alpha \alpha}{W}_{\alpha})\right] .
\end{aligned}
\right.
\end{equation*}
Conveniently, both of these terms are null terms, i.e., they vanish
when we substitute all factors with their $\gamma$ approximations in
\eqref{pactest}. Thus, all their contributions to $\tG^{(3)}$ and
$\tK^{(3)}$ are also classified as null.  The expression $G^{(2)} $
also appears in $(G^{(3)},K^{(3)})$, so those contributions are also
null. The remaining terms in  $(G^{(3)},K^{(3)})$ are easy to classify.

It remains to consider the output of $K^{(2),h}$ in the bilinear expressions
in \eqref{gkt}. Most of those are resonant, so in the final analysis 
we are left with $C^a( K^{(2),h}, \bar Q)$ and $D^a( K^{(2),h}, \bar W)$.
We consider the first of these expressions. Here the resonant 
interaction occurs when all three frequencies are equal, i.e., when 
the frequency of $K^{(2),h}$ is double the frequency of $Q$.
Thus, out of the symbol $C^{a}(\xi,\eta)$ we can extract a 
leading, resonant term which has symbol $C^a(\xi,\xi) = -\dfrac{3i}8$.
Similarly, out of the symbol $D^{a}(\xi,\eta)$ we can extract a 
leading, resonant term which has symbol $C^a(\xi,\xi) = -\dfrac{7i\xi}8$.

Based on this  discussion, we decompose
\begin{equation*}
\left\{
\begin{aligned}
\tilde G^{(3)} =&\tG^{(3)}_{r}+\tG^{(3)}_{nr}+\tG^{(3)}_{null}\\
\tilde K^{(3)} = & \tK^{(3)}_{r}+\tK^{(3)}_{nr}+\tK^{(3)}_{null}, 
\end{aligned}
\right.
\end{equation*}
where 
\begin{equation*}
\left\{
\begin{aligned}
\tG^{(3)}_{r}=&\ -\dfrac{3i}{8}K^{(2),h}\bar{Q},\\
\tG^{(3)}_{nr}=&\ 2C^h (K^{(2),h}, Q)+C^{a}(Q, \bar{K}^{(2), h})+B^a(G^{(2)}, \bar{W})
-Q_{\alpha}\bar{W_\alpha}^2-\bar{Q}_{\alpha}\vert W_\alpha\vert ^2,\\
\tG^{(3)}_{null}=& \ 2 C^h (K^{(2),a}, Q)+ C^{a}(Q, \bar{K}^{(2), a}) +2B^h(G^{(2)}, W)
 +B^a(W, \bar{G}^{(2)})   \\
&+\left(C^a(K^{(2),h}, \bar{Q})+\frac{3i}{8}K^{(2),h}\bar{Q}\right)  - W_\alpha (Q_\alpha \bar W_\alpha - \bar Q_\alpha W_\alpha)   ,\\
\tilde K^{(3)} _{r}= &\ -\frac{7}{8}\bar{W}_{\alpha}K^{(2), h} +  W_{\alpha}\vert Q_{\alpha}\vert ^2
-i\left( \frac{3}{4}W_{\alpha\alpha}\vert W_{\alpha}\vert^2-\frac{3}{8}\bar{W}_{\alpha\alpha}W_{\alpha}^2\right)  ,\\
\tilde K^{(3)} _{nr}=&\ A^h(W,K^{(2),h}) + A^a(W,\bar K^{(2),h}) +Q_{\alpha}^2W_{\alpha}+  \bar{W}_{\alpha}\vert Q_{\alpha}\vert ^2 
\\ & -i \left(\frac{15}{8} W_{\alpha}^2W_{\alpha \alpha} + \frac{3}{8}W_{\alpha\alpha}\bar{W}_{\alpha}^2-\frac{3}{4}\bar{W}_{\alpha \alpha}\vert W_{\alpha}\vert^2\right)   ,\\
\tilde K^{(3)} _{null}=&\ A^h(G^{(2)}, Q) + A^h(W,K^{(2),a}) 
+A^a(G^{(2)}, \bar{Q}) + A^a(W,K^{(2),a})+D^a (Q, \bar{G}^{(2)})
\\ & + D^a(K^{(2),a},\bar W)+ 
 \left( \frac{7}{8}\bar{W}_{\alpha}K^{(2),a}+ D^a(K^{(2),h}, \bar W)\right) -Q_{\alpha}(Q_\alpha \bar W_\alpha - \bar Q_\alpha W_\alpha).
\end{aligned}
\right.
\end{equation*}
Here  the leading projection in all terms
has been harmlessly discarded, since it can be moved onto the wave
packets, which decay rapidly at positive frequencies,
\[
\| (\ww,\qq) - P (\ww,\qq) \|_{H^N} \lesssim t^{-N}.
\]

We will place all cubic contributions into  the error term $err_{\gamma}$  
except for  the contribution of the resonant part  $\tG_{r}$ and $\tilde{K}_{r}$.  
We note that for the most part the exact form of the expressions above
is irrelevant.  The only significant matter is whether the
coefficients of the terms in $\tilde{G}^{(3)}_{r}$ and
$\tilde{K}^{(3)}_{r}$ are either real or purely imaginary, depending
on the term's parity.

One obstacle which we would encounter, had we tried to work directly
with $(\tG^{(3)},\tK^{(3)})$, is the presence of terms with undifferentiated 
$W$ factors; this is because our pointwise decay bounds for $W$ are 
worse than $t^{-\frac12}$.  In the proof of Proposition~\ref{p:low-en}, 
these terms are exactly $(\tG_0,\tK_0)$ which were
computed using the expansion \eqref{nft-lead}. However, the bad part 
of these terms is eliminated via the cubic normal form correction  
$(W_{[3],0}^{lhh\to h}, (W_{[3],0}^{lhh\to h})$. This correction allows us to replace
the terms $(\tG_0,\tK_0)$ with the better part $(\tG_0^{good},\tK_0^{good})$,
where we are allowed to move at least half a derivative onto the undifferentiated 
$W$ factor. The price to pay for this correction is the additional cubic term 
in the last terms in \eqref{ttgk-3}; this term still has components with undifferentiated
$W$ factors, but shares the above property that we can freely move 
at least half a derivative onto this  factor. 

The remaining cubic correction in $\ttK$, namely $\tK_2-\tK_2^{good}$,
contains only differentiated factors $W_\alpha$ and $Q_\alpha$, and is restricted
to the non-resonant case, where the three frequencies cannot be all equal.

The conclusion of the above discussion is that 
\begin{itemize}
\item[a)] The resonant part of $(\ttG^{(3)},\ttK^{(3)})$ is the same as 
the resonant part of $(\tG^{(3)},\tK^{(3)})$.

\item[b)]  All the terms in $\ttG^{(3)}$,  respectively  $\ttK^{(3)}$ can all be expressed 
as trilinear expression
\[
M_2(Q,Q,Q), \qquad M_2 (D^\frac12 W, D^\frac12 W, Q),
\]
respectively
\[
M_{5/2} (Q,Q,D^\frac12 W), \qquad M_{5/2} (D^\frac12 W, D^\frac12 W,
D^\frac12 W).
\]
\end{itemize}

We remark that in view of the better difference bounds \eqref{NFen-diff}, 
and \eqref{point-diff}, we can easily replace $(W,Q)$ in $(\ttG^{(3)},\ttK^{(3)})$
by the normal form variables $(\ttW,\ttQ)$, modulo errors which can be included
in $err_\gamma$. 

We separately consider the three cases below, namely null, non-resonant
and resonant terms.
\bigskip

{\bf 1. Null terms.} These are trilinear forms as above, with two
distinguishing features, namely that (i) they are of type $(2,1)$ and
(ii) their symbol vanishes on the resonant set, i.e., when all the
frequencies are equal. Thus we need the following:

\begin{lemma}
  Let $M$ be a trilinear operator of type $(2,1)$ as above, with
  symbol $M(\xi_1,\xi_2,\xi_3)$ and satisfying the null condition
\[
M(\xi,\xi,\xi) = 0.
\]
Then we have 
\begin{equation}\label{null}
\left |\int M(\ttQ,\ttQ,\ttQ) \bar \uu \, d\alpha\right | \lesssim \epsilon^3 t^{-\frac98},
\end{equation}
and its analogue when some or all of the arguments are $D^\frac12 W$.
\end{lemma}

\begin{proof}
  We separate each of the three factors into the hyperbolic part and
  the elliptic part, as in the proof of Proposition~\ref{p:K-S}. We consider
several cases.

\medskip

 {\bf a) Elliptic terms.} If  at least one elliptic term is present, then that term has better
  decay, so we can neglect the null condition and use directly
  Proposition~\ref{tri-est}, estimating
\[
\|M(\ttQ_{ell},\ttQ,\ttQ) \|_{L^p} \lesssim t^{-\frac98}, \qquad 2 \leq p < \infty .
\]
Then we bound the integral using H\"older's inequality for $p$ close to infinity.

It remains to consider $L(\ttQ_{hyp},\ttQ_{hyp},\ttQ_{hyp})$.  Here the dyadic frequency pieces
are matched to essentially disjoint spatial regions. We can restrict ourselves to 
frequencies $t^{-\frac16} \leq \lambda \leq t^\frac16$, as for the remaining frequencies 
we have a better bound (see Proposition~\ref{p:K-S}), and we can group them together 
with the elliptic part.

\medskip

{\bf b) Unbalanced frequencies.}
In this case we have spatial separation between the arguments of $M$.
Separating variables in the symbol of $M$, for fixed dyadic frequencies $\lambda_1$,
$\lambda_2$, $\lambda_3$ in the above range it suffices to look at $M$ 
with symbol of the form
\[
M(\xi) = l_1(\xi_1) l_2(\xi_2) l_3(\xi_3) ,
\]
where each of $l_1$, $l_2$, $l_3$ are smooth on the dyadic scale.
Then their kernels are rapidly decreasing on the $\lambda_j^{-1}$ scale.
But $\lambda_j^{-1} < t^{-\frac16}$, while the supports of $\ttQ_{\lambda_j,hyp}$ are 
at least $t^{\frac56}$ separated. Thus for distinct $\lambda_1$,
$\lambda_2$, $\lambda_3$ we obtain rapid decay,
\[
\| M(\ttQ_{\lambda_1,hyp},\ttQ_{\lambda_2,hyp},\ttQ_{\lambda_3,hyp})\|_{L^2 \cap L^\infty}
\lesssim \epsilon^3 t^{-N}.
\]
 
 \medskip

 {\bf c) Balanced mismatched frequencies.}  Here we consider the case
 $\lambda_1 \approx \lambda_2 \approx \lambda_3 \not\approx v^2$. Then
 the arguments of $M$ are spatially localized away  from the support of $\uu$, and
 arguing as in the previous case we see that the integral \eqref{null} decays rapidly.

 \medskip

 {\bf d) Balanced matched frequencies.}  
 Here we consider the remaining case
 $\lambda_1 \approx \lambda_2 \approx \lambda_3 \not\approx v^2 = \lambda$. It is only in this case
that we need to use the null condition \eqref{null}. Since all frequencies are 
comparable, the symbol of $M$ is smooth on the dyadic scale. Due to the null condition,
we can decompose $M$ as a sum of symbols of the form
\[
M = \sum_{(i,j)} (\xi_i - \xi_j ) M_{ij}(\xi),
\]
with $M_{ij}$ smooth. Separating variables we can factor out $M_{ij}$,
and be left with, say, $M = \xi_1 - \xi_2$. Then we can write $M$ in the form
\[
M(\ttQ^1_{\lambda,hyp},  \ttQ^2_{\lambda,hyp},  \ttQ^3_{\lambda,hyp})
= M\ttQ^1_{\lambda,hyp}  \ttQ^2_{\lambda,hyp}  \bar \ttQ^3_{\lambda,hyp} - \ttQ^1_{\lambda,hyp}  L \ttQ^2_{\lambda,hyp}  \bar \ttQ^3_{\lambda,hyp}.
\]
Here $\ttQ^i_{\lambda,hyp}$ are obtained by applying zero order
multipliers to $\ttQ$, so satisfy the same bounds as $\ttQ$, as described in
Section~\ref{s:ks}.  Now we use the better $L^2$ bound for
$M\ttQ^1_{\lambda,hyp}$ in order to estimate it in $L^2$. A similar
computation applies if one of the indices $(i,j)$ corresponds to the
conjugated variable.
\end{proof}

\bigskip

{\bf 2. Non-resonant terms.}
 These are trilinear forms as above, with 
the property that  they are either of type $(3,0)$ or of type $(1,2)$.
For these we have:

\begin{lemma}
  Let $M$ be a trilinear operator of type $(3,0)$ or of type $(1,2)$ as above.
Then we have 
\begin{equation}\label{non-res}
\left |\int M(\ttQ,\ttQ,\ttQ) \bar \uu \, d\alpha\right | \lesssim \epsilon^3 t^{-\frac98},
\end{equation}
and its analogue when some or all of the arguments are $D^\frac12 \ttW$.
\end{lemma}

\begin{proof}
The elliptic case, the unbalanced case and the balanced mismatched case 
are similar to the previous Lemma, so we are left with the balanced matched case.
Suppose we are in the $(3,0)$ case; the $(1,2)$ case is similar.
Separating variables we reduce the problem to the case when $L = 1$,
and our integral is 
\[
\int \ttQ^1_{\lambda,hyp}  \ttQ^2_{\lambda,hyp}  \ttQ^3_{\lambda,hyp} \bar \uu \, d\alpha.
\]
Denoting by $\gamma^j$ the variables associated to $\ttQ^i_{\lambda,hyp}$ as in 
Proposition~\ref{p:diff}, we can replace $\ttQ^i_{\lambda,hyp}$ by $t^{-\frac12}
\gamma^j e^{i\phi}$ with 
a better error. Thus, we can replace our integral by 
\[
t^{-\frac32} 
\int \gamma^1  \gamma^2  \gamma^3  \chi\left(\frac{\alpha-vt}{t^\frac12v^{-\frac12}}\right) e^{2i\phi}  \, d\alpha + err_\gamma .
\]
Here the important fact is that the phases did not cancel. Thus using nonstationary
we can put one derivative on either $\gamma^j$ or on $\chi$, each of which gains at least 
a $t^{-\frac12}$ factor. Then it suffices to conclude by H\"older's inequality.
\end{proof}

\bigskip

{\bf 3. Resonant terms.}
As we have organized our cubic nonlinearity, our non-resonant terms contain purely 
multiplicative terms. Using the approximation in  Proposition~\ref{p:diff}, we obtain the 
leading contribution in $\dot \gamma$.  To compute that, we recall that
\[
\ttQ \approx t^{-\frac12} e^{i\phi} \gamma, \qquad \ttW_\alpha \approx   
- \frac{2i \alpha}{3t}
t^{-\frac12} e^{i\phi} \gamma.
\]
Recalling that
\[
K^{(2),h} = -  \ttQ_\alpha^2  +\frac{3i}2  \ttW_{\alpha \alpha} \ttW_\alpha,   
\]
the leading term in  $\tG^{(3)}_{res}$ is computed as follows:
\[
\begin{split}
\tG^{(3)}_{res}  = & \ -\dfrac{3i}{8}  \left(-  \ttQ_\alpha^2  +\frac{3i}2  \ttW_{\alpha \alpha} \ttW_\alpha\right)   
\bar{\ttQ} 
\\ \approx  &  \  \frac{i\alpha^4}{27 t^4} t^{-\frac32} \gamma(t,v) |\gamma(t,v)|^2 e^{i\phi}.
\end{split}
\]
A similar computation for $\tK^{(3)}_{res}$ gives
\[
 \begin{split}
\tK^{(3)}_{res}  = & \ -\frac{7}{8}\bar{\ttW}_{\alpha}\left( -  \ttQ_\alpha^2  +\frac{3i}2  \ttW_{\alpha \alpha} \ttW_\alpha,  \right)  +  \ttW_{\alpha}\vert \ttQ_{\alpha}\vert ^2
-i\left( \frac{3}{4}\ttW_{\alpha\alpha}\vert \ttW_{\alpha}\vert^2-\frac{3}{8}\bar{\ttW}_{\alpha\alpha}\ttW_{\alpha}^2\right)
\\ \approx  &  \  -\frac{2\alpha^5i}{27 t^5}  t^{-\frac32} \gamma(t,v) |\gamma(t,v)|^2 e^{i\phi}.
\end{split}
\]
Then the leading contribution in $\dot \gamma$ is 
\[
\begin{split}
I =  &  \ \int i \tG^{(3)}_{res} \bar \ww_\alpha  + \tK^{(3)}_{res} \bar \qq \, d\alpha 
=   \ \int \left( \tK^{(3)}_{res}+   \frac{2v}{3}\tG^{(3)}_{res}\right)  \bar \uu   \, d\alpha + err_\gamma
\\ = & \  \int  -\frac{8i v^5 }{81} v^{-\frac12} t^{-\frac32} \chi\left(\frac{\alpha-vt}{t^\frac12v^{-\frac12}}\right)  \gamma(t,\alpha) |\gamma(t,\alpha)|^2\, d\alpha + err_\gamma .
\end{split}
\]
Since $\chi$ has integral $1$ and, by \eqref{gamma}, we have good
control in $L^2$ over the derivative of $\gamma$, the above integral
is equal to
\[
I =  -  \frac{8 i v^5}{81 t }\gamma(t,\alpha) |\gamma(t,\alpha)|^2 + Err.
\]
Therefore the conclusion of the proposition follows. 
\end{proof}

\bibliography{ww}{}

\begin{thebibliography}{10}

\bibitem{MR2805065}
T.~Alazard, N.~Burq, and C.~Zuily.
\newblock On the water-wave equations with surface tension.
\newblock {\em Duke Math. J.}, 158(3):413--499, 2011.

\bibitem{2014arXiv1404.4276A}
T.~{Alazard}, N.~{Burq}, and C.~{Zuily}.
\newblock {Strichartz estimates and the Cauchy problem for the gravity water
  waves equations}.
\newblock {\em ArXiv e-prints}, April 2014.

\bibitem{MR2931520}
Thomas Alazard, Nicolas Burq, and Claude Zuily.
\newblock Strichartz estimates for water waves.
\newblock {\em Ann. Sci. \'Ec. Norm. Sup\'er. (4)}, 44(5):855--903, 2011.

\bibitem{MR3429478}
Thomas Alazard and Jean-Marc Delort.
\newblock Global solutions and asymptotic behavior for two dimensional gravity
  water waves.
\newblock {\em Ann. Sci. \'Ec. Norm. Sup\'er. (4)}, 48(5):1149--1238, 2015.

\bibitem{MR3460636}
Thomas Alazard and Jean-Marc Delort.
\newblock Sobolev estimates for two dimensional gravity water waves.
\newblock {\em Ast\'erisque}, (374):viii+241, 2015.

\bibitem{MR2162781}
David~M. Ambrose and Nader Masmoudi.
\newblock The zero surface tension limit of two-dimensional water waves.
\newblock {\em Comm. Pure Appl. Math.}, 58(10):1287--1315, 2005.

\bibitem{MR963906}
Charles~J. Amick and Klaus Kirchg{\"a}ssner.
\newblock A theory of solitary water-waves in the presence of surface tension.
\newblock {\em Arch. Rational Mech. Anal.}, 105(1):1--49, 1989.

\bibitem{MR1637554}
Klaus Beyer and Matthias G{\"u}nther.
\newblock On the {C}auchy problem for a capillary drop. {I}. {I}rrotational
  motion.
\newblock {\em Math. Methods Appl. Sci.}, 21(12):1149--1183, 1998.

\bibitem{MR2997362}
B.~Buffoni, M.~D. Groves, S.~M. Sun, and E.~Wahl{\'e}n.
\newblock Existence and conditional energetic stability of three-dimensional
  fully localised solitary gravity-capillary water waves.
\newblock {\em J. Differential Equations}, 254(3):1006--1096, 2013.

\bibitem{MR2763354}
Hans Christianson, Vera~Mikyoung Hur, and Gigliola Staffilani.
\newblock Strichartz estimates for the water-wave problem with surface tension.
\newblock {\em Comm. Partial Differential Equations}, 35(12):2195--2252, 2010.

\bibitem{MR1780703}
Demetrios Christodoulou and Hans Lindblad.
\newblock On the motion of the free surface of a liquid.
\newblock {\em Comm. Pure Appl. Math.}, 53(12):1536--1602, 2000.

\bibitem{MR518170}
Ronald~R. Coifman and Yves Meyer.
\newblock {\em Au del\`a des op\'erateurs pseudo-diff\'erentiels}, volume~57 of
  {\em Ast\'erisque}.
\newblock Soci\'et\'e Math\'ematique de France, Paris, 1978.
\newblock With an English summary.

\bibitem{MR2867413}
Adrian Constantin.
\newblock {\em Nonlinear water waves with applications to wave-current
  interactions and tsunamis}, volume~81 of {\em CBMS-NSF Regional Conference
  Series in Applied Mathematics}.
\newblock Society for Industrial and Applied Mathematics (SIAM), Philadelphia,
  PA, 2011.

\bibitem{MR2291920}
Daniel Coutand and Steve Shkoller.
\newblock Well-posedness of the free-surface incompressible {E}uler equations
  with or without surface tension.
\newblock {\em J. Amer. Math. Soc.}, 20(3):829--930, 2007.

\bibitem{Delort}
J.-M. Delort.
\newblock Long-time {S}obolev stability for small solutions of quasi-linear
  {K}lein-{G}ordon equations on the circle.
\newblock {\em Trans. Amer. Math. Soc.}, 361(8):4299--4365, 2009.

\bibitem{DYACHENKO199673}
A.I. Dyachenko, E.A. Kuznetsov, M.D. Spector, and V.E. Zakharov.
\newblock Analytical description of the free surface dynamics of an ideal fluid
  (canonical formalism and conformal mapping).
\newblock {\em Physics Letters A}, 221(1):73 -- 79, 1996.

\bibitem{MR2993751}
P.~Germain, N.~Masmoudi, and J.~Shatah.
\newblock Global solutions for the gravity water waves equation in dimension 3.
\newblock {\em Ann. of Math. (2)}, 175(2):691--754, 2012.

\bibitem{MR3318019}
Pierre Germain, Nader Masmoudi, and Jalal Shatah.
\newblock Global existence for capillary water waves.
\newblock {\em Comm. Pure Appl. Math.}, 68(4):625--687, 2015.

\bibitem{MR2379653}
M.~D. Groves and S.-M. Sun.
\newblock Fully localised solitary-wave solutions of the three-dimensional
  gravity-capillary water-wave problem.
\newblock {\em Arch. Ration. Mech. Anal.}, 188(1):1--91, 2008.

\bibitem{MR2847283}
M.~D. Groves and E.~Wahl{\'e}n.
\newblock On the existence and conditional energetic stability of solitary
  gravity-capillary surface waves on deep water.
\newblock {\em J. Math. Fluid Mech.}, 13(4):593--627, 2011.

\bibitem{MR1613646}
Nakao Hayashi and Pavel~I. Naumkin.
\newblock Asymptotics for large time of solutions to the nonlinear
  {S}chr\"odinger and {H}artree equations.
\newblock {\em Amer. J. Math.}, 120(2):369--389, 1998.

\bibitem{HITW}
J.~K. {Hunter}, M.~{Ifrim}, D.~{Tataru}, and T.~K. {Wong}.
\newblock {Long time Solutions for a Burgers-Hilbert Equation via a Modified
  Energy Method}.
\newblock {\em ArXiv e-prints}, January 2013.

\bibitem{HIT}
John~K. Hunter, Mihaela Ifrim, and Daniel Tataru.
\newblock Two dimensional water waves in holomorphic coordinates.
\newblock {\em Comm. Math. Phys.}, 346(2):483--552, 2016.

\bibitem{TI}
Mihaela Ifrim and Daniel Tataru.
\newblock Global bounds for the cubic nonlinear {S}chr\"odinger equation
  ({NLS}) in one space dimension.
\newblock {\em Nonlinearity}, 28(8):2661--2675, 2015.

\bibitem{IT}
Mihaela Ifrim and Daniel Tataru.
\newblock Two dimensional water waves in holomorphic coordinates {II}: global
  solutions.
\newblock {\em Bull. Soc. Math. France}, 144(2):369--394, 2016.

\bibitem{IP-model}
A.~D. {Ionescu} and F.~{Pusateri}.
\newblock {Global analysis of a model for capillary water waves in 2D}.
\newblock {\em ArXiv e-prints}, June 2014.

\bibitem{IP3}
A.~D. {Ionescu} and F.~{Pusateri}.
\newblock {Global regularity for 2d water waves with surface tension}.
\newblock {\em ArXiv e-prints}, August 2014.

\bibitem{MR3121725}
Alexandru~D. Ionescu and Fabio Pusateri.
\newblock Nonlinear fractional {S}chr\"odinger equations in one dimension.
\newblock {\em J. Funct. Anal.}, 266(1):139--176, 2014.

\bibitem{2013arXiv1303.5357I}
Alexandru~D. Ionescu and Fabio Pusateri.
\newblock Global solutions for the gravity water waves system in 2d.
\newblock {\em Invent. Math.}, 199(3):653--804, 2015.

\bibitem{MR2850346}
Jun Kato and Fabio Pusateri.
\newblock A new proof of long-range scattering for critical nonlinear
  {S}chr\"odinger equations.
\newblock {\em Differential Integral Equations}, 24(9-10):923--940, 2011.

\bibitem{MR1682725}
Carlos~E. Kenig and Elias~M. Stein.
\newblock Multilinear estimates and fractional integration.
\newblock {\em Math. Res. Lett.}, 6(1):1--15, 1999.

\bibitem{Kl}
Sergiu Klainerman.
\newblock Uniform decay estimates and the {L}orentz invariance of the classical
  wave equation.
\newblock {\em Comm. Pure Appl. Math.}, 38(3):321--332, 1985.

\bibitem{MR2138139}
David Lannes.
\newblock Well-posedness of the water-waves equations.
\newblock {\em J. Amer. Math. Soc.}, 18(3):605--654 (electronic), 2005.

\bibitem{MR1934619}
Hans Lindblad.
\newblock Well-posedness for the linearized motion of an incompressible liquid
  with free surface boundary.
\newblock {\em Comm. Pure Appl. Math.}, 56(2):153--197, 2003.

\bibitem{MR2199392}
Hans Lindblad and Avy Soffer.
\newblock Scattering and small data completeness for the critical nonlinear
  {S}chr\"odinger equation.
\newblock {\em Nonlinearity}, 19(2):345--353, 2006.

\bibitem{MR2371442}
Camil Muscalu.
\newblock Paraproducts with flag singularities. {I}. {A} case study.
\newblock {\em Rev. Mat. Iberoam.}, 23(2):705--742, 2007.

\bibitem{MR1887641}
Camil Muscalu, Terence Tao, and Christoph Thiele.
\newblock Multi-linear operators given by singular multipliers.
\newblock {\em J. Amer. Math. Soc.}, 15(2):469--496, 2002.

\bibitem{MR0609882}
V.~I. Nalimov.
\newblock The {C}auchy-{P}oisson problem.
\newblock {\em Dinamika Splo\v sn. Sredy}, (Vyp. 18 Dinamika Zidkost. so
  Svobod. Granicami):104--210, 254, 1974.

\bibitem{MR0347218}
L.~V. Ovsjannikov.
\newblock To the shallow water theory foundation.
\newblock {\em Arch. Mech. (Arch. Mech. Stos.)}, 26:407--422, 1974.
\newblock Papers presented at the Eleventh Symposium on Advanced Problems and
  Methods in Fluid Mechanics, Kamienny Potok, 1973.

\bibitem{MR2185062}
Ben Schweizer.
\newblock A priori estimates for the {E}uler equation with a free boundary.
\newblock In {\em E{QUADIFF} 2003}, pages 401--406. World Sci. Publ.,
  Hackensack, NJ, 2005.

\bibitem{MR803256}
Jalal Shatah.
\newblock Normal forms and quadratic nonlinear {K}lein-{G}ordon equations.
\newblock {\em Comm. Pure Appl. Math.}, 38(5):685--696, 1985.

\bibitem{MR2388661}
Jalal Shatah and Chongchun Zeng.
\newblock Geometry and a priori estimates for free boundary problems of the
  {E}uler equation.
\newblock {\em Comm. Pure Appl. Math.}, 61(5):698--744, 2008.

\bibitem{MR1471885}
Sijue Wu.
\newblock Well-posedness in {S}obolev spaces of the full water wave problem in
  {$2$}-{D}.
\newblock {\em Invent. Math.}, 130(1):39--72, 1997.

\bibitem{MR1641609}
Sijue Wu.
\newblock Well-posedness in {S}obolev spaces of the full water wave problem in
  3-{D}.
\newblock {\em J. Amer. Math. Soc.}, 12(2):445--495, 1999.

\bibitem{MR2507638}
Sijue Wu.
\newblock Almost global wellposedness of the 2-{D} full water wave problem.
\newblock {\em Invent. Math.}, 177(1):45--135, 2009.

\bibitem{MR660822}
Hideaki Yosihara.
\newblock Gravity waves on the free surface of an incompressible perfect fluid
  of finite depth.
\newblock {\em Publ. Res. Inst. Math. Sci.}, 18(1):49--96, 1982.

\bibitem{MR728155}
Hideaki Yosihara.
\newblock Capillary-gravity waves for an incompressible ideal fluid.
\newblock {\em J. Math. Kyoto Univ.}, 23(4):649--694, 1983.

\end{thebibliography}
\bibliographystyle{plain}



\end{document}